\documentclass[10pt]{amsart}
\pdfoutput=1
\usepackage{amsmath,amssymb,amsthm}
\usepackage{enumerate}             
\usepackage{cite}                  
\usepackage[margin=1in]{geometry}  
\usepackage{ifthen}
\usepackage[numbered]{bookmark}    
\usepackage{graphicx}
\usepackage{xstring}
\usepackage{etoolbox}
\usepackage{wrapfig,floatrow}
\usepackage{pgf,pgffor}
\usepackage{tikz}
\usetikzlibrary{patterns}
\usetikzlibrary{decorations.pathreplacing}
\usetikzlibrary{arrows}

\renewcommand{\ge}{\geqslant}
\renewcommand{\le}{\leqslant}

\newtheorem{theorem}{Theorem}[section]
\newtheorem*{theorem-nn}{Theorem}
\newtheorem*{theorem-main}{Main Theorem}
\newtheorem{corollary}[theorem]{Corollary}
\newtheorem{lemma}[theorem]{Lemma}
\newtheorem{proposition}[theorem]{Proposition}
\theoremstyle{definition}
\newtheorem{definition}[theorem]{Definition}
\theoremstyle{remark}
\newtheorem{remark}[theorem]{Remark}

\newlength{\bulletRaiseLen}
\setbox1=\hbox{$\bullet$}\setbox2=\hbox{\tiny$\bullet$}
\renewcommand{\emph}[1]{{\bf #1}}

\newcommand{\rfin}{\bigl[\mathbb{R}^{>0}\bigr]^{\mathrm{fin}}}

\newcommand{\fix}{\mathrm{Fix}}
\newcommand{\per}{\mathrm{Per}}
\newcommand{\period}{\mathrm{per}}
\newcommand{\free}{\mathrm{Free}}

\newcommand{\Lem}{\mathrm{Lem}\,}

\newcommand{\dist}{\mathrm{dist}}
\newcommand{\freqa}{\mathrm{fr}_\alpha}
\newcommand{\tileable}{\mathcal{T}}
\newcommand{\tiled}{\mathfrak{T}}

\newcommand{\dom}{\mathrm{dom}}
\newcommand{\es}{\varnothing}

\newcommand{\fgr}[1]{[\![#1]\!]}
\newcommand{\acts}{\curvearrowright}
\newcommand{\mff}{\mathfrak{F}}
\newcommand{\invm}{\mathcal{\mathcal{E}}}
\newcommand{\omsp}{\Omega_{\mathit{sp}}}

\makeatletter

\newsavebox\boxFormula
\newsavebox\boxOverline
\newlength\lenFormula
\newlength\lenOverline
\newlength\lskip

\def\closure{\@ifnextchar[{\closure@i}{\closure@i[0]}}
\def\closure@i[#1]{\@ifnextchar[{\closure@ii[#1]}{\closure@ii[#1][0]}}
\def\closure@ii[#1][#2]#3{%
  \setlength\lskip{#1pt}%
  \StrLeft{#3}{1}[\fsym]%
  \StrRight{#3}{1}[\lsym]%
  \sbox{\boxFormula}{$\m@th#3$}%
  \setbox\boxOverline\null%
  \ht\boxOverline=\ht\boxFormula%
  \dp\boxOverline=\dp\boxFormula%
  \setlength\lenOverline{\the\wd\boxFormula}%
  \advance\lenOverline by -#1pt%
  \advance\lenOverline by -#2pt%
  \ifthenelse{\equal{\fsym}{L}}{%
    \advance\lenOverline by -1.8pt%
    \advance\lskip by 1pt%
  }{}%
  \ifthenelse{\equal{\fsym}{X}}{%
    \advance\lenOverline by -2.5pt%
    \advance\lskip by 2pt%
  }{}%
  \ifthenelse{\equal{\fsym}{\eta}}{%
    \advance\lenOverline by -1pt%
    \advance\lskip by 1pt%
  }{}%
  \wd\boxOverline=\lenOverline%
  \sbox\boxOverline{$\m@th\overline{\copy\boxOverline}$}%
  \setlength\lenFormula{\the\wd\boxFormula}%
  \addtolength\lenFormula{-\the\wd\boxOverline}%
  \rlap{\hskip \lskip  \usebox\boxOverline}{\usebox\boxFormula}%
}
\newcommand{\cl}{\closure}

\def\orbit{\@ifnextchar[{\orbit@i}{\orbit@i[ ]}}
\def\orbit@i[#1]{\mathrm{Orb}_{#1}}

\def\oer{\@ifnextchar[{\oer@i}{\oer@i[]}}
\def\oer@i[#1]#2{\mathsf{E}^{#1}_{#2}}

\def\rgap{\@ifnextchar[{\rgap@i}{\rgap@i[ ]}}
\def\rgap@i[#1]{\mathrm{ga}\vec{\mathrm{p}}_{#1}}

\makeatother

\begin{document}
\title{Regular cross sections of Borel flows}

\author{Konstantin Slutsky}
\address{Department of Mathematics, Statistics, and Computer Science\\
University of Illinois at Chicago\\
322 Science and Engineering Offices (M/C 249)\\
851 S. Morgan Street\\
Chicago, IL 60607-7045}
\email{kslutsky@gmail.com}
\thanks{Research has been partially supported by Denmark's Council for Independent Research (Natural
  Sciences  Division), grant no. 10-082689/FNU}
\begin{abstract}
  Any free Borel flow is shown to admit a cross section with only two possible distances between
  adjacent points.  Non smooth flows are proved to be Lebesgue orbit equivalent if and only if they
  admit the same number of invariant ergodic probability measures.
\end{abstract}

\maketitle

\section{Introduction}
\label{sec:introduction}

This paper is a contribution to the theory of Borel flows, that is Borel actions of the real line on
a standard Borel space.  If this level of abstraction makes the reader uncomfortable, it is safe to
think about continuous actions of the real line on a Polish space, i.e., on a separable completely
metrizable topological space.  No generality is lost this way.  An important tool in studying Borel
flows is the concept of a cross section of the action.  A cross section for a flow
\( \mathbb{R} \acts X \) is a Borel set \( \mathcal{C} \subseteq X \) which intersects every orbit in
a non-empty countable set.

\begin{wrapfigure}[11]{R}{5.5cm}
  \centering
  \begin{tikzpicture}
    \draw[thick] plot [smooth, tension=1] coordinates { (0,2) (1.5,2.5) (3,2) (4.5,2.2)};
    \draw[thick] (0,0) -- (4.5,0); \filldraw (1.5,1) circle (1pt); \draw [->,>=stealth] (1.5,1) --
    (1.5,2.5); \draw [dotted] (1.5,1) -- (1.5,0); \filldraw (1.5,0) circle (1pt); \filldraw (3.5,0)
    circle (1pt); \draw [->,>=stealth] (3.5,0) -- (3.5,0.6); \draw (0.2,-0.26) node
    {\( \mathcal{C} \)}; \draw (1.5,-0.3) node {\( x \)}; \draw (3.5,-0.3) node
    {\( \phi_{\mathcal{C}}(x) \)}; \draw (4, 2.35) node {\( f \)}; \draw (0.5, 1) node
    {\( \mathcal{C} \times f \)};
  \end{tikzpicture}
  \caption{}
  \label{fig:flow-under-function-intro}
\end{wrapfigure}

A ``sufficiently nice''\footnote{Meaning that each point should have a successor and a predecessor.}
cross section \( \mathcal{C} \) of a flow can be endowed with an induced automorphism
\( \phi_{\mathcal{C}} : \mathcal{C} \to \mathcal{C} \) which sends a point in \( \mathcal{C} \) to
the next one within the same orbit.  This gives a representation of the flow
\( \mathbb{R} \acts X \) as a flow under a function depicted in Figure
\ref{fig:flow-under-function-intro} (We refer reader to Subsection \ref{sec:flow-under-function} for
a formal treatment).  It has been known since the work of V.~M.~Wagh \cite{wagh_descriptive_1988}
that any Borel flow admits a cross section.  The representation as in Figure
\ref{fig:flow-under-function-intro} has two parameters \textemdash{}\, the ``gap'' function \( f \),
and the base automorphism \( \phi_{\mathcal{C}} \).  Since there are many ways of presenting a given
flow as such under a function, one naturally wants to understand the flexibility of these
parameters.  We concentrate here on the gap function.  Given a flow \( \mathbb{R} \acts X \) we aim
at finding the simplest function \( f \) such that \( \mathbb{R} \acts X \) can be represented as a
flow under \( f \).  But before we state our main results, let us make a short detour and describe
this problem from the point of view of ergodic theory.

Borel dynamics and ergodic theory are mathematical siblings \textemdash{} they study dynamical systems
from similar perspectives and share a large portion of methods.  Yet they also have some fundamental
differences.  A slightly non-standard (but equivalent to the classical) way of setting up an ergodic
theoretical system is to consider a Borel action of a locally compact group\footnote{The relevant
  case here is \( G = \mathbb{R} \).} \( G \) on a standard Borel space \( X \) which is moreover
equipped with a probability measure \( \mu \).  We assume that the action \( G \acts X \) preserves
\( \mu \).  So one difference from Borel dynamics is that we consider only actions which admit
invariant measure.  But much more importantly we disregard sets of measure zero, in the sense that
we are satisfied if the result holds true {\it almost} everywhere with respect to \( \mu \), and may
throw away a set of points of zero measure if necessary.  The latter is a significant luxury, which
is not available in the context of Borel dynamics.  We shall elaborate on this below, but first let
us describe the relevant theorems that have been proved in ergodic theory.

\subsection{Known results}
\label{sec:known-results}

The idea of constructing cross sections and reducing analysis of the flow to the analysis of the
induced automorphism goes back to H.~Poincar\'e, and has been very useful in many areas of dynamical
systems.  In the generality of measure preserving flows on standard Lebesgue spaces, existence of
cross sections was proved by W.~Ambrose and S.~Kakutani
\cite{ambrose_representation_1941, ambrose_structure_1942}.

\begin{theorem-nn}[Ambrose--Kakutani]
  \label{thm:ambrose}
  Any measure preserving flow \( \mathbb{R} \acts X \) on a standard Lebesgue space admits a cross
  section on an invariant subset of full measure.
\end{theorem-nn}

Ambrose also established a criterion for a flow to admit a cross section with constant gaps, i.e., to
admit a representation under a constant function.  Flows admitting such cross sections turn out to
be very special, and a typical flow is not a flow under a constant function.  It came as a surprise
that, as proved by D.~Rudolph \cite{rudolph_two-valued_1976}, any flow can be represented, at least
as far as an ergodic theorist is concerned, as a flow under a two-valued function.

\begin{theorem-nn}[Rudolph]
  \label{thm:rudolph-two-values}
  Let \( \alpha \) and \( \beta \) be positive rationally independent reals.  When restricted to an
  invariant subset of full measure, any measure preserving flow \( \mathbb{R} \acts X \) on a
  standard Lebesgue space admits a cross section with \( \alpha \) and \( \beta \) being the only
  possible distances between adjacent points; i.e., after throwing away a set of zero measure, any
  flow can be represented as flow under a two-valued function.
\end{theorem-nn}

\begin{wrapfigure}{L}{4.5cm}
  \centering
  \begin{tikzpicture}
    \draw[thick] (0,1.5) -- (1.5,1.5);
    \draw[thick] (1.5,1) -- (3.5,1);
    \draw (0.75,1.7) node {\( \alpha \)};
    \draw (2.5,1.2) node {\( \beta \)};
    \draw[thick] (0,0) -- (3.5,0);
    \draw (1.5,0.1) -- (1.5,-0.1);
    \draw (0.75,-0.26) node {\( \mathcal{C}_{\alpha} \)};
    \draw (2.5,-0.26) node {\( \mathcal{C}_{\beta} \)};
  \end{tikzpicture}
  \vspace{-0.3mm}
  \caption{}
  \label{fig:flow-under-two-valued-function}
\end{wrapfigure}
So given \( \alpha \) and \( \beta \), for any measure preserving flow \( \mathbb{R} \acts X \), one
may find a Borel invariant subset \( Y \subseteq X \) and a cross section
\( \mathcal{C} \subseteq Y \) for the restriction \( \mathbb{R} \acts Y \) such that
\( \dist\bigl(x, \phi_{\mathcal{C}}(x)\bigr) \in \{\alpha, \beta\} \) for all
\( x \in \mathcal{C} \) (see Figure \ref{fig:flow-under-two-valued-function}).  Let
\[\hspace{4.5cm}
\mathcal{C}_{\alpha} = \bigl\{ x \in \mathcal{C} : \dist\bigl(x, \phi_{\mathcal{C}}(x)\bigr) =
\alpha\bigr\}. \]
As proved by Ambrose \cite{ambrose_representation_1941}, the measure \( \mu \) on \( X \) can be
disintegrated into \( \nu \times \lambda \), where \( \nu \) is a \( \phi_{\mathcal{C}} \)-invariant
finite measure on \( \mathcal{C} \) and \( \lambda \) is the restriction of the Lebesgue measure.
In fact, there is a one-to-one correspondence between finite measures on \( X \) invariant with
respect to the flow \( \mathbb{R} \acts X \) and finite measures on \( \mathcal{C} \) invariant
under the induced automorphism \( \phi_{\mathcal{C}} \).  U.~Krengel \cite{krengel_rudolphs_1976}
strengthened Rudolph's result by showing that for any \( \rho \in (0, 1) \), one may always find a
cross section \( \mathcal{C} \) as in Rudolph's Theorem such that moreover
\( \nu(\mathcal{C}_{\alpha}) = \rho \nu(\mathcal{C}) \); i.e., the proportion of \( \alpha \)-points
is exactly \( \rho \).

This summarizes the relevant ergodic theoretical results.  All the arguments
employed in the aforementioned works require throwing certain sets (of zero measure) away.  The
first result in the purely descriptive set theoretical context was obtained by Wagh
\cite{wagh_descriptive_1988}, where he proved the analog of Ambrose's result for Borel flows.

\begin{theorem-nn}[Wagh]
  \label{thm:wagh}
  Any Borel flow \( \mathbb{R} \acts X \) on a standard Borel space admits a Borel cross section.
\end{theorem-nn}

\subsection{Main Theorem}
\label{sec:main-theorem}

The question of finding the simplest possible gap function for a flow, and in particular whether the
analogs of Rudolph's Theorem and its refinement due to Krengel hold true in the Borel context
remained open, and was explicitly posed by M.~G.~Nadkarni \cite[Remark 2 after Theorem
12.37]{nadkarni_basic_1998}.  Borel theoretic version of Rudolph's original method was worked out by
Nadkarni and Wagh in \cite{nadkarni_wagh_????}.  They have constructed cross sections with only two
possible gaps for flows which satisfy a certain technical condition, which, in particular, implies
that the flow must be sparse (a flow is sparse if it admits a cross section with arbitrarily
large gaps within each orbit, see Section \ref{sec:sparse-flows} below).  Though presented in a
different way, their method is essentially equivalent to the argument outlined at the beginning of
Subsection \ref{sec:sparse-case}.

The property of being sparse is a significant restriction (see, for example, Proposition
\ref{prop:minimal-homeo-are-not-sparse}), and the main result of this work is the affirmative answer
to Nadkarni's question in full generality: Every Borel flow admits a cross section with only two
gaps between adjacent points.  Moreover, we also provide a Borel strengthening of Krengel's
result.    

\begin{theorem-main}[see Theorem \ref{thm:tiling-general-flows}]
  \label{thm:main-theorm}
  Let \( \alpha \) and \( \beta \) be positive rationally independent reals and let \( \rho \in
  (0,1) \).  Any free Borel flow \( \mathbb{R} \acts X \) on a standard Borel space \( X \) admits a
  cross section \( \mathcal{C} \subseteq X \) such that 
  \[ \dist\bigl(x, \phi_{\mathcal{C}}(x)\bigr) \in \{\alpha, \beta\} \quad \textrm{for all } x \in
  \mathcal{C},\]
  and moreover for any \( \eta > 0 \) there is \( N \in \mathbb{N} \) such that for all \( n \ge N
  \) and all \( x \in \mathcal{C} \) one has
  \[ \Bigl| \rho - \frac{1}{n}\sum_{i=0}^{n-1}
  \chi_{\mathcal{C}_{\alpha}}\bigl(\phi^{i}_{\mathcal{C}}(x)\bigr) \Bigr| < \eta, \]
  where \( \chi_{\mathcal{C}_{\alpha}} \) is the characteristic function of
  \( \mathcal{C}_{\alpha} \).
\end{theorem-main}

As an application of the main theorem, we derive a classification of Borel flows up to Lebesgue
orbit equivalence.
\begin{theorem-nn}[see Theorem \ref{thm:djk-for-flows}]
  Non-smooth free Borel flows are Lebesgue orbit equivalent if and only if they have the same number
  of invariant ergodic probability measures.
\end{theorem-nn}

\subsection{Borel Dynamics versus Ergodic Theory}
\label{sec:borel-dynamics-vs}

A reader with background in ergodic theory and little experience in Borel dynamics may be puzzled by
the following question.  How big is the difference between proving a certain statement almost
everywhere and achieving the same result on literally every orbit?  We would like to take an
opportunity and address this question now.

The short answer is that the difference between an everywhere and an
almost everywhere argument is often huge, with the former one posing more difficulties.  Of course,
it is perfectly possible for a certain property to hold almost everywhere but to fail on some
orbits.  For instance, as shown in Section \ref{sec:sparse-flows}, every flow admits a sparse cross
section on an invariant subset of (uniformly) full measure, i.e., every flow is sparse as far as an
ergodic theorist would care, but any cross section for a free minimal continuous flow on a compact
metrizable space has bounded gaps on a comeager set, Proposition
\ref{prop:minimal-homeo-are-not-sparse}.  So from the topological notion of largeness such a flow
is the opposite of being sparse.

When an ergodic theoretical result happens to be true everywhere (as opposed to just almost
everywhere) it usually does so for a non-trivial reason and the proof frequently relies upon a
different set of ideas.  Let us give some examples.  One of the high points in ergodic theory is a
theorem of D.~S.~Ornstein and B.~Weiss \cite{ornstein_entropy_1987} which shows that any action of
an amenable group is necessarily hyperfinite.  In the Borel context current state of the art is the
work of S.~Schneider and B.~Seward \cite{schneider_locally_2013} based on the methods developed by
S.~Gao and S.~Jackson \cite{gao_countable_2015}.  Schneider and Seward proved that all Borel actions
of countable locally nilpotent groups are hyperfinite.  What happens beyond this class of groups is
widely open.

Arguments used in the works of Gao--Jackson and Schneider--Seward do not use in a direct way the
fact that the corresponding statements are known to be true almost everywhere, instead a general
ingenious construction witnessing hyperfiniteness everywhere is provided.  But sometimes a different
approach is more successful.  Quite often an ergodic theoretical argument that works with respect to
a given invariant measure can be run with mild additional effort uniformly over all invariant
measures simultaneously.  This makes it possible to reduce the problem to an action with no finite
invariant measures.  The latter has a ``positive'' reformulation discovered by Nadkarni
\cite{nadkarni_existence_1990} for the actions of \( \mathbb{Z} \) and generalized to arbitrary
countable equivalence relations by H.~Becker and A.~S.~Kechris \cite[Theorem
4.3.1]{becker_descriptive_1996}.  It turns out that not having a finite invariant measure is
equivalent to being compressible (see Section \ref{sec:lebesg-orbit-equiv}).  This suggests a
different strategy for a proof in the Borel world: First run a uniform version of the ergodic
theoretical argument and then complete the proof by giving a different argument for the compressible
case.  The pivotal example of the power of this approach is the classification of hyperfinite
equivalence relations by R.~Dougherty, S.~Jackson, and A.~S.~Kechris
\cite{dougherty_structure_1994}.  A baby version of this idea is also used in the proof of Theorem
\ref{thm:frequency-rho-can-map} in Section \ref{sec:lebesg-orbit-equiv}.  When the proof follows
this ambivalent path, one effectively provides two reasons for the statement to be true
\textemdash{} an ergodic theoretical reason and another one based on compressibility.  The second part
exists in the Borel context only and is usually the main reason for the step up in the complexity of
proof.

The proof of our main theorem also splits in two parts, but the interaction between ergodic
theoretical and ``compressible'' cases is more intricate, and unlike the previous examples we start
with a ``compressible'' argument which is then complemented by an ``ergodic theoretical'' method.
We first run a construction which aims at a weaker goal than the one prescribed in the Main Theorem:
instead of constructing a regular cross section we construct a cross section with arbitrarily large
regular blocks within each orbit.  Despite the a priory weaker goal, the algorithm may accidentally
achieve the result of Main Theorem on certain parts of the space.  On the complementary part, on the
other hand, its failure will manifest existence of a sparse cross section which will be enough for
another construction, inspired by the ergodic theoretical technique of Rudolph
\cite{rudolph_two-valued_1976}, to succeed.  We believe that the general approach of a
sparse/co-sparse decomposition used in this paper can be of value for attacking other problems in
Borel dynamics.

We hope that the reader has been convinced by now that an everywhere case may be noticeably
different from an almost everywhere one.  Let us now offer some reasons for working in the Borel
context.  One reason comes from topological dynamics where it may be unnatural to disregard sets of
measure zero even in the presence of invariant measure.  And as shows, for instance, a simple
argument in Section \ref{sec:sparse-flows}, topological and measurable genericities may be completely
different.  Another motivation comes from the theory of Borel equivalence relations.  An important
subclass of Borel equivalence relations consists of equivalence relations coming from group actions.
Those coming from the actions of \( \mathbb{R} \) are the orbit equivalence relations of Borel
flows.  And the main theorem of this paper implies a classification of Borel flows up to Lebesgue
orbit equivalence.  This is the content of Section \ref{sec:lebesg-orbit-equiv}.  In fact, most of the
results which prove existence of regular cross sections in ergodic theory and Borel dynamics were
tools developed to solve some concrete problems.  Rudolph \cite{rudolph_two-valued_1976}
used his construction of regular cross sections to settle in the case of finite entropy a problem of
Sinai about equivalence between two definition of \( K \)-flows.  Krengel applied his strengthening
of Rudolph's construction to prove a version of Dye's Theorem for flows.  The
classification up to Lebesgue orbit equivalence given in Theorem \ref{thm:djk-for-flows} is
analogous to the Dougherty--Jackson--Kechris classification of hyperfinite equivalence relation.

\medskip

The paper is structured as follows.  Section \ref{sec:basic-concepts} covers the basics of the
theory of Borel flows.  Section \ref{sec:sparse-flows} introduces the family of sparse flows, which
form the right class of flows for which ergodic theoretical methods can be applied.  In
Section \ref{sec:overview-proof} we give an overview of the proof of the main theorem motivating
some of the further analysis.  Sections \ref{sec:tools}, \ref{sec:propagation-freedom}, and
\ref{sec:flex-part} form a technical core of the paper.  In Section \ref{sec:rudolphs-method} we
prove the main theorem under an  additional assumption that the flow is sparse, and Section
\ref{sec:tiling-general-flows} provides the complementary argument proving the general case.
Finally, Section \ref{sec:lebesg-orbit-equiv} gives an application of the
tiling result to the classification problem of Borel flows up to Lebesgue orbit equivalence.

\section{Basic concepts of Borel flows}
\label{sec:basic-concepts}
This section will serve as a foundation for our study.  We recall some well-known results and
establish notation to be used throughout the paper.

A \emph{Borel flow\/} is a Borel measurable action \( \mff : \mathbb{R} \acts \Omega \) of
the group of reals on a standard Borel space \( (\Omega, \mathcal{B}_{\Omega}) \).  For
\( r \in \mathbb{R} \) and \( \omega \in \Omega \) we use \( \omega + r \) as a shortcut for a more
formal \( \mff(r,\omega) \), provided the flow \( \mff \) is unambiguous from the
context.  The orbit of a point \( \omega \in \Omega \) is denoted by
\( \orbit[\mff](\omega) \) or just by \( \orbit(\omega) \).  With any flow
\( \mff \) we associate an \emph{orbit equivalence relation} \( \oer[\mff]{\Omega} \) (or just \(
\oer{\Omega} \) when the flow in understood)
on \( \Omega \) defined by \( \omega_{1}\, \oer{\Omega}\, \omega_{2} \) whenever
\( \orbit(\omega_{1}) = \orbit(\omega_{2}) \).  This equivalence relation is Borel as a subset of \(
\Omega \times \Omega \).  A flow is
\emph{free\/} if \( \omega + r_{1} \ne \omega + r_{2} \) for any \( r_{1} \ne r_{2} \) and all
\( \omega \in \Omega \).

For subsets \( S \subseteq \mathbb{R} \) and \( \mathcal{C} \subseteq \Omega \) expression
\( \mathcal{C} + S \) denotes the union of translates
\[ \mathcal{C} + S = \bigcup_{r \in S} \mathcal{C} + r. \]
We shall frequently use the following fact: If \( \mathcal{C} \) is Borel and its intersection with
any orbit is countable, then \( \mathcal{C} + S \) is Borel for any Borel
\( S \subseteq \mathbb{R} \).  In essence, this follows from Luzin-Novikov's Theorem (see, for
example, \cite[18.10]{kechris_classical_1995}).  Here is a detailed explanation. Recall that by
Miller's Theorem \cite[9.17]{kechris_classical_1995} stabilizers of Borel actions of Polish groups
are necessarily closed subgroups.  In the case of flows, a stabilizer of \( \omega \in \Omega \) can
therefore be either the whole real line (whenever \( \omega \) is a fixed point) or a subgroup
\( \lambda \mathbb{Z} \) for some \( \lambda \ge 0 \).  In particular, the set of fixed points
\[ \fix(\mff) = \{\, \omega \in \Omega \mid \omega + r = \omega \textrm{ for all } r \in \mathbb{R}
\,\} \] is Borel, since it is enough to quantify over the rationals:
\[ \fix(\mff) = \{\, \omega \in \Omega \mid \omega + q = \omega \textrm{ for all } q \in
\mathbb{Q}\,\}. \]
If \( \mathcal{C} \subseteq \Omega \) has countable intersection with any orbit of \( \mff \), and
if \( S \subseteq \mathbb{R}\) is Borel, then the map
\[ \mathcal{C} \times S \ni (\omega, r) \mapsto \omega + r \in \Omega \]
is countable-to-one for \( \omega \in \mathcal{C} \setminus \fix(\mff) \) and \( r \in S \).  By
Luzin-Novikov's Theorem its image is Borel, hence so is \( \mathcal{C} + S \).  Borelness of sets of
this form will be routinely used throughout the paper.

We established that the set of fixed points \( \fix(\mff) \) is Borel.  In fact, so is the set of
periodic points
\[ \per(\mff) = \{\, \omega \in \Omega \mid \omega + r = \omega \textrm{ for some } r \in
\mathbb{R}^{>0} \,\}. \]
It is immediate to see that this set is analytic.  Borelness is established using the following
fairly general trick: Pick a discrete cross section \( \mathcal{C} \), enlarge it by adding small
intervals around each point, \( \mathcal{C} + [0,\epsilon] \), and then express \( \per(\mff) \) by
quantifying over the rationals.  Here is a more formal argument.  First of all, it is enough to show
that \( \per_{>0}(\mff) = \per(\mff) \setminus \fix(\mff) \) is Borel.  We may therefore restrict
our flow to \( \Omega \setminus \fix(\mff) \) and assume that it has no fixed points.  Under this
assumption, Wagh's Theorem \cite{wagh_descriptive_1988} claims existence of a Borel set
\( \mathcal{C} \) intersecting every orbit, and such that
\( \{\, r \in \mathbb{R} \mid \omega + r \in \mathcal{C}\,\} \) is a separated\footnote{A subset
  \( S \subseteq \mathbb{R} \) is \emph{separated} if there is \( \epsilon > 0\) such that
  \( |x - y| > \epsilon \) for all distinct \( x, y \in S \).}  subset of \( \mathbb{R} \) for any
\( \omega \in \mathcal{C} \).  In particular, for any \( \omega \in \mathcal{C} \) the set
\( \{\, y \in \mathcal{C} \mid \omega\, \oer{\Omega}\, y \,\} \) is countable, and if furthermore
\( \omega \in \per_{>0}(\mff) \), then it is necessarily finite.  Orbit equivalence
\( \oer{\Omega} \) induces a Borel relation
\[ \oer{\mathcal{C}} = \oer{\Omega} \cap \mathcal{C} \times \mathcal{C} \]
on \( \mathcal{C} \), and the subset \( \mathcal{C}' \subseteq \mathcal{C} \) consisting of points
with finite \( \oer{\mathcal{C}} \)-equivalence classes is Borel.  We saw that
\( \mathcal{C} \cap \per_{>0}(\mff) \subseteq \mathcal{C}' \), but \( \mathcal{C}' \) may also
include some points from the free part \( \free(\mff) = \Omega \setminus \per(\mff) \).  Since each
\( \oer{\mathcal{C}} \) class in \( \mathcal{C}' \) is finite, it admits a \emph{Borel transversal}
\textemdash{} a Borel \( \mathcal{D} \subseteq \mathcal{C}' \) which picks a representative from each class.
Note that the \emph{saturation} 
\[ [\mathcal{D}]_{\mff} = [\mathcal{D}]_{\oer[\mff]{\Omega}} := \mathcal{D} + \mathbb{R} \quad
\textrm{is Borel}\]
Finally, Borelness of \( \per_{>0}(\mff) \) is witnessed by the following equality:
\[ \per_{>0}(\mff) = \bigl[ \bigl\{\, \omega \in \mathcal{D} \bigm| \forall \epsilon \in \mathbb{Q}
\cap (0,1)\ \exists r \in \mathbb{Q}^{>1} \textrm{ such that } \omega + r \in \mathcal{D} +
[0,\epsilon] \,\bigr\} \bigr]_{\mff}.\]

To summarize, the decomposition of any flow into a fixed part, periodic part, and free part is Borel:
\[ \Omega = \fix(\mff) \sqcup \per_{>0}(\mff) \sqcup \free(\mff), \]
It may also be convenient to know that the function
\( \period : \per_{>0}(\mff) \to \mathbb{R}^{>0} \), which assigns to a point
\( \omega \in \per_{>0}(\mff) \) its period, i.e., the unique \( \lambda > 0 \) such that
\( [0,\lambda) \ni r \mapsto \omega +r \in \orbit(\omega) \) is a bijection, is Borel.  Indeed, in
the notation above let \( \mathcal{D}' = \mathcal{D} \cap \per_{>0}(\mff) \), and let
\( s : \per_{>0}(\mff) \to \mathcal{D}' \) be a selector function: \( s(\omega) = x \) if and only
if \(\omega\, \oer{\Omega}\, x \) and \( x \in \mathcal{D}' \).  The graph of \( \period \) is
Borel, since it can be written as
\[ \bigl\{\, (\omega,\lambda) \bigm| \omega + \lambda = \omega \textrm{ and } \forall \epsilon \in
\mathbb{Q}^{>0}\ \forall \delta \in (\epsilon, \lambda) \cap \mathbb{Q}\quad s(\omega) + \delta \not
\in \mathcal{D}' + [0, \epsilon]\,\bigr\},\]
and thus \( \period : \per_{>0}(\mff) \to \mathbb{R}^{>0} \) is Borel.

\medskip
{\it From now on all the flows are assumed to be free unless stated otherwise.}
\medskip

\subsection{Cross sections}
\label{sec:cross sections-flows}

When a flow \( \mff \) is free, any orbit can be ``identified'' with a copy of the real line.
Concrete identification cannot be done in a Borel way throughout all orbits, unless the flow is
smooth.  But we may nonetheless unambiguously transfer translation invariant notions from
\( \mathbb{R} \) to each orbit.  In particular, we shall talk about:
\begin{itemize}
\item Distances between points within an orbit: \(
  \dist(\omega_{1}, \omega_{2}) = r \in \mathbb{R}^{\ge 0} \) if \( \omega_{1} + r = \omega_{2} \)
  or \( \omega_{1} - r = \omega_{2} \).
\item Lebesgue measure on an orbit: for a Borel \( A \subseteq \Omega \) we set
  \[ \lambda_{x}(A) = \lambda\bigl(\{r \in \mathbb{R} : x + r \in A\}\bigr), \]
  where \( \lambda \) is the Lebesgue measure on \( \mathbb{R} \).  Note that
  \( \lambda_{x} = \lambda_{y} \) whenever \( x\, \oer{\Omega} \, y \) and \( \lambda_{x} \) is
  supported on \( \orbit(x) \).
\item Linear order within orbits: \( \omega_{1} < \omega_{2} \) if
  \( \omega_{1} + r = \omega_{2} \) for some \( r > 0 \).
\end{itemize}

A \emph{countable cross section}, or just \emph{cross section}, for a flow \( \mff \) on
\( \Omega \) is a Borel set \( \mathcal{C} \subseteq \Omega \) which intersects every orbit of
\( \mff \) in a non-empty countable set.  A cross section \( \mathcal{C} \) is \emph{lacunary} if
the gaps in \( \mathcal{C} \) are bounded away from zero: there exists \( \epsilon > 0 \) such that
\( \dist(x, y) \ge \epsilon \) for all \( x, y \in \mathcal{C} \) with
\( x\, \oer{\mathcal{C}}\, y \).  A lacunary cross section is \emph{bi-infinite} if for any
\( \omega \in \Omega \) there are \( y_{1}, y_{2} \in \mathcal{C} \) such that
\( y_{1} < \omega < y_{2 }\).

\medskip
{\it Unless stated otherwise all cross sections are assumed to be lacunary and bi-infinite.}
\medskip

With such a cross section \( \mathcal{C} \) one can associate an
\emph{induced automorphism} \( \phi_{\mathcal{C}} : \mathcal{C} \to \mathcal{C} \) that sends
\( x \in \mathcal{C} \) to the next point from \( \mathcal{C} \) within the same orbit:
\( \phi_{\mathcal{C}}(x) = y \) whenever \( x, y \in \mathcal{C} \), \( x < y \) and for no
\( z \in \mathcal{C} \) one has \( x < z < y \).  Finally, with any cross section we also associate a
\emph{gap function} \( \rgap[\mathcal{C}] : \mathcal{C} \to \mathbb{R}^{>0} \) which measures
distance to the next point: \( \rgap[\mathcal{C}](x) = \dist\bigl(x, \phi_{\mathcal{C}}(x)\bigr)\).
The gap function \( \rgap[\mathcal{C}] \) and the induced automorphism \( \phi_{\mathcal{C}} \) are
Borel.

The concept of a cross section also makes sense for the actions of \( \mathbb{Z} \), i.e., for Borel
automorphisms.  In this case distance between any two points within an orbit is a natural number and
condition of lacunarity is therefore automatic.
\subsection{Flow under a function}
\label{sec:flow-under-function}
An important concept in the theory of flows is the notion of a \emph{flow under a function} (also
known as a \emph{suspension flow}).  Given a standard Borel space \( \mathcal{C} \), a free Borel
automorphism \( \phi : \mathcal{C} \to \mathcal{C} \) and a bounded away from zero Borel function
\( f : \mathcal{C} \to \mathbb{R}^{>0} \), \( f(x) \ge \epsilon > 0 \) for all
\( x \in \mathcal{C} \), one defines the space \( \mathcal{C} \times f \) of points ``under the
graph of \( f \):''
\[ \mathcal{C} \times f = \bigl\{\, (x,s) \in \mathcal{C} \times \mathbb{R}^{\ge 0} \bigm| 0 \le s <
f(x) \,\bigr\}, \]
and a flow on \( \mathcal{C} \times f \) by letting points flow upward until they reach the graph of
\( f \) and then jump to the next fiber as determined by \( \phi \), (see Figure
\ref{fig:flow-under-function}).  In symbols, for \( \omega = (x, s) \in \mathcal{C} \times f \) and
\( r \in \mathbb{R}^{\ge 0} \)
\[ \omega + r = \Bigl(\phi^{n}(x), s + r - \sum_{i=0}^{n-1}f\bigl(\phi^{i}(x)\bigr)\Bigr)
\]
for the unique \( n \in \mathbb{N} \) such that
\( 0 \le s + r - \sum_{i=0}^{n-1}f\bigl(\phi^{i}(x)\bigr) < f\bigl(\phi^{n}(x)\bigr) \); and for \(
r \in \mathbb{R}^{<0} \)
\[ \omega + r = \Bigl(\phi^{-n}(x), s + r + \sum_{i=1}^{n}f\bigl(\phi^{-i}(x)\bigr)\Bigr)
\]
for the unique \( n \in \mathbb{N} \) such that
\( 0 \le s + r + \sum_{i=1}^{n}f\bigl(\phi^{-i}(x)\bigr) < f\bigl(\phi^{-n}(x)\bigr) \).

\begin{wrapfigure}[10]{L}{5.5cm}
  \centering
  \begin{tikzpicture}
    \draw[thick] plot [smooth, tension=1] coordinates { (0,2) (1.5,2.5) (3,2) (4.5,2.2)};
    \draw[thick] (0,0) -- (4.5,0);
    \filldraw (1.5,1) circle (1pt);
    \draw [->,>=stealth] (1.5,1) -- (1.5,2.5);
    \draw [dotted] (1.5,1) -- (1.5,0);
    \filldraw (1.5,0) circle (1pt);
    \filldraw (3.5,0) circle (1pt);
    \draw [->,>=stealth] (3.5,0) -- (3.5,0.6);
    \draw (0.2,-0.26) node {\( \mathcal{C} \)};
    \draw (1.5,-0.3) node {\( x \)};
    \draw (3.5,-0.3) node {\( \phi(x) \)};
    \draw (4, 2.35) node {\( f \)};
    \draw (0.5, 1) node {\( \mathcal{C} \times f \)};
  \end{tikzpicture}
  \vspace*{-2mm}
  \caption{Suspension flow}
  \label{fig:flow-under-function}
\end{wrapfigure}
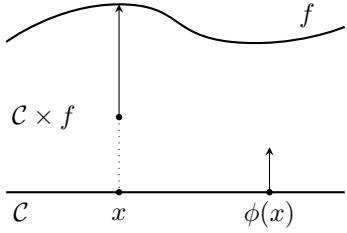

\vspace{1mm}
Note that \( \mathcal{C} \), when identified with the subset
\( \{\, (x,0) \mid x \in \mathcal{C} \,\} \) of \( \mathcal{C} \times f \), is a (lacunary
bi-infinite) cross section, the function \( f \) is the gap function \( \rgap[\mathcal{C}] \) of
this cross section and \( \phi \) coincides with the induced automorphism \( \phi_{\mathcal{C}} \).
Conversely, if \( \mff \) is a free Borel flow on \( \Omega \), and
\( \mathcal{C} \subseteq \Omega \) is a (lacunary bi-infinite) cross section, then the flow under a
function \( \rgap[\mathcal{C}] \) and the induced automorphism
\( \phi_{\mathcal{C}} : \mathcal{C} \to \mathcal{C} \) is naturally isomorphic to \( \mff \).  Our
point here is that {\it realizing a free \( \mff \) as a flow under a (bounded away from zero)
  function is the same as finding a (lacunary bi-infinite) cross section.}  In this terminology
Wagh's Theorem \cite{wagh_descriptive_1988} implies the following.

\begin{theorem}[Wagh]
  \label{thm:Wagh-under-a-function-version}
  Any free Borel flow is isomorphic to a flow
  under a (bounded away from zero) function, which one may moreover assume to be bounded
  from above.
\end{theorem}

While the notion of a cross section makes sense for actions of any Polish group, its simple
geometric interpretation as in Figure \ref{fig:flow-under-function} seems to be specific to actions
of the real line.  In the language of cross sections Theorem \ref{thm:Wagh-under-a-function-version}
is valid for all locally compact groups as showed by Kechris in \cite{kechris_countable_1992}.

One can refine the formulation of Theorem \ref{thm:Wagh-under-a-function-version} by specifying
bounds on the gap function, Corollary \ref{cor:wagh-gap-bounds}.  First of all we recall a simple
marker lemma for aperiodic Borel automorphisms.  This proposition follows easily from the
observation that for any natural \( d \ge 1 \) any sufficiently large integer is of the form
\( md + n(d+1)\) for some \( m, n \in \mathbb{N} \).  This is a particular case of
\cite[Lemma 2.9]{gao_countable_2015}.

\begin{proposition}
  \label{prop:marker-lemma-Z-actions}
  For any free Borel automorphism \( T : \mathcal{C} \to \mathcal{C} \) on a standard Borel
  space\footnote{This proposition will typically be applied to the induced automorphism associated
    with a cross section of a flow, hence the notation \( \mathcal{C} \) for the phase space.}
  \( \mathcal{C} \) and for any natural \( d \ge 1 \) there exists a bi-infinite Borel cross section
  \( \mathcal{D} \subseteq \mathcal{C} \) with gaps of size \( d \) or \( d + 1 \): for all
  \( x \in \mathcal{D} \)
  \[ \min \{ i \ge 1 : T^{i}(x) \in \mathcal{D} \} \in \{d, d+1\}. \]
\end{proposition}

\begin{proof}
  To begin with, note that for any \( d \ge 1 \) and any \( N \ge d^{2} \) there are \( m, n \in
  \mathbb{N} \) such that \( N = md + n(d+1) \).  Indeed, let \( N = qd + r \), where \( q,r \in
  \mathbb{N} \) and \( r \le N-1 \).  Since \( N \ge d^{2} \), \( q \ge d - 1 \), hence 
  \[ N = qd + r = (q - r)d + r(d+1). \]

  We may select a sub cross section \( \mathcal{C}' \subseteq \mathcal{C} \) such that \(
  \rgap[\mathcal{C}'](x) \ge d^{2} \) (see, for instance, \cite[7.25]{nadkarni_basic_1998}) for all
  \( x \in \mathcal{C}' \).  For each \( N \ge d^{2} \) fix a decomposition \( N = m_{N}d +
  n_{N}(d+1) \).  The cross section \( \mathcal{D} \) is given by
  \[ \mathcal{D} = \bigl\{\, T^{id}(x) : 0\le i \le m_{N},\ N = \rgap[\mathcal{C}'](x)\,\bigr\} \cup
  \bigl\{\, T^{m_{N}d + i(d+1)}(x) : 0\le i \le n_{N},\ N = \rgap[\mathcal{C}'](x)\,\bigr\}. \qedhere \]
\end{proof}

\begin{corollary}
  \label{cor:wagh-gap-bounds}
  Let \( k, K \in \mathbb{R}^{>0} \) be positive reals, \( k < K \).  For any Borel flow \(
  \mathbb{R} \acts \Omega \) there exists a cross section \( \mathcal{C} \subset \Omega \) such that
  \( \rgap[\mathcal{C}](x) \in [k, K] \) for all \( x \in \mathcal{C} \).
\end{corollary}

\begin{proof}
  First of all note that any sufficiently large real \( x \in \mathbb{R} \), \( x \ge N \), can be
  partitioned into pieces \( x = \sum_{i} z_{i} \) such that \( z_{i} \in [k, K] \).  Fix such an
  \( N \in \mathbb{R} \) and a Borel map\footnote{We use \( \rfin \) to denote the standard Borel
    space of finite subsets of \( \mathbb{R}^{>0} \).}
  \( \zeta : \mathbb{R}^{\ge N} \to \rfin \) such that
  \( \zeta(x) = \{z_{0}, z_{1}, z_{2}, \ldots, z_{n}\} \) where \( z_{0} = 0 \), \( z_{n} = x \) and
  \( z_{i} - z_{i-1} \in [k, K] \) for all \( 1 \le i \le n \).

  By Theorem \ref{thm:Wagh-under-a-function-version} we may pick a cross section
  \( \widetilde{\mathcal{C}} \subset \Omega \) and let \( c \in \mathbb{R}^{> 0} \) be so small that
  \( \rgap[\widetilde{\mathcal{C}}](x) \ge c \) for all \( x \in \widetilde{\mathcal{C}} \).  Set
  \( d = \lceil  N/c \rceil \) and apply Proposition \ref{prop:marker-lemma-Z-actions} to the
  induced flow
  \( \phi_{\widetilde{\mathcal{C}}} : \widetilde{\mathcal{C}} \to \widetilde{\mathcal{C}} \).  This
  results in a cross section \( \mathcal{D} \) of \( \phi_{\widetilde{\mathcal{C}}} \) which when
  viewed as a cross section of the flow \( \mathbb{R} \acts \Omega \) has gaps
  bounded from below by \( d c \ge N \).  The
  cross section
  \[ \mathcal{C} = \bigl\{x + z : x \in \mathcal{D} \textrm{ and } z \in \zeta\bigl(
  \rgap[\mathcal{D}](x) \bigr) \bigr\}\] has all its gaps in \( [k,K] \).
\end{proof}

\begin{remark}
  \label{rem:sub-cross-section-with-large-gaps}
  Similarly to the proof of Proposition \ref{prop:marker-lemma-Z-actions} one can show that given a
  cross section \( \mathcal{C} \) and a real \( N \in \mathbb{R}^{>0} \) one may always find a sub
  cross section \( \mathcal{C}' \subseteq \mathcal{C} \) with gaps of size above \( N \): \(
  \rgap[\mathcal{C}'](x) > N \) for all \( x \in \mathcal{C}' \).
\end{remark}

\subsection{Flows under constant functions}
\label{sec:flows-under-constant}
In this paper we are concerned with the following question: How simple the function in the Wagh's
Theorem \ref{thm:Wagh-under-a-function-version} can be?  The simplest case would be to have a
constant function.  This turns out not to be possible in many examples.  A criterion for a flow to
admit such a cross section was found by Ambrose \cite[Theorem 3]{ambrose_representation_1941}.
While Ambrose's original argument is carried in the measure-theoretical context, it immediately
adapts to our setting as well.  For reader's convenience we include the short proof of the
characterization.

\begin{proposition}[Ambrose]
  \label{prop:constant-gap-function}
  A Borel flow \( \mff \) on \( \Omega \) can be written as a flow under a constant function
  \[ \Omega = \mathcal{C} \times \lambda, \quad \lambda \in \mathbb{R}^{>0}, \]
  if and only if \( \mff \) has a nowhere zero eigenfunction with
  eigenvalue \( 2\pi/\lambda \), i.e., if and only if there is a Borel function
  \( h : \Omega \to \mathbb{C}\setminus\{0\} \) such that
  \[ h(\omega + r) = e^{\frac{2 \pi i r}{\lambda}} h(\omega) \]
  for all \( \omega \in \Omega \) and \( r \in \mathbb{R} \).
\end{proposition}

\begin{proof}
  \( \Rightarrow \) If \( \Omega = \mathcal{C} \times \lambda \), we may set
  \( h(\omega) = e^{\frac{2\pi i r}{\lambda}} \), where \( \omega = (x,r) \in \mathcal{C} \times
  \lambda \).

  \( \Leftarrow \) Let \( h : \Omega \to \mathbb{C} \setminus \{0\} \) be a nowhere zero
  eigenfunction with an eigenvalue \( 2\pi/\lambda \).  By considering
  \( \frac{h(\omega)}{|h(\omega)|} \) instead of \( h(\omega) \) we may assume without loss of
  generality that \( |h(\omega)| = 1 \) for all \( \omega \in \Omega \).  Let
  \( \mathcal{C} = h^{-1}(1) \).  The identity
  \[ h(\omega + r) = e^{\frac{2 \pi i r}{\lambda}} h(\omega) \]
  together with \( |h(\omega)| = 1 \) imply that \( \mathcal{C} \) is a cross section and
  \( \rgap[\mathcal{C}](x) = \lambda \) for all \( x \in \mathcal{C} \).
\end{proof}

This proposition has a more natural interpretation in the measure theoretical context.  If the flow
\( \mff \) preserves a finite measure \( \mu \) on \( \Omega \), we may associate a one-parameter
subgroup of unitary operators on \( L^{2}(\Omega, \mu) \) via the Koopman representation:
\( \bigl(U_{r}h\bigr)(\omega) = h(\omega + r) \) for \( r \in \mathbb{R} \),
\( h \in L^{2}(\Omega, \mu) \), and \( \omega \in \Omega \).  In this framework Proposition
\ref{prop:constant-gap-function} asserts equivalence between admitting a cross section with constant
gaps and having a nowhere vanishing eigenfunction for the associated one-parameter subgroup of
unitary operators.

\subsection{Regular cross sections}
\label{sec:regul-cross-sect}

Let \( \mathcal{C} \) be a cross section.  For a subset \( S \subseteq \mathbb{R}^{>0} \) we say
that \( \mathcal{C} \) is \( S \)-\emph{regular} if \( \rgap[\mathcal{C}](x) \in S \) for all \( x
\in  \mathcal{C} \).  For \( r \in \mathbb{R} \) we may say that \( x \in \mathcal{C} \) is an \( r
\)-point if \( \rgap[\mathcal{C}](x) = r \).

We shall need to consider equivalence relations on \( \mathcal{C} \) that are finer than
\( \oer{\mathcal{C}} \).  One example is the \( K \)-\emph{chain equivalence relation}
\( \oer[\le K]{\mathcal{C}} \), where \( K \in \mathbb{R}^{> 0} \), defined as follows.  Two points
\( x, y \in \mathcal{C} \) are \( \oer[\le K]{\mathcal{C}} \) related if one can get from one of the
points to the other via jumps of size at most \( K \): there exists \( n \in \mathbb{N} \) such that
\( \phi^{n}_{\mathcal{C}}(x) = y \) and
\( \rgap[\mathcal{C}]\bigl(\phi^{i}(x)\bigr) \le K \) for all \( 0 \le i < n \), or the same
condition holds with roles of \( x \) and \( y \) interchanged.  Evidently,
\( \oer[\le K]{\mathcal{C}} \) is Borel and is finer than \( \oer{\mathcal{C}} \), i.e.,
\( \oer[\le K]{\mathcal{C}} \subseteq \oer{\mathcal{C}}\).  We also note that
\( \oer[\le K]{\mathcal{C}} \subseteq \oer[\le L]{\mathcal{C}}\) whenever \( K \le L \).  When
\( \mathcal{C} \) is sparse, that is when \( \mathcal{C} \) has ``bi-infinitely'' unbounded gaps on
each orbit (see the next section for a rigorous definition), then \( \oer[\le K]{\mathcal{C}} \) is
a finite equivalence relation.

More generally, given a set \( S \subseteq \mathbb{R}^{>0} \) we let \( \oer[S]{\mathcal{C}} \) to
relate points connected by jumps of sizes in \( S \): \( x\, \oer[S]{\mathcal{C}}\, y \) if and only
if there exists \( n \in \mathbb{N} \) such that \( \phi^{n}_{\mathcal{C}}(x) = y \) and
\( \rgap[\mathcal{C}]\bigl(\phi^{i}(x)\bigr) \in S \) for all \( 0 \le i < n \) or the same
condition with roles of \( x \) and \( y \) interchanged.  In this notation \( \oer[\le
K]{\mathcal{C}} = \oer[(0,K{]}]{\mathcal{C}} \).  Note that a cross section is
\( S \)-regular if and only if \( \oer{\mathcal{C}} = \oer[S]{\mathcal{C}} \).

Of primary importance for us will be the relation \( \oer[S]{\mathcal{C}} \) when
\( S = \{\alpha, \beta\} \) is a pair of rationally independent positive reals
\( \alpha, \beta \in \mathbb{R}^{>0} \).  We shall abuse the notation slightly and denote it by
\( \oer[\alpha,\beta]{\mathcal{C}} \) omitting the curly brackets.

\subsection{Invariant measures}
\label{sec:invariant-measures}

An important invariant of a flow is its set of invariant measures.  Recall that a probability
measure \( \mu \) on the phase space \( \Omega \) is said to be \emph{ergodic} if for any Borel
invariant set \( Z \subseteq \Omega \) either \( \mu(Z) = 0 \) or \( \mu(Z) = 1 \).  Given a flow
\( \mathfrak{F} \) its set of ergodic invariant probability measures is denoted by
\( \invm(\mathfrak{F}) \).

Let \( \mathcal{C} \subset \Omega \) be a cross section of \( \mff \).  Ambrose \cite[Theorem
1]{ambrose_representation_1941} showed that for any finite \( \mff \)-invariant measure \( \mu \) on
\( \Omega \) there exists a \( \phi_{\mathcal{C}} \)-invariant measure \( \nu_{\mu} \) on
\( \mathcal{C} \) such that \( \mu \) is the product of \( \nu_{\mu} \) with the Lebesgue measure
\( \lambda \) on \( \mathbb{R} \).  More formally, when \( \Omega \) is viewed as subset of
\( \mathcal{C} \times \mathbb{R} \) via the identification with
\[ \bigl\{(x,r) \in \mathcal{C} \times \mathbb{R} : 0 \le r < \rgap[\mathcal{C}](x)\bigr\}, \]
then \( \mu = (\nu_{\mu} \times \lambda)|_{\Omega} \).

The definition of \( \nu_{\mu} \) is simple.  If \( c \in \mathbb{R}^{>0} \) is such that
\( \rgap[\mathcal{C}](x) \ge c \) for all \( x \in \mathcal{C} \), then for any Borel
\( A \subseteq \mathcal{C} \)
\[ \nu_{\mu}(A) = \frac{\mu(A \times [0,c])}{c}. \]
The definition is independent of the choice of \( c \).

The above construction of \( \nu_{\mu} \) is valid for any cross section \( \mathcal{C} \).  When
\( \mathcal{C} \) moreover admits an upper bound on its gap function, we also have a map in the
other direction.  For any \( \phi_{\mathcal{C}} \)-invariant finite measure \( \nu \) we define an
\( \mff \)-invariant \( \mu_{\nu} \) on \( \Omega \) by setting for \( A \subseteq \Omega \)
\[ \mu_{\nu}(A) = \int_{\mathcal{C}}\! \tilde{\lambda}_{x}(A)\, d\nu(x), \]
where
\[ \tilde{\lambda}_{x}(A) = \lambda\bigl(\{r \in \mathbb{R} : 0 \le r \le \rgap[\mathcal{C}](x)
\textrm{ and }x + r \in A \}\bigr). \]
Boundedness of gaps from above is needed to ensure that the integral is finite.

The maps \( \mu \mapsto \nu_{\mu} \) and \( \nu \mapsto \mu_{\nu} \) are inverses of each other and
provide a bijection between finite \( \mff \)-invariant measure on \( \Omega \) and finite \(
\phi_{\mathcal{C}} \)-invariant measures on a cross section \( \mathcal{C} \subset \Omega \) with
bounded gaps.  These maps preserve ergodicity, but do not, in general, preserve normalization: \(
\mu(\Omega) \) may not in general be equal to \( \nu_{\mu}(\mathcal{C}) \).  When normalized
manually, \( \mu \mapsto \nu_{\mu}/ \nu_{\mu}(\mathcal{C}) \) is a bijection between \( \invm(\mff)
\) and \( \invm(\phi_{\mathcal{C}}) \).  

\begin{theorem}
  \label{thm:correspondence-between-invariant-measures}
  Let \( \mff : \mathbb{R} \acts \Omega \) be a Borel flow and \( \mathcal{C} \subset \Omega \) be a
  cross section with bounded gaps.  Sets of ergodic invariant probability measures \( \invm(\mff) \)
  and \( \invm(\phi_{\mathcal{C}}) \) have the same cardinalities.
\end{theorem}

This correspondence is valid, in fact, for all unimodular locally compact groups,
\cite[Proposition 4.4]{slutsky_lebesgue_2015}.

\section{Sparse flows}
\label{sec:sparse-flows}

\begin{definition}
  \label{def:sparse-flows}
  Let \( \mff : \mathbb{R} \acts \Omega \) be a free flow.  A cross section
  \( \mathcal{C} \subseteq \Omega \) is said to be \emph{sparse} if it has gaps bi-infinitely
  unbounded on each orbit: for each \( N \in \mathbb{R} \) and \( x \in \mathcal{C} \) there are
  integers \( n_{1} \ge 0 \) and \( n_{2} < 0 \) such that
  \( \rgap[\mathcal{C}]\bigl(\phi^{n_{1}}_{\mathcal{C}}(x)\bigr) \ge N \) and
  \( \rgap[\mathcal{C}]\bigl(\phi^{n_{2}}_{\mathcal{C}}(x)\bigr) \ge N \).  We say that a flow is
  sparse if it admits a sparse cross section.  The definition of a sparse cross section for a Borel
  automorphism is analogous.
\end{definition}

The requirement of having unbounded gaps in a bi-infinite fashion is a matter of convenience only.
For if \( \mathcal{C} \) is any cross section, the set of orbits where gaps are unbounded, but not
bi-infinitely unbounded is smooth.  Indeed, for any real \( N \) the set
\[ \mathcal{D}_{N}^{l} = \{ x \in \mathcal{C} : \rgap[\mathcal{C}](x) \ge N \textrm{ and }
\rgap[\mathcal{C}]\bigl(\phi^{n}_{\mathcal{C}}(x)\bigr) < N \textrm{ for all } n \in
\mathbb{Z}^{<0}\} \]
selects the minimal point with gap of size at least \( N \) whenever such a point exists.  Similarly,
\[ \mathcal{D}_{N}^{r} = \{ x \in \mathcal{C} : \rgap[\mathcal{C}](x) \ge N \textrm{ and }
\rgap[\mathcal{C}]\bigl(\phi^{n}_{\mathcal{C}}(x)\bigr) < N \textrm{ for all } n \in
\mathbb{Z}^{>0}\} \]
selects the maximal such point.
Therefore, for any \( \mathcal{C} \) the set of orbits where \( \mathcal{C} \) has unbounded but not
bi-infinitely unbounded gaps is a saturation of the set
\[ \bigcup_{N \in \mathbb{N}} (\mathcal{D}_{N}^{l} \cup \mathcal{D}_{N}^{r}) \]
and is therefore smooth.  We shall thus always assume that our sparse cross sections have
bi-infinitely unbounded gaps and this implies that in the notation of Subsection
\ref{sec:regul-cross-sect} each \( \oer[\le K]{\mathcal{C}} \)-class is finite for any
\( K \in \mathbb{R} \).

The notion of a sparse automorphism appeared for the first time (under the name of gapped
automorphisms) in the lecture notes of B.~D.~Miller \cite{miller_dynamics_2007}.  This class of
automorphisms and flows seems to isolate a purely Borel property for which some of the ergodic
theoretical methods can be applied.

In general, a Borel flow may not admit a sparse cross section.  We start by exhibiting an important
class of flows which never admit sparse cross sections.  Recall that a continuous flow is said to be
\emph{minimal} if every orbit is dense.

\begin{proposition}
  \label{prop:minimal-homeo-are-not-sparse}
  Free continuous flows on compact metrizable spaces do not admit sparse cross sections.  If the
  flow is moreover minimal, then any cross section has bounded gaps on a Borel invariant comeager set.
\end{proposition}
\begin{proof}
  Let \( \mff \) be a minimal flow on a compact metrizable space \( \Omega \) and let
  \( \mathcal{C} \) be a Borel cross section; put \( \mathcal{I} = \mathcal{C} + (-1,1) \).  Note
  that \( \mathcal{I} \) has unbounded gaps between intervals within every orbit of \( \mff \) if
  and only if \( \mathcal{C} \) is sparse.  Since
  \[ \bigcup_{n \in \mathbb{Z}} \bigl(\mathcal{I} + n\bigr) = \Omega, \]
  it follows that \( \mathcal{I} \) is non meager, therefore it is comeager in a non-empty open
  subset \( U \subseteq \Omega \).  By minimality of the action we get
  \( U + \mathbb{Q} = \Omega \), and therefore compactness ensures existence of
  \( q_{1}, \ldots, q_{N} \in \mathbb{Q} \) such that \( \Omega = \bigcup_{n = 1}^{N}(U + q_{n}) \).
  Set \( \mathcal{J} = \bigcup_{n=1}^{N} (\mathcal{I} + q_{n}) \) and note that gaps between
  intervals in \( \mathcal{J} \) are bounded within an orbit if and only if they are bounded within
  the same orbit for \( \mathcal{I} \).  Note also that \( \mathcal{J} \) is comeager in
  \( \Omega \), hence so is \( \bigcap_{q \in \mathbb{Q}} (\mathcal{J} + q) \), which is moreover
  invariant under the flow.  Since each \( \omega \in \mathcal{J} \) is a bounded distance away from
  a point in \( \mathcal{C} \), we conclude that gaps in
  \[ \mathcal{C} \cap \bigcap_{q \in \mathbb{Q}} (\mathcal{J} + q) \]
  are bounded by \( 2 \max\bigl\{ |q_{i}| + 2 : 1 \le i \le N \bigr\} \). This proves the
  proposition for minimal flows.

  The general case follows from the minimal one, since any flow has a minimal subflow: by Zorn's
  lemma the family of non-empty invariant closed subsets of \( \Omega \) ordered by inclusion has a
  minimal element \( M \subseteq \Omega \).  The restriction \( \mff|_{M} \) is minimal, and
  therefore does not admit a sparse cross section by the above argument.
\end{proof}

A similar argument shows that free homeomorphisms of compact metrizable spaces never admit sparse
cross sections.  We note also that if \( \mathcal{C} \) is a cross section with bounded gaps, then
the flow \( \mff \) is sparse if and only if the induced automorphism \( \phi_{\mathcal{C}} \) is
sparse.  Indeed if \( \mathcal{D} \) is a sparse cross section for \( \phi_{\mathcal{C}} \), then it
is also a sparse cross section for the flow \( \mff \) when viewed as a subset of \( \Omega \).  For
the other direction, if \( \mathcal{D} \subseteq \Omega = \mathcal{C} \times \rgap[\mathcal{C}] \)
is sparse for \( \mff \), then \( \mathrm{proj}_{\mathcal{C}}(\mathcal{D}) \) is sparse for
\( \phi_{\mathcal{C}} \).

While from the topological point of view minimal flows on compact spaces are far from sparse, any
flow is sparse as far as an ergodic theorists would care.

\begin{theorem}
  \label{thm:sparse-full-set}
  Let \( \mff \) be a free Borel flow on a standard Borel space \( \Omega \).  There exists a Borel
  invariant subset \( \omsp \subseteq \Omega \) such that \( \mff|_{\omsp} \) is sparse
  and \( \mu(\omsp) = 1 \) for any \( \mff \)-invariant probability measure \( \mu \).
\end{theorem}

\begin{proof}
  The theorem follows from a sequence of applications of descriptive Rokhlin's Lemma.  More
  formally, we may use \cite[Theorem 6.3]{slutsky_lebesgue_2015} to construct a Borel \( \mff
  \)-invariant subset \( \omsp \subseteq \Omega \) of full measure for any \( \mff
  \)-invariant finite measure on \( \Omega \), a sequence \( (l_{n})_{n=1}^{\infty} \), \( l_{n} \ge
  n\), and sets \( \mathcal{C}_{n} \subseteq \omsp \)
  satisfying the following properties fro all \( n \).
  \begin{enumerate}[(i)]
  \item\label{item:increase}
    \( \bigl( c + [-l_{n+1}, l_{n+1}) \bigr) \cap \mathcal{C}_{n} \ne \es \) for all
    \( c \in \mathcal{C}_{n+1} \);
  \item\label{item:exhaust-Omega-sp}
    \( \omsp = \bigcup_{n}\bigl( \mathcal{C}_{n} + [-l_{n}, l_{n})\bigr) \);
  \item\label{item:lacunary}
    \( \bigl( c + [-l_{n}, l_{n}) \bigr) \cap \bigl( c' + [-l_{n}, l_{n}) \bigr) = \es \) for all
    distinct \( c, c' \in \mathcal{C}_{n} \);
  \item\label{item:far-from-boundary}
    \( \mathcal{C}_{n} + [-l_{n}, l_{n}) \subseteq \mathcal{C}_{n+1} + [-l_{n+1} + n + 1, l_{n+1} -
    n -1) \).
  \end{enumerate}
  \begin{figure}[htb]
    \centering
    \begin{tikzpicture}
      \draw[dotted] (-0.7,0) -- (0, 0);
      \draw[dotted] (1,0) -- (1.5, 0);
      \draw[dotted] (2.5,0) -- (3.2, 0);
      \draw[dotted] (4.5,0) -- (5.1, 0);
      \draw[dotted] (6.1,0) -- (6.6, 0);
      \draw[dotted] (7.6,0) -- (8.4, 0);
      \draw (0, 0) -- (1,0);
      \draw (1.5, 0) -- (2.5, 0);
      \draw (5.1, 0) -- (6.1, 0);
      \draw (6.6,0) -- (7.6, 0);
      \foreach \x in {0, 1.5, 5.1, 6.6} {
        \draw (\x, 0.07) -- (\x, -0.07);
        \draw (\x+1, 0.07) -- (\x+1, -0.07);
      }
      \foreach \x in {-0.7, 4.5} {
        \draw (\x, 0.2) -- (\x, -0.2);
        \draw (\x+3.9, 0.2) -- (\x+3.9, -0.2);
      }
      \draw (1.25,-0.45) node {\( 2l_{2} \)};
      \draw (0.5,0.23) node {\( 2l_{1} \)};
      \draw (2.0,0.23) node {\( 2l_{1} \)};
      \draw[->] (0.95,-0.45) -- (-0.68,-0.45);
      \draw[->] (1.55,-0.45) -- (3.18,-0.45);
    \end{tikzpicture}
    \caption{Intervals in \( \mathcal{C}_{1} + [-l_{1}, l_{1}) \) are inside intervals of
      \( \mathcal{C}_{2} + [-l_{2}, l_{2}) \), and are far from their boundary.}
    \label{fig:increasing-interval-cross-sections}
  \end{figure}
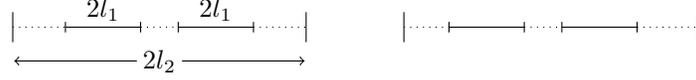
  By item \eqref{item:lacunary}, \( \mathcal{C}_{n} \) is lacunary, and by \eqref{item:increase} and
  \eqref{item:exhaust-Omega-sp} each \( \mathcal{C}_{n} \) intersects every orbit of \( \mff \) in
  \( \omsp \).  Using also \eqref{item:far-from-boundary}, it is not hard to see that each of
  \( \mathcal{C}_{n} \) must be a (lacunary bi-infinite) cross section for the restriction of
  \( \mff \) on \( \omsp \).  One can think of \( \mathcal{C}_{n} + [-l_{n}, l_{n}) \) as an
  ``interval cross section,'' where each point has been fattened into an interval of length \( 2
  l_{n} \).  The main property here is that each
  interval in \( \mathcal{C}_{n} + [-l_{n}, l_{n}) \) lies inside an interval of \( \mathcal{C}_{n+1}
  + [-l_{n+1}, l_{n+1}) \) and, moreover, it lies at distance at least \( n+1 \) from its boundary.
  Figure \ref{fig:increasing-interval-cross-sections} shows how intervals in \( \mathcal{C}_{1} +
  [-l_{1}, l_{1}) \) (solid lines) may lie relative to intervals of \( \mathcal{C}_{2} +[-l_{2},
  l_{2}) \) (dotted lines).

  The proof is completed by showing that \( \mathcal{C}_{1} \) is sparse.  Pick some
  \( x \in \mathcal{C}_{1} \), we show that \( (x + \mathbb{R}) \cap \mathcal{C}_{1} \) has
  arbitrarily large gaps.  Items \eqref{item:increase} and \eqref{item:far-from-boundary} guarantee
  that for any \( c \in \mathcal{C}_{n} \) the set
  \( \bigl(c + [-l_{n}, l_{n})\bigr) \cap \mathcal{C}_{1} \) is non-empty, and we set
  \begin{displaymath}
    \begin{aligned}
      d &= \max \bigl(c + [-l_{n}, l_{n})\bigr) \cap \mathcal{C}_{1}, \\
      d' &= \min \bigl(c' + [-l_{n}, l_{n})\bigr) \cap \mathcal{C}_{1}, \\
    \end{aligned}
  \end{displaymath}
  where \( c, c' \in \mathcal{C}_{n} \cap (x + \mathbb{R}) \) are any adjacent points, \( c < c' \).
  Elements \( d \) and \( d' \) must be adjacent in \( \mathcal{\mathcal{C}_{1}} \), and
  \( \dist(d, d') \ge 2n \) by \eqref{item:far-from-boundary}.  Since \( n \) was arbitrary,
  \( \mathcal{C}_{1} \) is sparse.
\end{proof}

\section{Overview of the Proof}
\label{sec:overview-proof}

This section contains a rough sketch of the proof of the Main Theorem.  During the first reading we
suggest to skim over and return to the relevant parts as the reader goes through the next chapters.

\subsection{The big picture}
\label{sec:big-picture}

The main argument of Theorem \ref{thm:tiling-general-flows} requires certain amount of technical
preparation, and we would like to take this opportunity and outline the big picture of the proof.
The argument splits into two parts.  First, we construct regular cross sections under the additional
assumption that the flow is sparse.  This is done in Theorem \ref{thm:tiling-sparse-flows}.  Now
suppose \( \mathcal{C} \) is a cross section which has arbitrarily large
\( \{\alpha, \beta\} \)-regular blocks within each orbit in the sense that for any
\( x \in \mathcal{C} \) and any \( N \in \mathbb{N} \) there exists \( y \in \mathcal{C} \) coming
from the same orbit, \( x\, \oer{\mathcal{C}}\, y \), such that the cardinality of the
\( \oer[\alpha,\beta]{\mathcal{C}} \)-class of \( y \) is at least \( N \):
\( \bigl| [y]_{\oer[\alpha,\beta]{\mathcal{C}}} \bigr| \ge N \).  Orbits in such a cross section
\( \mathcal{C} \) necessarily fall into three categories:
\begin{itemize}
\item Some orbits may be tiled completely, i.e., the whole orbit may constitute a single
  \( \oer[\alpha,\beta]{\mathcal{C}} \)-class.  On the part of the space which consists of these
  orbits the cross section \( \mathcal{C} \) is \( \{\alpha, \beta\} \)-regular.
\item It is possible that ``half'' of an orbit is tiled, meaning that
  \( \oer[\alpha,\beta]{\mathcal{C}} \) may have at least two equivalence classes
  within the orbit, one of which is infinite.  Restriction of the flow onto the set of such orbits
  is smooth for we may pick finite endpoints of infinite classes to get a Borel transversal.
  Constructing a regular cross section on this part of the space will therefore be trivial.
\item Finally, in a typical orbit each \( \oer[\alpha,\beta]{\mathcal{C}} \)-class will be finite.
  Consider the sub cross section that consists of endpoints of \( \oer[\alpha, \beta]{\mathcal{C}}
  \)-classes:
  \[ \mathcal{C}' = \bigl\{x \in \mathcal{C} : x = \min [x]_{\oer[\alpha,\beta]{\mathcal{C}}}
  \textrm{ or } x = \max [x]_{\oer[\alpha,\beta]{\mathcal{C}}} \bigr\}. \]
  A simple but crucial observation is that the assumption of having arbitrarily large \(
  \oer[\alpha, \beta]{\mathcal{C}} \)-classes within each orbit implies that the cross section \(
  \mathcal{C}' \) \textit{will necessarily be sparse!}
\end{itemize}
To summarize, given a cross section with unbounded \( \oer[\alpha, \beta]{\mathcal{C}} \)-classes
within each orbit, the phase space splits into three invariant Borel parts: a sparse piece where
Theorem \ref{thm:tiling-sparse-flows} applies, a smooth piece where constructing any kind of cross
section is a triviality, and a piece where \( \mathcal{C} \) is already an
\( \{\alpha, \beta\} \)-regular cross section.  Once Theorem \ref{thm:tiling-sparse-flows} is
proved, the problem of constructing a regular cross section for a general Borel flow is therefore
reduced to a problem of constructing a cross section with arbitrarily large \( \oer[\alpha, \beta]{}
\)-classes.  Theorem \ref{thm:tiling-general-flows} achieves just that.

\subsection{Sparse case}
\label{sec:sparse-case}

Let us now explain how sparsity of the flow is helpful in constructing a regular cross section.  If
\( \alpha \) and \( \beta \) are positive rationally independent reals, then the set \( \tileable \)
of reals of the form \( p\alpha + q\beta \), \( p,q \in \mathbb{N} \), (we call such reals
\emph{tileable}) is \emph{asymptotically dense} in \( \mathbb{R} \) in the sense that for any
\( \epsilon > 0 \) this set is \( \epsilon \)-dense in \( [K, \infty) \) for sufficiently large
\( K \) (see Definition \ref{def:asymptotic-density}).  We may therefore pick a sequence \( K_{n} \)
growing so fast that \( \tileable \) is \( \epsilon_{n+1} \)-dense in\footnote{In fact, we need
  \( \tileable \) to be \( \epsilon_{n+1} \)-dense in \( [K_{n}-4, \infty) \), because during our
  construction points will be shifted a bit, but the total shift of each point will never exceed
  \( 2 \), so if we start with two points which are distance \( K_{n} \) apart, then during the
  whole process they will remain at least \( K_{n} - 4 \) apart.  In this sketch we ignore minor
  changes in distance because of point shifts.}  \( [K_{n},\infty) \), where
\( \epsilon_{n} = 2^{-n} \).  Given a sparse cross section \( \mathcal{C} \) we may now argue as
follows.

By passing to a sub cross section if necessary we may suppose without loss of generality that
\( \rgap[\mathcal{C}](x) \ge K_{0} \) for all \( x \in \mathcal{C} \).  Consider the relation
\( \oer[\le K_{1}]{\mathcal{C}} \) on \( \mathcal{C} \).  Sparsity ensures that each
\( \oer[\le K_{1}]{\mathcal{C}} \)-class is necessarily finite.  Take a single
\( \oer[\le K_{1}]{\mathcal{C}} \)-class.  Let us say it consists of points
\( x_{0}, x_{1}, \ldots, x_{N} \) listed in the increasing order (Figure \ref{fig:tiling-sparse}).
By construction gaps between adjacent points within this class are between \( K_{0} \) and
\( K_{1} \).  By the choice of \( K_{0} \) we may shift \( x_{1} \) by at most
\( \epsilon_{1} = 1/2 \) to a new position \( x_{1}' \) such that
\( \dist(x_{0}, x_{1}') \in \tileable \).  Since \( \dist(x_{0}, x_{1}') \) is of the form
\( p\alpha + q\beta \), where \( p \) and \( q \) are non-negative integers, we may add
\( p + q - 1 \) points to the interval \( [x_{0}, x_{1}'] \) to make every gap between adjacent
points be either \( \alpha \) or \( \beta \).  We call the processing of adding such points
\emph{tiling the gap}.
\begin{figure}[ht]
  \centering
  \begin{tikzpicture}
    \foreach \p/\x in {0/0, 1/1.1, 2/2.3, 3/3.6} {
      \filldraw (\x, 0) circle (1pt);
      \draw (\x, 0.2) node {\( x_{\p} \)};
    }
    \draw (1.8,-1) node[text width=5cm] {
        Moving each point by at most \( \epsilon_{1} = 1/2 \)
        and tiling the gaps.
    };
    \filldraw (0, -2) circle (1pt);
    \draw (0, -2.3) node {\( x_{0} \)};
    \foreach \p/\x in {1/1.0, 2/2.4, 3/3.5} {
      \filldraw (\x, -2) circle (1pt);
      \draw (\x, -2.3) node {\( x'_{\p} \)};
    }
    \foreach \x in {0.0, 0.1, ..., 3.5} {
      \filldraw (\x,-2) circle (0.5pt);
    }
  \end{tikzpicture}
  \caption{Tiling sparse cross section}
  \label{fig:tiling-sparse}
\end{figure}
One now proceeds in the same fashion with \( x_{2} \) \textemdash{} the distance from \( x_{1}' \) to \( x_{2} \)
differs from \( \dist(x_{1}, x_{2}) \) by no more than \( \epsilon_{1} \), so is still large enough,
and we may therefore shift \( x_{2} \) by at most \( \epsilon_{1} \) to a new position \( x_{2}' \)
in such a way that \( \dist(x_{1}', x_{2}') \in \mathcal{T} \).  Once \( x_{2} \) has been shifted
we tile the gap by adding extra points to \( [x_{1}', x_{2}'] \).  This process is continued until we
reach the maximal point \( x_{N} \) within the given \( \oer[\le K_{1}]{\mathcal{C}} \)-class.
Note that the amount of shift does not grow \textemdash{} each point is shifted by at most \(
\epsilon_{1} \), and the distance \( \dist(x_{k}', x_{k+1}) \) is always at least \( K_{0} -
\epsilon_{1} \) for all \( k \), so it will always be possible to shift \( x_{k+1} \) to \( x_{k+1}'
\) which makes \( \dist(x_{k}', x_{k+1}') \in \mathcal{T} \) by the choice of \( K_{0} \).  This
procedure is applied to all \( \oer[\le K_{1}]{\mathcal{C}} \)-classes.

After the first step of the construction we arrive at a cross section \( \mathcal{C}_{1} \) which
differs from the cross section \( \mathcal{C} \) in two aspects.  Some of the points in
\( \mathcal{C}_{1} \) correspond to points from \( \mathcal{C} \) shifted by at most
\( \epsilon_{1} \), and other points have been added to tile the gaps.  Note that
\( \mathcal{C}_{1} \) is still sparse.  Consider now a \( \oer[\le K_{2}]{\mathcal{C}_{1}} \)-class.
It will consist of a number of \( \oer[\alpha, \beta]{\mathcal{C}_{1}} \)-classes separated by gaps
of size at least\footnote{In fact, gaps can be as small as \( K_{1} - \epsilon_{1} \), but we agreed
  to ignore shifts for now.} \( K_{1} \) (Figure \ref{fig:tiling-sparse-2}).
\begin{figure}[ht]
  \centering
  \begin{tikzpicture}
    \foreach \x in {0, 1.6, 3.3, 5.0, 6.6, 7.6} { \filldraw (\x, 0) circle (1pt); }
    \foreach \x in {0.0, 0.1, ..., 1.6,3.3, 3.4, ...,5.0,6.6,6.7, ...,7.6} {
      \filldraw (\x, 0) circle (0.5pt);
    }
    \draw (2.45, 0) node {\( \ge K_{1} \)};
    \draw (5.8, 0) node {\( \ge K_{1} \)};
    \draw (3.8, -1) node[text width=6cm] {
        Moving each \( \oer[\alpha,\beta]{\mathcal{C}_{1}} \)-class by at most \( \epsilon_{2} \)
        and tiling the gaps between them.
    };
    \foreach \x in {0, 1.6, 3.2, 4.9, 6.8, 7.8} { \filldraw (\x, -2) circle (1pt); }
    \foreach \x in {0.0, 0.1, ..., 7.8} {\filldraw (\x, -2) circle (0.5pt);}
  \end{tikzpicture}

  \caption{Tiling sparse cross section, step 2.}
  \label{fig:tiling-sparse-2}
\end{figure}
We may now run the same process as before, but now with smaller shifts applied to \( \oer[\alpha,
\beta]{\mathcal{C}_{1}} \)-blocks.  By the choice of \( K_{1}
\), the whole second  \( \oer[\alpha, \beta]{\mathcal{C}_{1}} \)-class can be shifted by at most \(
\epsilon_{2} \) in such a way that the distance between the first and the second class becomes
a real in \( \tileable \).  Doing this with each \( \oer[\alpha, \beta]{\mathcal{C}_{1}} \)-class
one by one and tiling the gaps in the midst results in a cross section \( \mathcal{C}_{2} \) in which
\( \oer[\alpha, \beta]{\mathcal{C}_{2}} \)-classes correspond to \( \oer[\le K_{1}]{\mathcal{C}_{1}}
\)-classes in \( \mathcal{C}_{1} \).  Moreover, \( \mathcal{C}_{2} \) is sparse, and the distance
between distinct \( \oer[\alpha, \beta]{\mathcal{C}_{2}} \)-classes is at least \( K_{2} \).  The
construction continues.

To summarize, at step \( n+1 \) we shift points in \( \mathcal{C}_{n} \) by at most
\( \epsilon_{n+1} \) and add a handful of new points between
\( \oer[\alpha, \beta]{\mathcal{C}_{n}} \)-classes.  Since \( \sum \epsilon_{n}\) converges, each
point ``converges'' to a limit position, and the limit tiling is \( \{\alpha, \beta\} \)-regular.

This sketch can easily be turned into a rigorous argument which is just a different presentation of
the one given in Section 2 of \cite{rudolph_two-valued_1976}.  Additional care seems to be necessary
if one wants to control the frequency of \( \oer[\alpha, \beta]{} \)-blocks during the construction.
More precisely, let \( x \) and \( y \) be respectively the minimal and the maximal point of an
\( \oer[\alpha, \beta]{\mathcal{C}_{n}} \)-class.  The distance \( \dist(x, y) \) is of the form
\( p\alpha + q\beta \) for some \( p, q \in \mathbb{N} \), and we consider the quantity
\( p/(p + q) \) which represents the frequency of gaps of size \( \alpha \) in this class.  This
quantity is called an \( \alpha \)-\emph{frequency} of \( \dist(x,y) \).  We need to run our
construction in a way ensuring this frequency converges to \( \rho \) as \( n \to \infty \).  This
runs into the following problem.  First of all, we need \( K_{n} \) to be so large that reals of the
form \( p\alpha + q\beta \), where
\[ \Bigl| \frac{p}{p+q} - \rho \Bigr| < \frac{1}{n}, \]
are \( \epsilon_{n+1} \)-dense in \( [K_{n}, \infty) \).  More importantly, we need to improve the
frequency, i.e., make it closer to \( \rho \), when passing from \( n^{\textrm{th}} \) to the
\( n+1^{\textrm{st}} \) step.  Difficulty lies in the fact that we have no control on the length of
\( \oer[\alpha, \beta]{\mathcal{C}_{n}} \)-blocks (they can be arbitrarily large) nor do we have any
control on the number of \( \oer[\alpha,\beta]{\mathcal{C}_{n}} \)-classes inside a single
\( \oer[\le K_{n+1}]{\mathcal{C}_{n}} \)-class.  To overcome these obstacles we run our construction
is such a way as to have an absolute control on the proximity of \( \alpha \)-frequencies of
\( \oer[\alpha,\beta]{\mathcal{C}_{n}} \)-classes to \( \rho \) which is achieved by introducing the
concept of a real \( p\alpha + q\beta \in \tileable \) being \( N \)-near \( \rho \) (see Definition
\ref{def:near-rho}).  Here is what it means.  We shall have a natural number \( N_{n} \) such that
given any \( \oer[\alpha,\beta]{\mathcal{C}_{n}} \)-class \( p\alpha + q\beta \) with the frequency
of \( \alpha \)-intervals being less than \( \rho \), i.e., \( p/(p+q) < \rho \), adding
\( N_{n} \)-many \( \alpha \)-intervals will push the frequency above \( \rho \):
\[ \frac{ p + N_{n}}{p + q + N_{n}} \ge \rho. \]
And also the other way around: for any \( \oer[\alpha,\beta]{\mathcal{C}_{n}} \)-class with the
frequency of \( \alpha \)-intervals above \( \rho \), adding \( N_{n} \)-many \( \beta \)-intervals
will bring the frequency below \( \rho \).  The construction will be run in such a way that
\( K_{n+1} \) is selected after \( N_{n} \) has been specified.  And this will make it possible to
ensure that after an \( \epsilon_{n+1} \)-shift we can make each gap between
\( \oer[\alpha,\beta]{\mathcal{C}_{n}} \)-classes to be a tileable real with \( \alpha \)-frequency
very close to \( \rho \) and with an ``excess'' of \( N_{n} \)-many \( \alpha \)- or
\( \beta \)-intervals at our choice.  This will allow us to have a better approximation of
\( \rho \) by \( \alpha \)-frequencies of \( \oer[\alpha,\beta]{\mathcal{C}_{n+1}} \)-classes.

 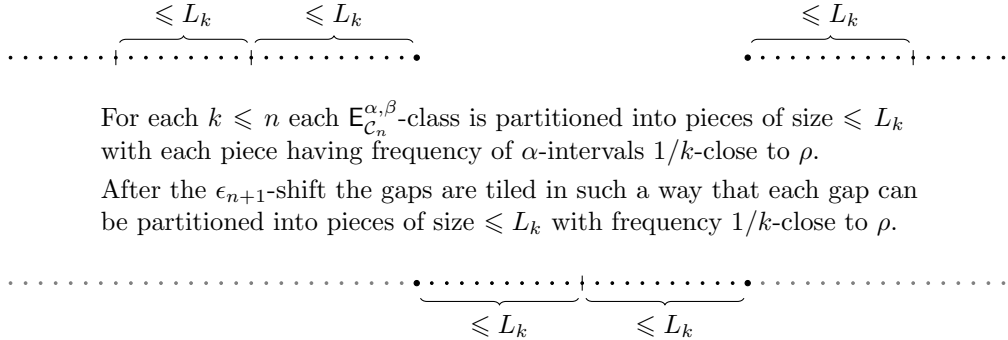
\begin{figure}[hbt]
   \centering
   \begin{tikzpicture}
     \filldraw (6,0) circle (1pt);
     \foreach \x in {0.6, 0.8, ..., 6.0} {
       \filldraw (\x, 0) circle (0.5pt);
     }
     \foreach \x in {2.0, 3.8} {
       \draw (\x,0.1) -- (\x,-0.1);
     }
     \draw[decoration={
      brace,
      raise=2mm
    }, decorate] (2.05,0) -- (3.75,0) node[pos=0.5, anchor=south, yshift=3mm]
    {\( \le L_{k} \)};
    \draw[decoration={
      brace,
      raise=2mm
    }, decorate] (3.85,0) -- (5.95,0) node[pos=0.5, anchor=south, yshift=3mm]
    {\( \le L_{k} \)};
    \filldraw (10.4,0) circle (1pt);
    \foreach \x in {10.4, 10.6, ..., 14.0} {
       \filldraw (\x, 0) circle (0.5pt);
     }
     \draw (12.6,0.1) -- (12.6,-0.1);
     \draw[decoration={
      brace,
      raise=2mm
    }, decorate] (10.45,0) -- (12.55,0) node[pos=0.5, anchor=south, yshift=3mm]
    {\( \le L_{k} \)};
    \draw (7.3,-1.0) node[text width=11cm] {
        For each \( k \le n \) each \( \oer[\alpha,\beta]{\mathcal{C}_{n}}
        \)-class is partitioned into pieces of size \( \le L_{k} \)
        with each piece having frequency
        of \( \alpha \)-intervals \( 1/k \)-close to \( \rho \).
    };
    \draw (7.3,-2.0) node[text width=11cm] {
        After the \( \epsilon_{n+1} \)-shift the gaps are tiled in such a way that
        each gap can be partitioned into pieces of size \( \le L_{k} \) with frequency \( 1/k
        \)-close to \( \rho \).
    };
    \foreach \x in {6.0, 6.2, ..., 10.4} {
       \filldraw (\x, -3) circle (0.5pt);
     }
     \filldraw (6,-3) circle (1pt);
     \filldraw (10.4,-3) circle (1pt);
     \draw (8.2,-2.9) -- (8.2,-3.1);
     \draw[decoration={
       brace,
       mirror,
      raise=2mm
    }, decorate] (6.05,-3) -- (8.15,-3) node[pos=0.5, anchor=north, yshift=-3mm]
    {\( \le L_{k} \)};
    \draw[decoration={
       brace,
       mirror,
      raise=2mm
    }, decorate] (8.25,-3) -- (10.35,-3) node[pos=0.5, anchor=north, yshift=-3mm]
    {\( \le L_{k} \)};
    \foreach \x in {0.6,0.8,...,5.8,10.6,10.8,...,13.8} {
      \filldraw[gray] (\x,-3) circle (0.5pt);
    }
  \end{tikzpicture}
  \caption{Witnessing uniform convergence of frequencies to \( \rho \).}
  \label{fig:uniform-to-rho}
\end{figure}

But we shall need more.  We want the convergence to \( \rho \) in the Main Theorem to be uniform
over all \( x \in \mathcal{C} \).  Lemma \ref{lem:equivalent-formulation-of-regularity} offers the
following witness for such a uniformity.  By step \( n \) we shall pick real numbers \( L_{k} \),
\( k \le n \), and for each \( k \le n \) we shall have a partition if each
\( \oer[\alpha,\beta]{\mathcal{C}_{n}} \)-class into pieces of length at most \( L_{k} \) each piece
having frequency of \( \alpha \)-intervals \( 1/k \)-close to \( \rho \).  Figure
\ref{fig:uniform-to-rho} illustrates this situation.  What we need is to shift
\( \oer[\alpha,\beta]{\mathcal{C}_{n}} \)-classes by at most \( \epsilon_{n+1} \) and tile the gaps,
but the tiling is supposed to be special: for each \( k \le n \) we want to be able to partition
the gap into pieces of size \( \le L_{k} \) each having frequency \( 1/k \)-close to \( \rho \).
Pieces from \( \oer[\alpha,\beta]{\mathcal{C}_{n}} \)-classes together with pieces from tiled gaps
will constitute the partition of \( \oer[\alpha,\beta]{\mathcal{C}_{n+1}} \)-class.  Lemma
\ref{lem:regular-tilings} computes how large given \( L_{k} \), \( k \le n \), we need to pick
\( K_{n+1} \) so as to make the above construction possible.

To finish the inductive step it remains to pick \( L_{n+1} \).  Remember that we have no control on
how many \( \oer[\alpha,\beta]{\mathcal{C}_{n}} \)-classes fall into a single
\( \oer[\le K_{n+1}]{\mathcal{C}_{n}} \)-class, so the resulting
\( \oer[\alpha,\beta]{\mathcal{C}_{n+1}} \) may be arbitrarily long, but we need to argue that there
exists a single \( L_{n+1} \) such that when \( \oer[\alpha, \beta]{\mathcal{C}_{n+1}} \) is built
according to the described procedure it can always be partitioned into blocks of size
\( \le L_{n+1} \) each having frequency \( 1/(n+1) \)-close to \( \rho \).  This is the content of
Lemmata \ref{lem:tiled-sum} and \ref{lem:sparse-induction-step} with the latter one encompassing
precisely the set up of the induction step in the sparse case.

This finishes a rough summary of the content of Section \ref{sec:flex-part}.  We have mentioned on a
number of occasions that each \( \oer[\le K_{n+1}]{\mathcal{C}_{n}} \)-class may consists of many
\( \oer[\le K_{n}]{\mathcal{C}_{n}} \)-classes, but for technical reasons it is desirable to know
that it always consists of at least two such classes.  Construction of such a cross section is the
content of Section \ref{sec:tools} and Lemma \ref{lem:at-least-two-classes-Kn} specifically.
Moreover, the cross section constructed therein has that property in a stable way: even if each
point in \( \mathcal{C} \) is perturbed by at most \( \sum \epsilon_{k} \), each
\( \oer[\le K_{n+1}]{} \)-class will still consist of at least two \( \oer[\le K_{n}]{} \)-classes.
This allows us not to worry about the minor change is gaps' sizes as the inductive construction
unfolds.

\subsection{Co-sparse case}
\label{sec:co-sparse-case}

As we have speculated in Subsection \ref{sec:big-picture} above, to boost the sparse argument to
the general situation it is enough to show how given a Borel flow to construct a cross section \(
\mathcal{C} \) with arbitrarily large \( \oer[\alpha,\beta]{\mathcal{C}} \)-classes within each
orbit.  Our basic approach is similar to the sparse case: we
construct a sequence of cross sections \( \mathcal{C}_{n} \), which ``converge'' to a
limit cross section \( \mathcal{C}_{\infty} \).  
The next cross section in the sequence is obtained from the previous one by shifting some
of its elements and adding extra points to tile a few gaps.  Convergence of \( \mathcal{C}_{n} \)
relies upon moving points by smaller and smaller amounts, and here lies the first important
difference.  In the sparse argument, there was the following
uniformity in the size of points' jumps: when constructing \( \mathcal{C}_{n+1} \) out of
\( \mathcal{C}_{n} \) any point was moved by no more than \( \epsilon_{n} \).  This time there will
be no uniformity of that sort.  But it will still be the case that the first jump of any point is
at most \( \epsilon_{1} \), the second one will be bounded by \( \epsilon_{2} \), etc., but each
point will have its own sequence \( (n_{k}) \) of indices when jumps are made.  When passing from
\( \mathcal{C}_{n} \) to \( \mathcal{C}_{n+1} \), some of the points will make their first jump,
which can be as large as \( \epsilon_{1} \), others will jump for the second time, etc., and there will
be points that were jumping all the time from \( \mathcal{C}_{0} \) to \( \mathcal{C}_{n+1} \).

The construction starts with a cross section \( \mathcal{C}_{0} \) with bounded gaps.  When building
\( \mathcal{C}_{1} \), we shall take sufficiently distant pairs of adjacent points in
\( \mathcal{C}_{0} \), within each pair we shall move the right point making the gap to the left one
tileable, and finally we shall add points tiling these gaps.  Figure
\ref{fig:constructing-C1-sketch} tries to illustrate the process.

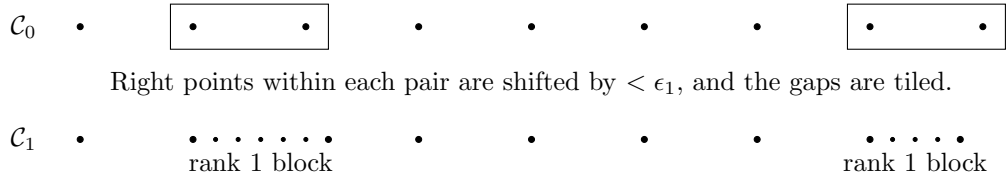
\begin{figure}[ht]
  \centering
  \begin{tikzpicture}[scale=1.5]
    \draw (-0.5,0) node {\( \mathcal{C}_{0} \)};
    \draw (-0.5,-1) node {\( \mathcal{C}_{1} \)};
    \foreach \x in {0,1,...,8}{
      \filldraw (\x,0) circle (0.75pt);
    }
    \draw (0.8,-0.2) rectangle (2.2,0.2);
    \draw (6.8,-0.2) rectangle (8.2,0.2);
    \draw (4,-0.5) node {Right points within each pair are shifted by \( < \epsilon_{1} \), and the
      gaps are tiled.};
    \foreach \x in {0,1,3,4,5,6,7}{
      \filldraw (\x,-1) circle (0.75pt);
    }
    \filldraw (2.2,-1) circle (0.75pt);
    \filldraw (7.8,-1) circle (0.75pt);
    \foreach \x in {1.2,1.4,...,2.0} {
      \filldraw (\x, -1) circle (0.5pt);
    }
    \foreach \x in {7.2,7.4,7.6} {
      \filldraw (\x, -1) circle (0.5pt);
    }
    \foreach \x in {1.6,7.4} {
      \draw (\x,-1.2) node {rank \( 1 \) block};
    }
  \end{tikzpicture}
  \caption{Constructing \( \mathcal{C}_{1} \).}
  \label{fig:constructing-C1-sketch}
\end{figure}

The relation \( \oer[\alpha,\beta]{\mathcal{C}_{1}} \) will have classes of two sorts.  Some of them
will be just isolated points, we call them rank \( 0 \) classes, while others will be given by the
selected pairs together with the new points added in the midst.  These are rank \( 1 \) blocks.

What we would like to do at the second step is to take sufficiently distant pairs of adjacent rank
\( 1 \) blocks, move the right block by at most \( \epsilon_{2} \), move rank \( 0 \) points inside
by no more than \( \epsilon_{1} \) making all the gaps tileable, and tile these gaps creating thus a
rank \( 2 \) block.  See Figure \ref{fig:Creating-rank-2}.

\begin{figure}[ht]
  \centering
  \begin{tikzpicture}
    \draw (-0.8,2.6) node {\( \mathcal{C}_{1} \)};
    \draw (-0.8,0) node {\( \mathcal{C}_{2} \)};
    \foreach \x in {0,0.5,...,15} {
      \filldraw[black] (\x, 2.6) circle (1pt);
    }
    \foreach \x in {0,2,...,14} {
      \foreach \y in {0.1,0.2,...,0.4}{
        \filldraw (\x+\y,2.6) circle (0.5pt);
      }
      \draw (\x+0.25,2.6) node[above] {\( 1 \)} coordinate(x axis);
    }
    \draw (1.8,2.4) rectangle (4.7,3.1);
    \draw (11.8,2.4) rectangle (14.7,3.1);
    \draw (7.5,1.3) node[text width=9cm] {Within each rectangle, isolated points are moved by
      \( < \epsilon_{1} \), the right rank \( 1 \) block is shifted by at most \( \epsilon_{2} \),
      and gaps are tiled.};
    \foreach \x in {0,0.5,...,15} {
      \filldraw[black] (\x, 0) circle (1pt);
    }
    \foreach \x in {0,2,2.5,3,3.5,4,6,8,10,12,12.5,13,13.5,14} {
      \foreach \y in {0.1,0.2,...,0.4}{
         \filldraw (\x+\y,0) circle (0.5pt);
      }

    }
    \foreach \x in {0,6,8,10} {
      \draw (\x+0.25,0) node[below] {\( 1 \)} coordinate(x axis);
    }
    \foreach \x in {2,12} {
      \draw (\x+1.25,0) node[below] {\( 2 \)} coordinate(x axis);
    }
  \end{tikzpicture}
  \caption{Constructing \( \mathcal{C}_{2} \)}
  \label{fig:Creating-rank-2}
\end{figure}
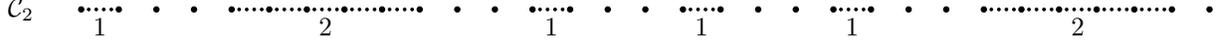

The key technicality lies in ensuring that each rank \( 1 \) block can be moved by at most
\( \epsilon_{2} \) rather than \( \epsilon_{1} \).  We don't have any problems in constructing
\( \mathcal{C}_{1} \) \textemdash{} we may choose \( \mathcal{C}_{0} \) to have all the gaps sufficiently
large as to allow for turning each gap into a tileable real after an \( \epsilon_{1} \)
perturbation.  But imagine now that in \( \mathcal{C}_{0} \) we always have several ways of moving
points by at most \( \epsilon_{1} \) to make gaps tileable; i.e., in the notation of Figure
\ref{fig:tiling-sparse}, we have several ways of moving \( x_{1} \).  For {\it each} such shift we
have {\it several} ways to shift \( x_{2} \), and for each shift of \( x_{2} \) we have several
possibilities for \( x_{3} \), etc.  While some of the arrangements may coincide, it is nonetheless
natural to expect that the number of possible terminal positions for \( x_{N} \), as each
\( x_{k} \), \( k \le N \), is shifted by \( \le \epsilon_{1} \), will grow with \( N \), and in
fact, will eventually densely pack \( \epsilon_{1} \)-neighborhood of \( x_{N} \).  Formalization of
this intuition is called ``propagation of freedom''.  In a short of form in can be summarized as
follows.  Let \( \epsilon > 0 \), let \( N \in \mathbb{N} \), and let
\( x_{0}, x_{1}, \ldots, x_{N} \) be a family of points such that for each \( i < N \) there are
``sufficiently many'' ways to shift \( x_{i+1} \) to \( x_{i+1}' \) by at most \( \epsilon \) so as
to make \( \dist(x_{i}, x_{i+1}') \) tileable.  Let \( \mathcal{A}_{N} \) denote the set of all
possible\footnote{\( \mathcal{U}_{\epsilon}(x_{N}) \) denote the \( \epsilon \)-neighborhood of
  \( x_{N} \), i.e., \( \mathcal{U}_{\epsilon}(x_{N}) = (x_{N} - \epsilon, x_{N} + \epsilon) \).}
\( \tilde{z} \in \mathcal{U}_{\epsilon}(x_{N}) \) for which there exists
\( x_{i}' \in \mathcal{U}_{\epsilon}(x_{i}) \) satisfying
\( \dist(x_{i}',x_{i+1}' ) \in \tileable \) for all \( i \le N \), \( x_{0}' = x_{0} \), and
\( x_{N}' = \tilde{z} \).  In the notation of Figure \ref{fig:tiling-sparse} this is the set of all
possible terminal positions for \( x_{N}' \).  The ``propagation of freedom'' principle states that
\( \mathcal{A}_{N} \) is \( \delta \)-dense in \( \mathcal{U}_{\epsilon}(x_{N}) \), where
\( \delta \to 0 \) as \( N \to \infty \).  In other words, the cumulative shift of \( x_{N} \) can
be made arbitrarily small by picking \( N \) sufficiently large.

The second step of the inductive construction corresponds to picking \( \delta = \epsilon_{2} \) in
the propagation of freedom principle.  In fact, to run the construction further we shall need to
take \( \delta \) so small that we have ``sufficiently many'' shifts of \( x_{N} \) by at most
\( \epsilon_{2} \).  Section \ref{sec:propagation-freedom} formalizes this approach and Lemma
\ref{lem:main-induction-lemma} encapsulates the inductive step of the construction.  There are
certain degenerate ways of picking admissible shifts for points \( x_{i} \), where the ``freedom''
does not ``propagate'', so to avoid these exceptional cases we need to control two things about sets
\( \mathcal{A}_{N} \).  First, we need to know how dense they are in
\( \mathcal{U}_{\epsilon}(x_{N}) \), but also we need to control the variety of
\( \alpha \)-frequencies of elements in \( \mathcal{A}_{N} \) (i.e., the \( \alpha \)-frequencies of
\( \dist(x_{0}, x_{N}') \)).

It is worth mentioning the following difference from the sparse argument here.  As was explained at
the beginning of Subsection \ref{sec:sparse-case}, constructing an \( \{\alpha,\beta\} \)-regular
cross section in the sparse case without the additional control on the distribution of
\( \alpha \)-intervals is much easier.  This does not seem to be the case in the co-sparse piece of
the argument.  Controlling frequencies is a natural way of avoiding the aforementioned degenerate
choices of admissible shifts, and we are not aware of any substantial simplifications in the
argument if one is not interested in the distribution of \( \alpha \)-intervals in the Main Theorem.

Once the ``propagation of freedom'' lemma is available, the inductive construction continues in a
way no different from the described step.  We shall continue by selecting sufficiently distant pairs
of adjacent rank \( 2 \) intervals and moving all the rank \( 2 \) blocks by at most
\( \epsilon_{3} \) in an a admissible way, i.e., in a way that results in moving each rank \( 1 \)
block in between by at most \( \epsilon_{2} \) and each rank \( 0 \) point by no more
than \( \epsilon_{1} \) and tiling all the resulting gaps, creating in this fashion rank \( 3 \)
blocks.  Since at each step of the construction we increase the size of the maximal block within any
orbit, it is evident that in the limit the cross section will have arbitrarily large
\( \oer[\alpha, \beta]{} \)-blocks within each orbit.

\section{Modifying cross sections}
\label{sec:tools}

Let \( \mathcal{C} \) be a sparse cross section for a flow \( \mff \).  For any real \( K \) each
\( \oer[\le K]{\mathcal{C}} \)-class is finite, and if \( (K_{n})_{n=0}^{\infty} \) is an increasing
sequence of reals, \( K_{n} \to \infty \), then
\( \oer{\mathcal{C}} = \bigcup_{n} \oer[\le K_{n}]{\mathcal{C}} \).  In this section we present
a construction of a cross section \( \mathcal{C} \) such that for a given \( (K_{n}) \) each
\( \oer[\le K_{n+1}]{\mathcal{C}} \)-class consists of at least two
\( \oer[\le K_{n}]{\mathcal{C}} \)-classes.

We start with subsets \( B(K_{0}, \ldots, K_{n}) \subset \mathbb{R} \) of the real line where each
\( \oer[\le K_{n+1}]{} \)-class consists of precisely two \( \oer[\le K_{n}]{} \)-classes.  For
example, consider a subset \( B(K_{0},K_{1},K_{2},K_{3}) \) with distances between adjacent points
of the form \( \frac{K_{n+1} + K_{n}}{2} \) depicted if Figure \ref{fig:constructing-BK1K2-Kn}.
\begin{figure}[ht]
  \centering
  \begin{tikzpicture}
    \foreach \x in {0, 1, 3.2, 4.2, 8.7, 9.7, 11.9, 12.9}{
      \filldraw (\x,0) circle (1pt);
    }
    \draw (0.5,0.3) node {\( \frac{K_{0} + K_{1}}{2} \)}
    (2.1,0.3) node {\( \frac{K_{1} + K_{2}}{2} \)}
    (3.7,0.3) node {\( \frac{K_{0} + K_{1}}{2} \)}
    (6.45,0.3) node {\( \frac{K_{2} + K_{3}}{2} \)}
    (9.2,0.3) node {\( \frac{K_{0} + K_{1}}{2} \)}
    (10.8,0.3) node {\( \frac{K_{1} + K_{2}}{2} \)}
    (12.4,0.3) node {\( \frac{K_{0} + K_{1}}{2} \)};
  \end{tikzpicture}
  \caption{The set \( B(K_{0},K_{1},K_{2},K_{3}) \).  The unique \( \mathsf{E}^{\le K_{3}} \)-class has two
    \( \mathsf{E}^{\le K_{2}} \)-classes, each having two \( \mathsf{E}^{\le K_{1}} \)-classes.}
  \label{fig:constructing-BK1K2-Kn}
\end{figure}
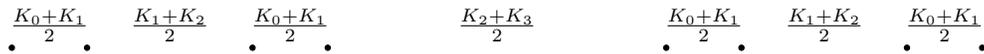
It is easy to see that each \( \oer[\le K_{n}]{} \)-class consists of two
\( \oer[\le K_{n-1}]{} \)-classes for \( n \) up to three.  We shall modify a sparse cross section
by adding sets of the form \( B(K_{0}, \ldots, K_{n}) \) into sufficiently large gaps.

Now, in a more formal language, given a finite or infinite increasing sequence
\( (K_{n})_{n=0}^{\infty} \) we define inductively subsets
\( B(K_{0}, \ldots, K_{n}) \subset \mathbb{R}^{\ge 0} \) and reals
\( b(K_{0}, \ldots, K_{n}) \in \mathbb{R}^{\ge 0} \) for \( n \in \mathbb{N} \) as follows.  For
\( n = 0 \) let \( B(K_{0}) = \{0\} \) and \( b(K_{0}) = 0 \), and set
\begin{displaymath}
  \begin{aligned}
 B(K_{0}, \ldots, K_{n+1}) &= B(K_{0}, \ldots, K_{n}) \cup \Bigl(B(K_{0},\ldots, K_{n}) +
\frac{K_{n+1} + K_{n}}{2}\Bigr), \\
 b(K_{0}, \ldots, K_{n+1}) &= 2b(K_{0}, \ldots, K_{n}) + \frac{K_{n+1}+K_{n}}{2}.
  \end{aligned}
\end{displaymath}
The number \( b(K_{0}, \ldots, K_{n}) \) is just the length of the block
\( B(K_{0}, \ldots, K_{n}) \).  Blocks \( B(K_{0}, \ldots, K_{n}) \) also have the property that
gaps between adjacent \( \oer[\le K_{i}]{} \)-classes within a \( \oer[\le K_{i+1}]{} \)-class are
in the interior of intervals \( [K_{i}, K_{i+1}] \).  This will imply stability of
\( \oer[\le K_{i}]{} \)-classes under small perturbations of points.

\begin{lemma}
  \label{lem:arbitrarily-large-sparse-two-classes}
  Let \( \mff \) be a sparse flow on a standard Borel space \( \Omega \).  For any sequence
  \( (M_{n})_{n=0}^{\infty} \) there is an increasing unbounded sequence \( (N_{n})_{n=0}^{\infty} \) and
  a sparse cross section \( \mathcal{C} \subset \Omega \) such that
  \begin{enumerate}[(i)]
  \item \( \rgap[\mathcal{C}](x) > M_{0} \) for all \( x \in \mathcal{C} \).
  \item \( N_{n} \ge M_{n} \) for all \( n \) and \( N_{0} = M_{0} \).
  \item Each \( \oer[\le N_{n+1}]{\mathcal{C}} \)-class consists of at least two \(
    \oer[\le N_{n}]{\mathcal{C}} \)-classes.
  \end{enumerate}
\end{lemma}

\begin{proof}
  Set \( N_{0} = M_{0} \) and
  \[ N_{n+1} = \max \bigl\{ M_{n+1}, b(N_{0}, \ldots, N_{n}) + 2N_{n} + 2 \bigr\}. \]
  Let \( \mathcal{D}_{0} \) be any sparse section such
  that \( \rgap[\mathcal{D}_{0}](x) > M_{0} \) for all \( x \in \mathcal{D}_{0} \) (it exists by
  Remark \ref{rem:sub-cross-section-with-large-gaps}) and define \( \mathcal{D}_{n} \) inductively by
  \[ \mathcal{D}_{n+1} = \mathcal{D}_{n} \cup \bigl\{\, x + N_{n} + 1 +B(N_{0}, \ldots, N_{n})
  \bigm| x \in \mathcal{D}_{n},\ x = \max [x]_{\oer[\le N_{n}]{\mathcal{D}_{n}}} \textrm{ and }
  [x]_{\oer[\le N_{n}]{\mathcal{D}_{n}}} = [x]_{\oer[\le N_{n+1}]{\mathcal{D}_{n}}} \bigr\}. \]
  Here is a less formal definition.  If an \( \oer[\le N_{n}]{\mathcal{D}_{n}} \)-class happens to
  coincide with the \( \oer[\le N_{n+1}]{\mathcal{D}_{n}} \)-class, then the gap to the next
  \( \oer[\le N_{n}]{\mathcal{D}_{n}} \)-class is more than \( N_{n+1} \) and in this gap we insert
  a \( B(N_{0}, \ldots, N_{n}) \) block at the distance \( N_{n} + 1 \) from the
  \( \oer[\le N_{n}]{\mathcal{D}_{n}} \)-class, Figure \ref{fig:constructing-Dn+1}.
  \begin{figure}[ht]
    \centering
    \begin{tikzpicture}
      \foreach \x in {0,0.5,...,4,13,13.5,14}{
        \filldraw (\x,0) circle (1pt);
      }
      \draw[decoration={
        brace,
        mirror,
        raise=3mm
      },
      decorate] (-0.1,0) -- (4.1,0) node[pos=0.5, anchor=north, yshift=-4mm]
      {\( \mathsf{E}^{\le N_{n}}_{\mathcal{D}_{n}} \)-class};
      \foreach \x in {7, 7.5,8.5, 9} {
        \filldraw (\x,0) circle (1pt);
      }
      \draw[decoration={
        brace,
        mirror,
        raise=3mm
      },
      decorate] (6.9,0) -- (9.1,0) node[pos=0.5, anchor=north, yshift=-4mm]
      {\( B(N_{0}, \ldots, N_{n}) \)};
      \draw (5.5,0.5) node {gap \( =   N_{n} +1 \)};
      \draw (11,0.5) node {gap \( \ge N_{n} +1 \)};
    \end{tikzpicture}
    \caption{Construction of \( \mathcal{D}_{n+1} \).  Adding a \( B(N_{0}, \ldots, N_{n}) \) block
      between two \( \mathsf{E}^{\le N_{n}}_{\mathcal{D}_{n}}\)-classes.}
    \label{fig:constructing-Dn+1}
  \end{figure}
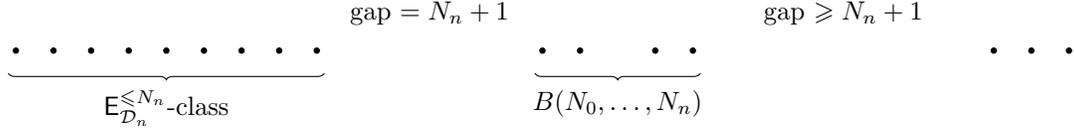

  Cross sections \( \mathcal{D}_{n} \) have the following properties.
  \begin{itemize}
  \item \( \rgap[\mathcal{D}_{n}](x) > M_{0} \) for all \( x \in \mathcal{D}_{0} \).
  \item \( \mathcal{D}_{n} \subseteq \mathcal{D}_{n+1} \).
  \item \( [x]_{\oer[\le N_{k}]{\mathcal{D}_{n}}} = [x]_{\oer[\le N_{k}]{\mathcal{D}_{n+1}}}  \) for any
    \( x \in \mathcal{D}_{n} \) and \( k \le n \).  This is so, because the new blocks \( B(N_{0},
    \ldots, N_{n}) \) that we added are at least \( N_{n} + 1 \) far from any point in \(
    \mathcal{D}_{n} \), and therefore do not change \( \oer[N_{k}]{} \)-classes for \( k \le n \).
  \item Each block \( B(N_{0}, \ldots, N_{n}) \) added to \( \mathcal{D}_{n} \) constitutes a single
    \( \oer[\le N_{n}]{\mathcal{D}_{n+1}} \)-class.
  \item Each \( \oer[\le N_{n+1}]{\mathcal{D}_{n+1}} \)-class consists of at least two
    \( \oer[\le N_{n}]{\mathcal{D}_{n+1}} \)-classes.  Indeed, consider an equivalence class
    \( [x]_{\oer[\le N_{n+1}]{\mathcal{D}_{n+1}}} \), \( x \in \mathcal{D}_{n+1} \).  If \( x \)
    belongs to some block \( B(N_{0}, \ldots, N_{n}) \) added during the construction of
    \( \mathcal{D}_{n+1} \), then \( [x]_{\oer[\le N_{n+1}]{\mathcal{D}_{n+1}}} \) contains a block
    \( B(N_{0}, \ldots, N_{n}) \), which gives one \( \oer[\le N_{n}]{\mathcal{D}_{n+1}} \)-class, and
    it also contains at least one \( \oer[\le N_{n}]{\mathcal{D}_{n+1}} \)-class corresponding to a
    \( \oer[\le N_{n}]{\mathcal{D}_{n}} \)-class.  The same holds true if
    \( x \in \mathcal{D}_{n+1} \) comes from \( \mathcal{D}_{n} \) and
    \( [x]_{\oer[\le N_{n}]{\mathcal{D}_{n}}} = [x]_{\oer[\le N_{n+1}]{\mathcal{D}_{n}}} \).  Finally,
    if \( x \in \mathcal{D}_{n} \), but
    \( [x]_{\oer[\le N_{n}]{\mathcal{D}_{n}}} \ne [x]_{\oer[\le N_{n+1}]{\mathcal{D}_{n}}} \), then
    already \( [x]_{\oer[\le N_{n+1}]{\mathcal{D}_{n}}} \) consists of at least two
    \( {\oer[\le N_{n}]{\mathcal{D}_{n}}} \)-classes, and therefore so does
    \( [x]_{\oer[\le N_{n+1}]{\mathcal{D}_{n+1}}} \), since \( \oer[\le N_{n+1}]{\mathcal{D}_{n+1}} \)
    restricted onto \( \mathcal{D}_{n} \) is coarser than \( \oer[\le N_{n+1}]{\mathcal{D}_{n}} \).
  \end{itemize}

  One concludes that any \( \oer[\le N_{k+1}]{\mathcal{D}_{n+1}} \)-class consists of at least two
  \( \oer[\le N_{k}]{\mathcal{D}_{n+1}} \)-classes for \( k \le n \).  Note also that all
  \( \mathcal{D}_{n} \) are sparse cross sections.  We claim that
  \( \mathcal{C} = \bigcup_{n} \mathcal{D}_{n} \) is the desired cross section.  To begin with, it
  is sparse, since blocks \( B(N_{0}, \ldots, N_{n}) \) have gaps of size at least \( N_{n-1} \).
  Finally, any \( \oer[\le N_{n+1}]{\mathcal{C}} \)-class consists of two or more
  \( \oer[\le N_{n}]{\mathcal{C}} \)-classes.  Indeed, take \( x \in \mathcal{C} \) and let
  \( m \ge n+1 \) be such that \( x \in \mathcal{D}_{m} \).  We know that
  \( [x]_{\oer[\le N_{n+1}]{\mathcal{D}_{m}}} \) consists of at least two
  \( \oer[\le N_{n}]{\mathcal{D}_{m}} \)-classes.  Since
  \( [x]_{\oer[\le N_{n+1}]{\mathcal{D}_{m}}} = [x]_{\oer[\le N_{n+1}]{\mathcal{D}_{m'}}} \) for any
  \( m' \ge m \), it follows that
  \( [x]_{\oer[\le N_{n+1}]{\mathcal{D}_{m}}} = [x]_{\oer[\le N_{n+1}]{\mathcal{C}}} \), and
  therefore \( \mathcal{C} \) satisfies the conclusion of the lemma.
\end{proof}

Recall that for a real \( x \in \mathbb{R} \) and \( \epsilon > 0 \) we let the
\( \epsilon \)-neighborhood \( (x-\epsilon, x +\epsilon) \) of \( x \) to be denoted by
\( \mathcal{U}_{\epsilon}(x) \).

\begin{lemma}
  \label{lem:at-least-two-classes-Kn}
  Let \( \mff \) be a sparse flow on a standard Borel space.  Given an \( \epsilon > 0 \) and a
  sequence of reals \( (K_{n})_{n=0}^{\infty} \), \( K_{n+1} \ge K_{n} + 2\epsilon \), there exists
  a sparse cross section \( \mathcal{C} \) such that
  \begin{enumerate}[(i)]
  \item\label{item:lessK0-is-trivial} \( \oer[\le K_{0}]{\mathcal{C}} \) is the trivial equivalence
    relation: \( x \oer[\le K_{0}]{\mathcal{C}} y \) if and only if \( x= y \).  In other words \(
    \rgap[\mathcal{C}](x) > K_{0} \) for all \( x \in \mathcal{C} \).
  \item\label{item:class-consists-of-two-subclasses} Any \( \oer[\le K_{n+1}]{\mathcal{C}} \)-class
    consists of at least two \( \oer[\le K_{n}]{\mathcal{C}} \)-classes.
  \item\label{item:gaps-are-close-to-average} Gaps between adjacent \( \oer[\le K_{n}]{\mathcal{C}}
    \)-classes within a \( \oer[\le K_{n+1}]{\mathcal{C}} \)-class are in
    \( \mathcal{U}_{\epsilon}\bigl( \frac{K_{n} + K_{n+1}}{2}\bigr) \):
    \[ \textrm{if } \rgap[\mathcal{C}](x) \in (K_{n}, K_{n+1}], \textrm{ then }
    \rgap[\mathcal{C}](x) \in \mathcal{U}_{\epsilon}(\frac{K_{n}+K_{n+1}}{2}). \]
  \end{enumerate}
\end{lemma}
\begin{proof}
  Let \( M_{n} \) be so large that into any gap of length at least \( M_{n} \) once can insert
  blocks \( B(K_{0}, \ldots, K_{n}) \) in such a way that the distance between two adjacent blocks
  is \( \epsilon \)-close to \( \frac{K_{n+1} + K_{n}}{2} \), see Figure
  \ref{fig:inserting-BK-blocks}.  In other words, \( M_{n} \) is so large that one can find a
  Borel function \( \xi_{n} : [M_{n}, \infty) \to \rfin \) such that \( \xi_{n}(x) \)
  corresponds to the set of left endpoint of blocks \( B(K_{0}, \ldots, K_{n}) \) in Figure
  \ref{fig:inserting-BK-blocks}, i.e., if \( \xi_{n}(x) = \{y_{1}, \ldots, y_{m}\} \) with \( y_{1}
  < y_{2} < \ldots < y_{m} \), then
  \begin{itemize}
  \item \( y_{1} \in \mathcal{U}_{\epsilon}(\frac{K_{n} + K_{n+1}}{2}) \);
  \item \( y_{k+1} - y_{k} - b(K_{0}, \ldots, K_{n}) \in \mathcal{U}_{\epsilon}(\frac{K_{n} +
      K_{n+1}}{2}) \) for all \( k < m \);
  \item  \( x - y_{m} - b(K_{0}, \ldots, K_{n}) \in \mathcal{U}_{\epsilon}(\frac{K_{n} +
      K_{n+1}}{2}) \).
  \end{itemize}

  \newsavebox{\boxBlock}
  \savebox{\boxBlock}{\vbox{\hbox to 4cm{\hfil block \hfil}%
\hbox to 4cm{\hfil\( B(K_{0}, \ldots, K_{n}) \)\hfil}}}

  \begin{figure}[ht]
    \centering
    \begin{tikzpicture}
      \filldraw (0,0) circle (1.3pt);
      \draw (0,-0.3) node {\( x \)};
      \draw (9.9,-0.3) node {\( \phi_{\mathcal{D}}(x) \)};
      \filldraw (9.9,0) circle (1.3pt);
      \foreach \x in {2, 2.4, 3.3, 3.7, 6, 6.4, 7.3, 7.7}{
        \filldraw (\x,0) circle (1pt);
      }
      \draw[decoration={
        brace,
        mirror,
        raise=2mm
      },decorate] (1.95,0) -- (3.75,0) node[pos=0.5, anchor=north, yshift=-4mm]
      {\usebox{\boxBlock}};
      \draw[decoration={
        brace,
        mirror,
        raise=2mm
      },decorate] (5.95,0) -- (7.75,0) node[pos=0.5, anchor=north, yshift=-4mm]
      {\usebox{\boxBlock}};
      \draw (4.95,1.5) node {Gaps of size in \(
        \mathcal{U}_{\epsilon}\bigl( \frac{K_{n}+K_{n+1}}{2} \bigr) \)};
      \draw[decoration={
        brace,
        raise=2mm},decorate] (0.05,0) -- (1.95,0);
      \draw[decoration={
        brace,
        raise=2mm},decorate] (3.75,0) -- (5.95,0);
      \draw[decoration={
        brace,
        raise=2mm},decorate] (7.75,0) -- (9.85,0);
      \node[anchor=east] at (1.1,0.4) (left) {};
      \node[anchor=west] at (2.8,1.5) (gapl) {};
      \node[anchor=west] at (8.7,0.4) (right) {};
      \node[anchor=east] at (7.1,1.5) (gapr) {};
      \draw[very thin,->] (4.85,1.2) -- (4.85,0.55);
      \draw[very thin,->] (3.85,1.2) -- (1.85,0.55);
      \draw[very thin,->] (5.85,1.2) -- (7.85,0.55);
    \end{tikzpicture}
    \caption{Inserting blocks \( B(K_{0}, \ldots, K_{n}) \) into the  interval \( [0,x] \).}
    \label{fig:inserting-BK-blocks}
  \end{figure}
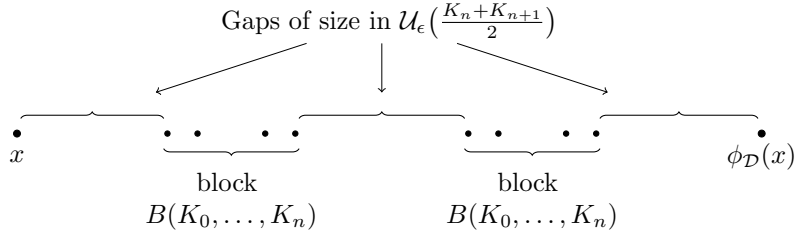

  A moment of thought will convince the reader that such an \( M_{n} \) exists.

  By Lemma \ref{lem:arbitrarily-large-sparse-two-classes}, we may find a sparse cross section
  \( \mathcal{D} \), \( \rgap[\mathcal{D}](x) > M_{0} \), and an increasing sequence
  \( (N_{n})_{n=0}^{\infty} \) such that \( N_{n} \ge M_{n} \), \( N_{0} = M_{0} \), and each
  \( \oer[\le N_{n+1}]{\mathcal{D}} \)-class has at least two
  \( \oer[\le N_{n}]{\mathcal{D}} \)-classes.  The cross section \( \mathcal{C} \) is defined by
  \[ \mathcal{C} = \mathcal{D} \cup \Bigl\{\, x + r + B(K_{0}, \ldots, K_{n}) \Bigm| x \in
  \mathcal{D},\ \rgap[\mathcal{D}](x) \in (N_{n}, N_{n+1}],\ r \in
  \xi_{n}\bigl(\rgap[\mathcal{D}](x)\bigr) \,\Bigr\}. \]
  For \( x \in \mathcal{C} \) we need to show that \( [x]_{\oer[\le K_{n+1}]{\mathcal{C}}} \)
  consists of at least two \( \oer[\le K_{n}]{\mathcal{C}} \)-classes.  There are two cases.

  \textbf{Case I:} \( [x]_{\oer[\le K_{n+1}]{\mathcal{C}}} \cap \mathcal{D} \ne \es \).  Let
  \( y_{1} = \min \bigl([x]_{\oer[\le K_{n+1}]{\mathcal{C}}} \cap \mathcal{D} \bigr) \) and \( y_{2} =
  \max \bigl([x]_{\oer[\le K_{n+1}]{\mathcal{C}}} \cap \mathcal{D} \bigr)\). Notice that:
  \begin{itemize}
  \item \( \rgap[\mathcal{D}](y_{2}) > N_{n+1} \).  Indeed, if \( \rgap[\mathcal{D}](y_{2}) \le
    N_{n+1} \), then by construction \( y_{2}\, \oer[\le K_{n+1}]{\mathcal{C}}\, \phi_{\mathcal{D}}(y_{2})
    \), contradicting the choice of \( y_{2} \).  For a similar reason one also has
  \item \( \rgap[\mathcal{D}]\bigl(\phi_{\mathcal{D}}^{-1}(y_{1})\bigr) > N_{n+1} \).
  \item \( y_{1} \oer[\le N_{n+1}]{\mathcal{D}} y_{2} \).  Suppose not.  Let \( y_{1} \le z < y_{2} \)
    be such that \( \rgap[\mathcal{D}](z) > N_{n+1} \).  By construction, 
    \[ \rgap[\mathcal{C}](z) \in \mathcal{U}_{\epsilon}(\frac{K_{n+1}+K_{n+2}}{2}). \]
    Since \( K_{n+2} \ge K_{n+1} + 2\epsilon \),
    \( \rgap[\mathcal{C}](z) > \frac{K_{n+1} + K_{n+1}+2\epsilon}{2} - \epsilon = K_{n+1} \), whence
    \( \neg \bigl( y_{1}\, \oer[\le K_{n+1}]{\mathcal{C}}\, y_{2} \bigr) \), contradicting the choice of
    \( y_{1} \) and \( y_{2} \).
  \end{itemize}

  The items above prove that 
  \[ [x]_{\oer[\le K_{n+1}]{\mathcal{C}}} \cap \mathcal{D} = [x]_{\oer[\le
    N_{n+1}]{\mathcal{D}}} \textrm{ for all } n \in \mathbb{N}. \]
  By the choice of \( \mathcal{D} \), there are some
  \( x_{1}, x_{2} \in [x]_{\oer[\le N_{n+1}]{\mathcal{D}}} \) such that
  \( [x_{i}]_{\oer[\le N_{n}]{\mathcal{D}}} \subseteq [x]_{\oer[\le N_{n+1}]{\mathcal{D}}} \) and
  \( \neg \bigl( x_{1} \oer[\le N_{n}]{\mathcal{D}} x_{2}\bigr) \).  But
  \( [x_{i}]_{\oer[\le N_{n}]{\mathcal{D}}} = [x_{i}]_{\oer[\le K_{n}]{\mathcal{C}}} \cap
  \mathcal{D} \), therefore \( [x_{i}]_{\oer[\le K_{n}]{\mathcal{C}}} \), \( i=1,2 \), are two
  distinct \( \oer[\le K_{n}]{\mathcal{C}} \)-classes inside
  \( [x]_{\oer[\le K_{n+1}]{\mathcal{C}}} \).

  \textbf{Case II:} \( [x]_{\oer[\le K_{n+1}]{\mathcal{C}}} \cap \mathcal{D} = \es \).  In this case
  \( x \) belongs to some block \( B(K_{0}, \ldots, K_{m}) \) for some \( m \ge n+1 \) (if
  \( m < n+1 \), \( [x]_{\oer[\le K_{n+1}]{\mathcal{C}}} \cap \mathcal{D} \ne \es \)).  Therefore
  \( [x]_{\oer[\le K_{n+1}]{\mathcal{C}}} \) contains at least two \( \oer[\le K_{n}]{\mathcal{C}}
  \) classes by properties of \( B(K_{0}, \ldots, K_{m}) \).

  To see \eqref{item:gaps-are-close-to-average}  we need to show that for any \( x \in \mathcal{C} \),
  \( \rgap[\mathcal{C}](x) \in \mathcal{U}_{\epsilon}(\frac{K_{n} + K_{n+1}}{2}) \) for some \( n
  \).  If there is a block \( B(K_{0}, \ldots, K_{m}) \) such that both \( x \) and \(
  \phi_{\mathcal{C}}(x) \) belong to this block, then this is true by the properties of such a
  block.  Otherwise \( \rgap[\mathcal{C}](x) \) corresponds to a gap between blocks, or
  between a point from \( \mathcal{D} \) and a block.  In both cases it is in \(
  \mathcal{U}_{\epsilon}(\frac{K_{n} + K_{n+1}}{2}) \)  for some \( n \) by the choice of \( \xi_{n} \).
\end{proof}

\section{Propagation of freedom}
\label{sec:propagation-freedom}

{\it For the rest of the paper we fix two rationally independent reals \( \alpha, \beta \in
\mathbb{R}^{>0} \), \( \alpha < \beta \), and a real \( \rho \in (0,1) \).}

\subsection{Tileable reals}
\label{sec:tileable-reals}

\begin{definition}
  \label{def:tileable}
  A positive real number \( x \in \mathbb{R}^{>0} \) is said to be \emph{tileable} if there are
  \( p, q \in \mathbb{N} \) such that \( x = p\alpha + q\beta \).  The term tileable signifies that
  an interval of length \( x \) can be tiled by intervals of length \( \alpha \) and \( \beta \).
  The set of tileable reals is denoted by \( \tileable \).  Note that \( \tileable \) is just the
  subsemigroup of \( \mathbb{R} \) generated by \( \alpha \) and \( \beta \).
\end{definition}

\begin{definition}
  \label{def:asymptotic-density}
  Let \( I \subseteq \mathbb{R}^{\ge 0} \) be a finite or infinite interval.  Given an
  \( \epsilon > 0\), we say that a set \( S \subseteq \mathbb{R} \) is \emph{\( \epsilon \)-dense}
  in \( I \) if it intersects any subinterval of \( I \) of length \( \epsilon \):
  \( S \cap \mathcal{U}_{\epsilon/2}(x) \ne \es \) for any \( x \in I \) such that
  \( \mathcal{U}_{\epsilon/2}(x) \subseteq I \).

  A set \( S \subseteq \mathbb{R}^{\ge 0} \) is \emph{asymptotically dense} in \( \mathbb{R} \) if
  for any \( \epsilon > 0 \) there exists \( N \in \mathbb{N} \) such that \( S \) is \( \epsilon
  \)-dense in \( [N,\infty) \).
\end{definition}

\begin{remark}
  \label{rem:unsual-eps-density}
  This definition of \( \epsilon \)-density is equivalent up to a multiple of \( \epsilon \) to an
  arguably more familiar 
  \[ \forall x \in I\ \exists s \in S\ |x - s| < \epsilon. \]
\end{remark}

Given a tileable \( x \in \tileable \), \( x = p\alpha + q\beta \), we define the
\emph{\( \alpha \)-frequency} of \( x \), denoted by \( \freqa(x) \), to be the quotient
\( p/(p+q) \).

Our first lemma quantifies the following rather obvious fact: if \( x \) and \( y  \) are tileable
reals, and if \( x \) is minuscule compared to \( y \), then the \( \alpha \)-frequency of \( x+y \)
is very close to the \( \alpha \)-frequency of \( y \).

\begin{lemma}
  \label{lem:frequency-length-estimate}
  Let \( x, y \in \tileable \) be two tileable reals and let \( \epsilon > 0 \).  If
  \(\displaystyle \frac{x}{y} < \frac{\alpha \epsilon}{2\beta} \), then
  \( \bigl| \freqa(x + y) - \freqa(y)\bigr| < \epsilon \).
\end{lemma}

\begin{proof}
  Let \( x = p_{x} \alpha + q_{x} \beta \) and \( y = p_{y}\alpha + q_{y}\beta \).  Recall that \(
  \alpha < \beta \) and therefore
  \[ \frac{\alpha \epsilon}{2\beta} > \frac{p_{x} \alpha + q_{x} \beta}{p_{y} \alpha + q_{y}
    \beta} \ge \frac{\alpha}{\beta} \cdot \frac{p_{x} + q_{x}}{p_{y} + q_{y}} \implies \frac{p_{x} +
    q_{x}}{p_{y} + q_{y}} < \frac{\epsilon}{2}. \]
  Therefore also
  \[ \frac{p_{x}}{p_{x} + q_{x} + p_{y} + q_{y}} \le \frac{p_{x} + q_{x}}{p_{x} + q_{x} + p_{y} +
    q_{y}} \le \frac{p_{x} + q_{x}}{p_{y} + q_{y}} < \frac{\epsilon}{2}. \]
  Since
  \[ \freqa(x + y) = \frac{p_{x}}{p_{x} + q_{x} + p_{y} + q_{y}} + \frac{p_{y}}{p_{x} + q_{x} +
    p_{y} + q_{y}},\]
  we get
  \begin{displaymath}
  \begin{aligned}
    \bigl| \freqa(y) - \freqa(x +y) \bigr| &< \Bigl| \frac{p_{y}}{p_{y} + q_{y}} -
    \frac{p_{y}}{p_{x} + q_{x} + p_{y} + q_{y}} \Bigr| + \epsilon/2 \\&=
    \frac{p_{y}}{p_{y} + q_{y}} \cdot \frac{p_{x} + q_{x}}{p_{x} + q_{x} + p_{y} + q_{y}} +
    \frac{\epsilon}{2}\\
    &\le \frac{p_{x} + q_{x}}{p_{y} + q_{y}} + \frac{\epsilon}{2} < \epsilon.
  \end{aligned}
\end{displaymath}
The lemma follows.
\end{proof}

Given an \( \{\alpha,\beta\} \)-regular cross section \( \mathcal{C} \), we shall be interested in
the frequency of occurrence of \( \alpha \)-points within orbits.  More precisely, we
shall consider sums of the form
\( \frac{1}{n}\sum_{k=0}^{n-1}\chi_{\mathcal{C}_{\alpha}}\bigl(\phi_{\mathcal{C}}^{k}(x)\bigr) \),
where \( x \in \mathcal{C} \) and \( \chi_{\mathcal{C}_{\alpha}} \) is the characteristic function of
\( \mathcal{C}_{\alpha} \).  Note that
\[ \frac{1}{n}\sum_{k=0}^{n-1}\chi_{\mathcal{C}_{\alpha}}\bigl(\phi_{\mathcal{C}}^{k}(x)\bigr) =
\freqa\Bigl(\dist\bigl(x, \phi_{\mathcal{C}}^{n}(x)\bigr)\Bigr). \]
Our next lemma gives a criterion of uniform convergence of this sums to \( \rho \in (0,1) \).

\begin{lemma}
  \label{lem:equivalent-formulation-of-regularity}
  Let \( \mathcal{C} \) be an \( \{\alpha,\beta\} \)-regular cross section.  The following are
  equivalent.
  \begin{enumerate}[(i)]
  \item\label{item:freq-uniformly-rho} For any \( \eta > 0 \) there is \( N(\eta) \) such that
    \[ \bigl| \rho - \frac{1}{n}\sum_{k=0}^{n-1} \chi_{\mathcal{C}_{\alpha}}\bigl(\phi^{k}(x)\bigr)
    \bigr| < \eta \]
    for all \( x \in \mathcal{C} \) and all \( n \ge N(\eta) \).
  \item\label{item:sub-cross-rho-gaps} For any \( \eta > 0 \) there is a bounded sub cross section
    \( \mathcal{D}_{\eta} \subseteq \mathcal{C} \) with \( \alpha \)-frequencies of gaps \( \eta
    \)-close to \( \rho \):  there exists
    \( M(\eta) \in \mathbb{R} \) such that
    \[ \rgap[\mathcal{D}_{\eta}](x) \le M(\eta) \textrm{ and }\bigl|\rho -
    \freqa\bigl(\rgap[\mathcal{D}_{\eta}](x)\bigr) \bigr| < \eta \quad \textrm{holds for all } x \in
    \mathcal{D}_{\eta}. \]
  \end{enumerate}
\end{lemma}

\begin{proof}
  It is clear that \ref{item:freq-uniformly-rho} \( \implies \) \ref{item:sub-cross-rho-gaps}, since
  it is enough to take for \( \mathcal{D}_{\eta} \) a bounded sub cross section with at least
  \( N(\eta) \)-many points from \( \mathcal{C} \) between any two adjacent points in
  \( \mathcal{D}_{\eta} \) (see Proposition \ref{prop:marker-lemma-Z-actions}).  For instance, one
  can take \( M(\eta) = \beta \bigl( N(\eta) + 2 \bigr) \).

  We need to show \ref{item:sub-cross-rho-gaps} \( \implies \) \ref{item:freq-uniformly-rho}.  Fix
  \( \eta > 0 \) and consider \( \mathcal{D}_{\eta/2} \) and \( M(\eta/2) \) guaranteed by
  \ref{item:sub-cross-rho-gaps}.  In view of Lemma \ref{lem:frequency-length-estimate}, there is
  \( \widetilde{N} \in \mathbb{R} \) such that \( |\freqa(z+x) - \freqa(z)| < \eta/2 \) for all
  tileable \( z \ge \widetilde{N} - 2M(\eta/2) \) and all \( x \le 2M(\eta/2) \),
  \( x \in \tileable \).  We claim that \( N(\eta) = \widetilde{N}/\alpha \) works.  Pick
  \( x \in \mathcal{C} \) and some natural \( n \ge N(\eta) \).  Consider a block of
  \( n+1 \)-points \( x, \phi_{\mathcal{C}}(x), \ldots, \phi_{\mathcal{C}}^{n}(x) \).  Let
  \( k_{1} < \cdots < k_{m} \) be all the indices for which
  \( \phi^{k_{i}}_{\mathcal{C}}(x) \in \mathcal{D}_{\eta/2} \):
  \[ \{ \phi^{i}_{\mathcal{C}}(x)\}_{i=0}^{n} \cap \mathcal{D}_{\eta/2} =
  \{\phi^{k_{i}}_{\mathcal{C}}(x)\}_{i=1}^{m}. \] Set \( x_{i} = \phi_{\mathcal{C}}^{k_{i}}(x) \).

  \begin{figure}[ht]
    \centering
    \begin{tikzpicture}
      \foreach \x in {0, 12}{
        \filldraw (\x,0) circle (2pt);
      }
      \foreach \x in {-0.6, 12.4}{
        \draw (\x,0) circle (1.5pt);
      }
      \foreach \x in {1.8,3.8,6,8.2,10}{
        \filldraw (\x,0) circle (1.5pt);
      }
      \foreach \x in {0.0, 0.2,...,12}{
        \filldraw (\x,0) circle (0.5pt);
      }
      \draw (0,-0.3) node {\( x \)};
      \draw (12,-0.3) node {\( \phi_{\mathcal{C}}^{n}(x) \)};
      \draw (2.1,0.3) node {\( x_{1} \)};
      \draw (3.8,0.3) node {\( x_{2} \)};
      \draw (6,0.3) node {\( \cdots \)};
      \draw (8.2,0.3) node {\( x_{m-1} \)};
      \draw (9.7,0.3) node {\( x_{m} \)};
      \draw[decoration={
        brace,
        mirror,
        raise=3mm
      },decorate] (1.85,0) -- (3.75,0) node[pos=0.5, anchor=north] {};
            \draw[decoration={
        brace,
        mirror,
        raise=3mm
      },decorate] (3.85,0) -- (5.95,0) node[pos=0.5, anchor=north] {};
            \draw[decoration={
        brace,
        mirror,
        raise=3mm
      },decorate] (6.05,0) -- (8.15,0) node[pos=0.5, anchor=north] {};
      \draw[decoration={
        brace,
        mirror,
        raise=3mm
      },decorate] (8.25,0) -- (9.95,0) node[pos=0.5, anchor=north] {};
      \draw (6,-1.3) node[text width=9cm]{%
          Blocks between \( \mathcal{D}_{\eta/2} \)-points.  
          Each block has length at most \( M(\eta/2) \) 
          and \( \alpha \)-frequency \( \eta/2 \)-close to \( \rho \).
      };
      \draw[decoration={
        brace,
        raise=3mm
      },decorate] (0.1,0) -- (1.7,0) node[pos=0.5, anchor=north,yshift=9mm] {error term};
      \draw[decoration={
        brace,
        raise=3mm
      },decorate] (10.1,0) -- (11.9,0) node[pos=0.5, anchor=north,yshift=9mm] {error term};
    \end{tikzpicture}
    \caption{\( \mathcal{D}_{\eta/2} \)-blocks within a segment of \( N(\eta) \)-many points.}
    \label{fig:unifromly-reg-gaps}
  \end{figure}

  Since \( \rgap[\mathcal{D}_{\eta/2}] \le
  M(\eta/2) \), it follows that \( \dist(x,x_{1}) \le M(\eta/2) \) and \(
  \dist\bigl(x_{m}, \phi^{n}_{\mathcal{C}}(x) \bigr) \le M(\eta/2) \).   Let \( z =
  \dist(x_{1}, x_{m}) \).  One has
  \[ z \ge \dist\bigl(x, \phi^{n}_{\mathcal{C}}(x)\bigr) - 2M(\eta/2) \ge n\alpha - 2M(\eta/2)
  \ge \widetilde{N} - 2M(\eta/2). \]
  Let \( y = \dist(x,x_{1}) + \dist\bigl(x_{m}, \phi^{n}_{\mathcal{C}}(x)\bigr) \) be the
  error term.  By the choice of \( \widetilde{N} \) we conclude that
  \[ |\freqa(z + y) - \freqa(z)| < \eta/2. \]
  Since \( |\freqa(z)-\rho| < \eta/2 \), it follows that
  \[ \bigl| \freqa\bigl( \dist\bigl( x, \phi^{n}_{\mathcal{C}}(x) \bigr) \bigr) -\rho \bigr| <
  \eta/2 + \eta/2 = \eta.\qedhere \]
\end{proof}

\subsection{Propagation of Freedom}
\label{sec:propagation-freedom-subsec}

Let \( (d_{k})_{k=1}^{n} \) be a sequence of positive reals, and let \( \epsilon > 0 \).  Suppose
that for each \( 1 \le k \le n \) we are given a set
\( R_{k} \subseteq \mathcal{U}_{\epsilon}(d_{k}) \).  Define
\[ \mathcal{A}_{n} = \mathcal{A}_{n}\bigl( \epsilon, (d_{k})_{k=1}^{n}, (R_{k})_{k=1}^{n} \bigr) \]
to consists of all \( x \in \mathcal{U}_{\epsilon}\bigl( \sum_{k=1}^{n} d_{k} \bigr) \) for which
there exist \( y_{k} \in R_{k} \), \( k \le n \), such that
\( \bigl| \sum_{k=1}^{r}(d_{k}-y_{k}) \bigr| < \epsilon \) for all \( r \le n \), and
\( x = \sum_{k=1}^{n} y_{k} \).  Properties of sets \( \mathcal{A}_{n} \) are in the core of this
section, so let us try to give a more geometrical explanation of the definition.  Imagine a family of
points \( z_{0}, \ldots, z_{n} \) with \( \dist(z_{k-1},z_{k}) = d_{k} \), see Figure
\ref{fig:def-of-An}.
\begin{figure}[ht]
  \centering
  \begin{tikzpicture}
    \foreach \p in {0.0, 2.6, 4.6, 8.1, 10.5} {
      \filldraw (\p,0) circle (1pt);
    }
    \draw(0,-0.3) node {\( z_{0} \)};
    \draw(2.6,-0.3) node {\( z_{1} \)};
    \draw(4.6,-0.3) node {\( z_{2} \)};
    \draw(8.1,-0.3) node {\( z_{n-1} \)};
    \draw(10.5,-0.3) node {\( z_{n} \)};
    \draw(6.35,-0) node {\( \cdots \)};
    \draw[decoration={
      brace,
      raise=2mm},
    decorate] (0.1, 0) -- (2.5,0) node[pos=0.5, anchor=north, yshift=9mm] {\( d_{1} \)};
    \draw[decoration={
      brace,
      raise=2mm},
    decorate] (2.7, 0) -- (4.5,0) node[pos=0.5, anchor=north, yshift=9mm] {\( d_{2} \)};
    \draw[decoration={
      brace,
      raise=2mm},
    decorate] (8.2, 0) -- (10.4,0) node[pos=0.5, anchor=north, yshift=9mm] {\( d_{n} \)};
    \draw (5.2,-1) node {Moving each \( z_{k} \) by at most \( \epsilon \).};
    \foreach \p in {0.0, 2.8, 4.4, 8.3, 10.4} {
      \filldraw (\p,-2) circle (1pt);
    }
    \draw(0,-1.7) node {\( z_{0} \)};
    \draw(2.8,-1.7) node {\( z_{1}' \)};
    \draw(4.4,-1.7) node {\( z_{2}' \)};
    \draw(8.3,-1.7) node {\( z_{n-1}' \)};
    \draw(10.4,-1.7) node {\( z_{n}' \)};
    \draw(6.35,-2) node {\( \cdots \)};
    \draw[decoration={
      brace,
      mirror,
      raise=2mm},
    decorate] (0.1, -2) -- (2.7,-2) node[pos=0.5, anchor=south, yshift=-9mm] {in \( R_{1} \)};
    \draw[decoration={
      brace,
      mirror,
      raise=2mm},
    decorate] (2.9, -2) -- (4.3,-2) node[pos=0.5, anchor=south, yshift=-9mm] {in \( R_{2} \)};
    \draw[decoration={
      brace,
      mirror,
      raise=2mm},
    decorate] (8.4, -2) -- (10.3,-2) node[pos=0.5, anchor=south, yshift=-9mm] {in \( R_{n} \)};
    \foreach \p in {0.0, 2.8, 4.4, 8.3, 10.4} {
      \filldraw (\p,-2) circle (1pt);
    }
  \end{tikzpicture}
  \caption{The set \( \mathcal{A}_{n}\bigl(\epsilon, (d_{k}), (R_{k})\bigr) \) consists of all
    possible \( x = \dist(z_{0}, z_{n}') \).}
  \label{fig:def-of-An}
\end{figure}
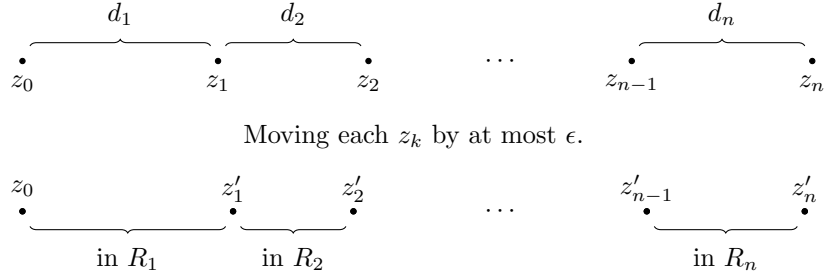
We would like to move each \( z_{k} \) by at most \( \epsilon \) to some new point \( z_{k}' \) so
as to make distances from \( z_{k-1}' \) to \( z_{k}' \) lie in \( R_{k} \).  Let us say that a
distance \( \dist(z_{k-1}',z_{k}') \) is \emph{admissible} if
\( \dist(z_{k-1}', z_{k}') \in R_{k} \).  The leftmost point \( z_{0} \) stays fixed,
\( z_{0}' = z_{0} \).  Given any \( y_{1} \in R_{1} \), we may take \( z_{1}' = z_{0} + y_{1} \).
We pick some \( y_{2} \in R_{2} \).  In general, \( z_{2}' = z_{1}' + y_{2} \) may not work,
since such a \( z_{2}' \) may not be in the \( \epsilon \)-neighborhood of \( z_{2} \) anymore, but
if, for example, \( R_{2} \) contains elements both smaller and larger than
\( \dist(z_{1}, z_{2}) \), then one can always pick \( y_{2} \in R_{2} \) such that
\( z_{2}' = z_{1}' + y_{2} \) is within \( \epsilon \) from \( z_{2} \).  Depending on the sets
\( R_{k} \), there may or may not be any ways to move points in the described way, but when
\( R_{k} \) are diverse enough, there will necessarily be many such arrangements.  We are
specifically concerned with the set of possible positions for the last point \( z_{n}' \).  The set
\( \mathcal{A}_{n} \) is the collection of possible distances from \( z_{0} \) to \( z_{n}' \).

Sets of admissible distances \( R_{k} \) will consist of tileable reals.  All the elements in
\( \mathcal{A}_{n} \), being sums of elements of \( R_{k} \), will therefore also be tileable.  The
main technical ingredient of our argument is quantification of the richness of
\( \mathcal{A}_{n} \), and in particular we expect the cardinality of \( \mathcal{A}_{n} \) to grow
with \( n \) and to pack densely the \( \epsilon \)-neighborhood of \( z_{n} \).  The details are
subtle.  There are some degenerate ways to pick \( R_{k} \) and \( d_{k} \) resulting in the failure
of the above expectations, with sets \( \mathcal{A}_{n} \) not growing with \( n \).  The section
thus concentrates on finding sufficient conditions to avoid such pathologies.  As a matter of fact,
we need to control two aspects of \( \mathcal{A}_{n} \).  First of all, we need to measure the
density of \( \mathcal{A}_{n} \) in \( \mathcal{U}_{\epsilon}(z_{n}) \).  Second, we need
information about the diversity of \( \alpha \)-frequencies of elements in \( \mathcal{A}_{n} \).
We are unable to do both items at once, but fortunately once \( \mathcal{A}_{n} \) is known to be
\( \delta \)-dense it will be easy to increase \( n \) improving \( \alpha \)-frequencies while
keeping \( \delta \)-density because of the following Frequency Boost Lemma.

\begin{lemma}[Frequency Boost]
  \label{lem:frequency-boost}
  For any \( D \in \mathbb{R}^{\ge 0} \) and \( \zeta \in \mathbb{R}^{>0} \) there is
  \( M = M_{\Lem\ref{lem:frequency-boost}}(D, \zeta) \) such that for any \( n \ge M \), any
  \( \eta \in \mathbb{R}^{>0} \), any \( \gamma \in (\rho - \eta, \rho + \eta) \), any
  \( 0 < \epsilon \le 1 \), any family of reals \( (d_{k})_{k=1}^{n} \) satisfying
  \( 2\epsilon + \alpha \le d_{k} \le D \), and any sequence of sets
  \( R_{k} \subseteq \mathcal{U}_{\epsilon}(d_{k}) \cap \tileable \) such that
  \( R_{k} \cap \freqa^{-1}[0, \rho-\eta] \) and \( R_{k} \cap \freqa^{-1}[\rho + \eta, 1] \) are
  \( \epsilon \)-dense in \( \mathcal{U}_{\epsilon}(d_{k}) \),
  there is \( x \in \mathcal{A}_{n} \) such that \( |\freqa(x) - \gamma| < \zeta \).
\end{lemma}

Here is a less formal statement of the lemma.  In the context of Figure \ref{fig:def-of-An}, we want
to move points \( z_{k} \) in such a way so that the distance \( \dist(z_{0},z_{n}') \) will have
\( \alpha \)-frequency \( \zeta \)-close to \( \gamma \), where \( \gamma \) is somewhere in
\( (\rho-\eta, \rho+\eta) \).  Suppose sets \( R_{k} \) allow us to move \( z_{k} \) at our choice
to the left or to the right in a way that, again at our choice, adds a real of
\( \alpha \)-frequency at least \( \rho+\eta \) or at most \( \rho - \eta \).  The lemma claims that
under this assumptions one can always find such an \( x \in \mathcal{A}_{n} \) provided that \( n \)
is sufficiently large in terms of the bound on gaps \( D \) and precision \( \zeta \).  The proof of
is similar to the proof of Riemann Rearrangement Theorem.
\begin{proof}
  We construct \( M \) as \( M = M_{1} + M_{2} \).  Integer \( M_{1} \) will be picked so large that
  after \( M_{1} \)-steps the frequency does not change by more than \( \zeta \), and \( M_{2} \)
  will ensure that we have enough intervals to change the frequency by at least \( \eta - \zeta \).
  We take
  \[ M_{1} > \frac{D + 2}{\zeta} \cdot \frac{2\beta}{\alpha^{2}} \quad \textrm{and} \quad M_{2}
  > \frac{M_{1}D + 2}{\zeta} \cdot \frac{2 \beta}{\alpha^{2}}. \]
  Note that if \( (y_{k})_{k=1}^{n} \) is any sequence such that
  \( \bigl|\sum_{k=1}^{r}(d_{k}-y_{k}) \bigr| < \epsilon \) for all \( r \le n \), and
  \( y_{k} \in \tileable \), \( k \le n \), where \( n \ge M_{1} \), then by Lemma
  \ref{lem:frequency-length-estimate} and the choice of \( M_{1} \), for any \( m \in \mathbb{N} \)
  satisfying \( M_{1} \le m < n \), the contribution of \( y_{m+1} \) to the frequency of
  \( \sum_{k=1}^{m+1} y_{k} \) is small:
  \begin{equation}
    \label{eq:small-freq-change}
    \Bigl| \freqa \Bigl( \sum_{k=1}^{m+1} y_{k} \Bigr) - \freqa\Bigl( \sum_{k=1}^{m} y_{k} \Bigr)
    \Bigr| < \zeta.
  \end{equation}

  We construct the required \( x \in \mathcal{A}_{n} \) as follows.  For \( y_{1} \) pick any
  element of \( R_{1} \).  Suppose \( y_{k} \) has been constructed.  Take
  \( y_{k+1} \in R_{k+1} \) satisfying the following two conditions:
  \begin{itemize}
  \item if \( \sum_{i=1}^{k}y_{i} \le \sum_{i=1}^{k}d_{i} \), then \( y_{k+1} \ge d_{k+1} \); if
    \( \sum_{i=1}^{k}y_{i} > \sum_{i=1}^{k}d_{i} \), then \( y_{k+1} \le d_{k+1} \);
  \item if \( \freqa(\sum_{i=1}^{k}y_{i}) \le \gamma \), then \( \freqa(y_{k+1}) \ge \rho + \eta \); if
    \( \freqa(\sum_{i=1}^{k}y_{i}) > \gamma \), then \( \freqa(y_{k+1}) \le \rho - \eta \).
  \end{itemize}
  The possibility to choose such an \( y_{k+1} \) is guaranteed by the assumption that
  \[ R_{k} \cap \freqa^{-1}[0, \rho-\eta] \textrm{ and } R_{k} \cap \freqa^{-1}[\rho + \eta, 1]
  \textrm{ are \( \epsilon \)-dense in } \mathcal{U}_{\epsilon}(d_{k}). \]
  It is obvious that \( \sum_{k=1}^{r}y_{k} \in \mathcal{A}_{r} \) for all \( r \).  We claim that
  \( \bigl| \freqa\bigl(\sum_{i=1}^{n}y_{i}\bigr) -\gamma \bigr| \le \zeta \) for all \( n \ge M \).

  Consider \( \sum_{i=1}^{M_{1}} y_{i} \).  We have two cases: either
  \( \freqa(\sum_{i=1}^{M_{1}} y_{i}) \le \gamma \), or
  \( \freqa(\sum_{i=1}^{M_{1}} y_{i}) > \gamma \).  Suppose for definiteness that
  \( \freqa(\sum_{i=1}^{M_{1}} y_{i}) \le \gamma \).  If there is \( k \), \( M_{1} \le k < M \),
  such that \( \freqa(\sum_{i=1}^{k} y_{i}) \le \gamma \) but
  \( \freqa(\sum_{i=1}^{k+1} y_{i}) > \gamma \), then the claim follows by the above estimate
  \eqref{eq:small-freq-change}.  We may therefore suppose that for any \( k \le M \) we have
  \( \freqa(\sum_{i=1}^{k} y_{i}) \le \gamma \).  In particular,
  \( \freqa(\sum_{i=1}^{M} y_{i}) \le \gamma \).  By the construction,
  \( \freqa(\sum_{i=M_{1}+1}^{M} y_{i}) \ge \rho + \eta \), therefore in view of Lemma
  \ref{lem:frequency-length-estimate} and the choice of \( M_{2} \), we have
  \( \bigl| \freqa(\sum_{i=1}^{M}y_{i}) - \freqa(\sum_{i=M_{1}+1}^{M}y_{i}) \bigr| < \zeta. \)
  Since
  \[ \freqa\bigl(\sum_{i=1}^{M} y_{i}\bigr) \le \gamma \le \rho + \eta \le \freqa\bigl(\mkern-10mu
  \sum_{i=M_{1}+1}^{M} \mkern-10mu y_{i}\bigr),
  \]
  we conclude that \( \bigl| \freqa(\sum_{i=1}^{M}y_{i}) - \gamma \bigr| \le \zeta \).
\end{proof}

\subsection{Moving points: Constant case.}
\label{sec:moving-points-regul}

In this subsection we study the behaviour of sets
\[ \mathcal{A}_{n}\bigl(\epsilon, (d_{k})_{k=1}^{n}, (R_{k})_{k=1}^{n}\bigr) \]
with constant parameters: \( d_{i} = d_{j} \) and \( R_{i} = R_{j} \) for all \( 1 \le i,j \le n \).
We let \( d = d_{k} \) and \( R = R_{k} \) and
\( \mathcal{A}_{n}\bigl(\epsilon, (d)_{k=1}^{n}, (R)_{k=1}^{n}\bigr) \) will be denoted simply by
\( \mathcal{A}_{n} \).  

\begin{lemma} Sets \( \mathcal{A}_{n} \) have the following additivity properties.
  \label{lem:additive-property}
  \begin{enumerate}[(i)]
  \item\label{item:additivity-one-dim} If \( y_{i} \in R \), \( 1 \le i \le n \), are such that
    \[ \Bigl| nd - \sum_{i=1}^{n} y_{i} \Bigr| < \epsilon, \]
    then \( \sum_{i=1}^{n} y_{i} \in \mathcal{A}_{n} \).
  \item\label{item:additivity-prop} If \( x_{i} \in \mathcal{A}_{n_{i}} \), \( 1 \le i \le k
    \), are such that
    \[ \Bigl| \sum_{i=1}^{k}(x_{i} - n_{i}d) \Bigr| < \epsilon, \]
    then \( \sum_{i=1}^{k} x_{i} \in \mathcal{A}_{\sum_{i=1}^{k}n_{i}}\).
  \item\label{item:inclusion} If \( d \in R \) and \( m \le n \), then
    \( \mathcal{A}_{m} + (n-m)d \subseteq \mathcal{A}_{n} \).
  \end{enumerate}
\end{lemma}

\begin{proof}
  \eqref{item:additivity-one-dim} We need to show that there exists a permutation 
  \[ \pi : \{1, \ldots, n \} \to \{1, \ldots, n\} \]
  such that
  \[ \Bigl| kd - \sum_{i=1}^{k}y_{\pi(i)} \Bigr| < \epsilon\quad \textrm{for all } k \le n. \]
  Set \( \pi(1) = 1 \), and a define \( \pi(k+1) \) inductively as follows.  Let 
  \[ P = \{y_{i} : i \ne \pi(j) \textrm{ for any } 1 \le j \le k \} \]
  to denote the set of \( y_{i} \)'s that we haven't used yet.  If
  \( \sum_{i=1}^{k}y_{\pi(i)} \le kd \), we search for the smallest \( i \in P \) such that
  \( y_{i} \ge d \) and set \( \pi(k+1) = i \); if no such \( i \in P \) exists, we set
  \( \pi(k+1) = \min P \).  Similarly, if \( \sum_{i=1}^{k}y_{\pi(i)} > kd \), we search for the
  smallest \( i \in P \) such that \( y_{i} \le d \) and set \( \pi(k+1) = i \); if no such
  \( i \in P \) exists, we set \( \pi(k+1) = \min P \).  We claim that such a permutation \( \pi \)
  satisfies
  \[ \Bigl| kd - \sum_{i=1}^{k}y_{\pi(i)} \Bigr| < \epsilon\quad \textrm{for all } k \le n. \]
  For \( k = 1 \) the inequality holds trivially.  Suppose it holds for \( k \).  We check it for
  \( k+1 \), and let us assume for definiteness that \( \sum_{i=1}^{k} y_{\pi(i)} \le kd \).  We
  have two cases depending on how \( \pi(k+1) \) was selected.  If \( y_{\pi(k+1)} \ge d \), then
  \[ \Bigl| (k+1)d - \sum_{i=1}^{k+1} y_{\pi(i)} \Bigr| = \Bigl|  d - y_{\pi(k+1)} + kd -
  \sum_{i=1}^{k} y_{\pi(i)} \Bigr| < \epsilon,\]
  because
  \[ - \epsilon < d - y_{\pi(k+1)} \le 0 \quad \textrm{ and } \quad 0 \le kd - \sum_{i=1}^{k}
  y_{\pi(i)} < \epsilon. \]
  So, it remains to consider the case when \( y_{i} < d\) for all \( i \in P \).  In this case
  \[ (k+1)d - \sum_{i=1}^{k+1}y_{\pi(i)} > 0. \]
  Let
  \( P' = P \setminus \{\pi(k+1)\} = P \setminus \min P \).  We have
  \[ \sum_{i=1}^{n} y_{i} = \sum_{i=1}^{n}y_{\pi(i)} = \sum_{i=1}^{k+1} y_{\pi(i)} + \sum_{j \in P'}
  y_{j} \]
  and \( \sum_{j \in P'} y_{j} < |P'| d = (n-k-1)d  \).  If
  \[ (k+1)d - \sum_{i=1}^{k+1}y_{\pi(i)} \ge \epsilon, \]
  then 
  \[ nd - \sum_{i=1}^{n} y_{i} = (k+1)d - \sum_{i=1}^{k+1} y_{\pi(i)} + (n-k-1)d - \sum_{j \in P'}
  y_{j} \ge \epsilon, \]
  contradicting the assumption \( \bigl| nd - \sum_{i=1}^{n}y_{i} \bigr| < \epsilon \), so we
  have to have
  \[ \Bigl| (k+1)d - \sum_{i=1}^{k+1} y_{\pi(i)} \Bigr| < \epsilon \]
  as claimed.
  
  \eqref{item:additivity-prop} Since \( x_{i} \in \mathcal{A}_{n_{i}} \), there are \( y_{j}^{i} \),
  \( 1 \le j \le n_{i} \), such that
  \begin{itemize}
  \item \( y_{j}^{i} \in R \);
  \item \( \bigl| rd - \sum_{j=1}^{r} y_{j}^{i} \bigr| < \epsilon \) for all \( r \le n_{i} \);
  \item \( \sum_{j=1}^{n_{i}} y_{j}^{i} = x_{i} \).
  \end{itemize}
  By \eqref{item:additivity-one-dim}, \( \sum_{i,j} y_{j}^{i} \in \mathcal{A}_{\sum_{i=1}^{k}n_{i}}
  \), and the item follows.

  \eqref{item:inclusion} Follows from item \eqref{item:additivity-prop}, since \( d \in R \) and
  thus \( (n-m)d \in \mathcal{A}_{n-m} \) for \( n > m \).
\end{proof}

Additivity of sets \( \mathcal{A}_{n} \) is a luxury we have in the case of constant parameters.
Given \( x, y \in \mathcal{A}_{m} \) such that \( x < md < y \) we may therefore start forming sums
of \( x \) and \( y \), e.g.,
\[ x + y + y + x + y + \cdots, \]
as long as any initial segment of this sum is \( \epsilon \)-close to the multiple of \( d \).  By
the above lemma, all such sums will lie in \( \mathcal{A}_{n} \) for the relevant \( n \).  The next
lemma quantifies the density of \( \mathcal{A}_{n} \) in \( \mathcal{U}_{\epsilon}(nd) \) that can
be achieved by forming such combinations.  Note that if \( d \in \mathbb{R} \), then \(
\mathcal{A}_{n} + d \subseteq \mathcal{A}_{n+1} \), and therefore the degree of density cannot
decrease: if \( \mathcal{A}_{n} \) is \( \delta \)-dense in \( \mathcal{U}_{\epsilon}(nd) \), then
\( \mathcal{A}_{n+1} \) is \( \delta \)-dense in \( \mathcal{U}_{\epsilon}(nd + d) \).

For two reals \( a, b \in \mathbb{R} \) we let \( \gcd(a,b) \) denote the greatest positive real \(
c \) such that both \( a \) and \( b \) are integer multiples of \( c \).  If \( a \) and \( b \) are
rationally independent, we take \( \gcd(a,b) = 0 \).

\begin{lemma}
  \label{lem:generating-elements-of-An}
  Let \( \epsilon > 0 \), let \( 0 < \delta \le \epsilon \), and let \( x, y \in \mathcal{A}_{m} \),
  \( m \ge 1 \), be given.  Set \( a = x - md \) and \( b = y - md \).  Suppose that \( d \in R \),
  and \( a < 0 < b \).
  There exists \( N = N_{\Lem \ref{lem:generating-elements-of-An}}(R, m, \epsilon, \delta, d, x, y) \)
  such that for all \( n \ge N \)
  \begin{itemize}
  \item if \( \delta > \gcd(a,b) \), then the set \( \mathcal{A}_{n} \) is \( \delta \)-dense in \(
  \mathcal{U}_{\epsilon}(nd) \);
  \item if \( \delta \le \gcd(a,b) \), then the set \( \mathcal{A}_{n} \) is \( \kappa \)-dense in \(
    \mathcal{U}_{\epsilon}(nd) \) for any \( \kappa > \gcd(a,b) \) and
    moreover 
    \[ nd + k\gcd(a,b) \in \mathcal{A}_{n}\ \textrm{ for all integers } k \textrm{ such that } nd +
    k\gcd(a,b) \in \mathcal{U}_{\epsilon}(nd). \]
  \end{itemize}
\end{lemma}

\begin{proof}
  Starting with \( a_{0} = a \) and \( b_{0} = b  \) we define for \( k \ge 1 \)
  \begin{displaymath}
    \label{eq:ab-formulas}
    \begin{aligned}
      a_{k} &= a_{k-1} + l_{k} b_{k-1},\\
      b_{k} &= b_{k-1} + l_{k}' a_{k},\\
    \end{aligned}
  \end{displaymath}
  where \( l_{k} \) is the largest natural number such that \( a_{k-1} + l_{k} b_{k-1} \le 0 \), and
  \( l_{k}' \) is the largest natural satisfying \( b_{k-1} + l_{k}' a_{k} \ge 0 \).  If
  \( a_{k} = 0 \), we take \( l_{k}' = 0 \).  If either \( a_{k} = 0 \) or \( b_{k} = 0 \), we stop
  constructing the sequence.  Here are a few things to note.
  \begin{itemize}
  \item \( a_{k} \le 0 \) and \( b_{k} \ge 0\) for
    all \( k \).
  \item If \( a_{k} = 0 \), but \( b_{k-1} \ne 0 \), then \( b_{k-1} = \gcd(a,b) \).  Similarly, if
    \( b_{k} = 0 \), but \( a_{k} \ne 0 \), then \( a_{k} = -\gcd(a,b) \).
  \item If \( a_{k} \ne 0 \), then \( b_{k} \le 1/2 \cdot b_{k-1} \) and
    \( | a_{k} | \le b_{k-1} \); in particular sequences \( (a_{k}) \) and \( (b_{k}) \) converge to
    \( 0 \) exponentially fast.
  \end{itemize}
  We take \( K = K(\delta,a,b) \) to be the minimal natural such that one of \( a_{K} \), \( b_{K} \) is
  \( 0 \) or
  \( \max\{|a_{K}|,b_{K}\} < \delta \).

  \medskip

  \textbf{Claim.}  There is \( \tilde{N} \in \mathbb{N} \) such that
  \( a_{K} + \tilde{N}d \) and \( b_{K} + \tilde{N}d \in \mathcal{A}_{\tilde{N}} \).

  \textit{Proof of claim.}  Unraveling the formulas for \( a_{k} \) and \( b_{k} \) one checks that
  there are naturals
  \( q, q', p, p' \in \mathbb{N} \)  such that
  \begin{equation}
    \begin{aligned}
      a_{K} &= pa + qb,\\
      b_{K} &= p'a + q'b.\\
    \end{aligned}
  \end{equation}
  From item
  \eqref{item:additivity-prop} of Lemma \ref{lem:additive-property} we know that
  \begin{displaymath}
    \begin{aligned}
      a_{K} + (p+q)md &= p(a + md) + q(b + md) = px
      + qy \in \mathcal{A}_{(p + q)m},\\
      b_{K} + (p'+q')md &= p'x + q'y \in \mathcal{A}_{(p' + q')m}.\\
    \end{aligned}
  \end{displaymath}
  Therefore by Lemma \ref{lem:additive-property}\eqref{item:inclusion} we may take \( \tilde{N}
  = m\cdot\max\{p + q,p'+q'\} \).  \hfill\( \square_{\textrm{claim}} \)

  \medskip

  The proof now splits into three cases.
  
  {\bf Case I:} \( a_{K} < 0 \) and \( b_{K} > 0 \).  In this case \( \delta > \gcd(a,b) \).  Let
  \( N_{a} \) and \( N_{b} \) be such that
  \[ -\epsilon < N_{a}a_{K} \le -\epsilon + a_{K} \textrm{ and } \epsilon - b_{K}
  \le N_{b}b_{K} < \epsilon,  \]
  and set \( N = \max\{N_{a}, N_{b}\} \cdot \tilde{N} \).  Again, by Lemma
  \ref{lem:additive-property} and Claim for any \( n \ge N \)
  \begin{itemize}
  \item \( ka_{K} + nd \in \mathcal{A}_{n} \) for all \( 0 \le k \le N_{a} \);
  \item \( kb_{K} + nd \in \mathcal{A}_{n} \) for all \( 0 \le k \le N_{b} \).
  \end{itemize}
  Therefore \( \mathcal{A}_{n} \) is \( \delta \)-dense in \( \mathcal{U}_{\epsilon}(nd) \), since
  \( |a_{K}| < \delta \) and \( b_{K} < \delta \), see Figure \ref{fig:An-dense-item-one}.
  \begin{figure}[ht]
    \centering
    \begin{tikzpicture}
      \draw (-5.8,0) -- (4.8,0);
      \draw (-6.3,0) node {\( \ldots \)};
      \draw (5.3,0) node {\( \ldots \)};
      \foreach \x in {-5.4,-3.6,-1.8,1.5,3,4.5}{
        \draw (\x,0.1) -- (\x,-0.1);
      }
      \foreach \xind/\x in {3/-5.4,2/-3.6,/-1.8}{
        \draw (\x,-0.3) node {\( \scriptstyle nd+\xind a_{K} \)};
      }
      \filldraw (0,0) circle (1.5pt);
      \draw (0,-0.3) node {\( \scriptstyle nd \)};
      \foreach \xind/\x in {/1.5,2/3,3/4.5}{
        \draw (\x,-0.3) node {\( \scriptstyle nd+\xind b_{K} \)};
      }
      \draw[
      decoration={
        brace,
        raise=0.2cm
      },
      decorate
      ] (1.5,0) -- (3,0) node [pos=0.5,anchor=south,yshift=0.3cm]{\( < \delta \)};
            \draw[
      decoration={
        brace,
        raise=0.2cm
      },
      decorate
      ] (-3.6,0) -- (-1.8,0) node [pos=0.5,anchor=south,yshift=0.3cm]{\( < \delta \)};
    \end{tikzpicture}
    \caption{Some elements of \( \mathcal{A}_n \) witnessing its \( \delta \)-density in
      \( \mathcal{U}_{\epsilon}(nd) \).}
    \label{fig:An-dense-item-one}
  \end{figure}
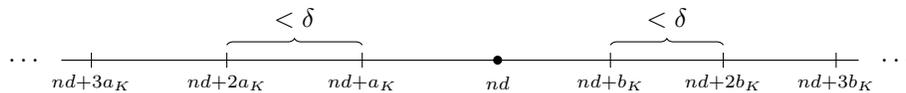

  {\bf Case II:} \( a_{K} = 0 \).  Note that by the choice of \( K \), \( b_{K-1} \ne 0 \).  We let
  \( c = b_{K-1} = \gcd(a,b) \).  Since \( b_{K-1} = b_{K} = c \), we have
  \( c + \tilde{N}d \in \mathcal{A}_{\tilde{N}} \), where \( \tilde{N} \) is given by Claim.  Thus
  also \( c + l_{K}\tilde{N}d \in \mathcal{A}_{l_{K} \tilde{N}} \).  By increasing \( \tilde{N} \)
  if necessary, we also have \( a_{K-1} + \tilde{N}d \in \mathcal{A}_{\tilde{N}} \) and
  \( b_{K-1} + \tilde{N}d \in \mathcal{A}_{\tilde{N}}\), so by Lemma \ref{lem:additive-property}
  \[ -c + l_{K}\tilde{N}d = a_{K} - b_{K-1} + l_{K}\tilde{N}d = a_{K-1} + (l_{K}-1)b_{K-1} + l_{K}
  \tilde{N} d \in \mathcal{A}_{l_{K}\tilde{N}}. \]
  Therefore both \( l_{K}\tilde{N}d + c \) and \( l_{K}\tilde{N}d - c \) are elements of
  \( \mathcal{A}_{l_{K}\tilde{N}} \).

  Finally, let \( N_{c} \) be such that
  \( \epsilon - c \le N_{c} c < \epsilon \) and take \( N = N_{c}l_{K} \tilde{N} \).  Then
  \[ nd + kc \in \mathcal{A}_{n} \textrm{ for all } n \ge N_{c} l_{K} \tilde{N} \textrm{ and all }
  -N_{c} \le k \le N_{c}. \]
  The natural \( N = N_{c} l_{K} \tilde{N} \) satisfies the conclusion of the lemma.

  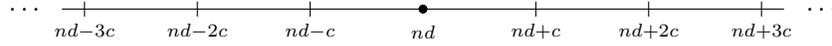
\begin{figure}[ht]
    \centering
    \begin{tikzpicture}
      \draw (-4.8,0) -- (4.8,0);
      \draw (-5.3,0) node {\( \ldots \)};
      \draw (5.3,0) node {\( \ldots \)};
      \foreach \x in {-4.5,-3.0,-1.5,1.5,3,4.5}{
        \draw (\x,0.1) -- (\x,-0.1);
      }
      \foreach \xind/\x in {3/-4.5,2/-3.0,/-1.5}{
        \draw (\x,-0.3) node {\( \scriptstyle nd-\xind c \)};
      }
      \filldraw (0,0) circle (1.5pt);
      \draw (0,-0.3) node {\( \scriptstyle nd \)};
      \foreach \xind/\x in {/1.5,2/3,3/4.5}{
        \draw (\x,-0.3) node {\( \scriptstyle nd+\xind c \)};
      }
    \end{tikzpicture}
    \caption{Some elements of \( \mathcal{A}_n \) witnessing its \( \kappa \)-density in
      \( \mathcal{U}_{\epsilon}(nd) \) for \( \kappa > \gcd(a,b) = c \).}
    \label{fig:An-dense-item-two}
  \end{figure}
  
  {\bf Case III:} \( a_{K} \ne 0 \), but \( b_{K} = 0 \) is treated in a similar way by taking \( c =
  -a_{K} \).
\end{proof}

From now on the set \( R \subseteq \mathcal{U}_{\epsilon}(d) \) is assumed to consist of tileable reals
\( R \subseteq \tileable \).  Note that given \( d \in \mathbb{R}^{>0} \) the set of possible
choices for \( R \subseteq \mathcal{U}_{\epsilon}(d) \cap \mathcal{T} \) is finite.

\begin{lemma}
  \label{lem:approximating-specified-frequency}
  Let \( \eta > 0 \).  For any \( N \) there exists
  \( M = M_{\Lem \ref{lem:approximating-specified-frequency}}(d, \eta) \) such that for any
  \( R \subseteq \mathcal{U}_{\epsilon}(d) \cap \tileable \) for which
  \( R \cap \freqa^{-1}[0, \rho - \eta] \) and \( R \cap \freqa^{-1}[\rho + \eta,1] \) are
  \( \epsilon \)-dense in \( \mathcal{U}_{\epsilon}(d) \) and for all \( n \ge M \) sets
  \( \mathcal{A}_{n} \) have at least \( N \)-many elements: \( |\mathcal{A}_{n}| \ge N \).
\end{lemma}
\begin{proof}
  Set \( \gamma_{i} = i\eta/ N \) for \( -N < i < N \) and let \( \zeta = \eta/2N \).  By Lemma
  \ref{lem:frequency-boost} there exists \( M = M_{\Lem \ref{lem:frequency-boost}}(d,\zeta) \) so
  large that for any \( n \ge M \) and any \( -N < i < N \) there is \( x_{i} \in \mathcal{A}_{n} \)
  such that \( \bigl| \freqa(x_{i}) - \gamma_{i} \bigr| < \zeta \).  Since
  \( |\gamma_{i} - \gamma_{j}| \ge 2\zeta \) for \( i \ne j \), elements \( x_{i} \) must be
  pairwise distinct, implying \( |\mathcal{A}_{n}| \ge 2N-1 \ge N \).
\end{proof}

\begin{lemma}
  \label{lem:delta-density}
  For any tileable \( d > 0 \), any \( 0 < \eta \le 1 \), \( 0 < \epsilon \le 1 \), and
  \( 0 < \delta \le \epsilon \) there exists
  \( M = M_{\Lem \ref{lem:delta-density}}(d,\epsilon, \delta,\eta) \) such that for any
  \( n \ge M \) and any tileable family \( R \subseteq \mathcal{U}_{\epsilon}(d) \cap \tileable \)
  which satisfies
  \begin{itemize}
  \item \( d \in R \);
  \item \( R \cap \freqa^{-1}[0,\rho - \eta] \) and
    \( R \cap \freqa^{-1}[\rho + \eta,1] \) are \( \epsilon \)-dense in
    \( \mathcal{U}_{\epsilon}(d) \);
  \end{itemize}
  the set \( \mathcal{A}_{n} = \mathcal{A}_{n}\bigl(\epsilon, (d)_{k=1}^{n}, (R)_{k=1}^{n}\bigr) \)
  is \( \delta \)-dense in \( \mathcal{U}_{\epsilon}(nd) \).
\end{lemma}

\begin{proof}
  By Lemma \ref{lem:approximating-specified-frequency} we may find \( M_{1} \) such that
  \[ \Bigl|\mathcal{A}_{M_{1}}\bigl(\epsilon, (d)_{1}^{M_{1}}, (R)_{1}^{M_{1}} \bigr)\Bigr| \ge 2 +
  4\epsilon/\delta \]
  for all \( R \) satisfying the assumptions.
  We may therefore pick non-zero \( y_{1}, y_{2} \in \mathcal{A}_{M_{1}} \) such that
  \( |y_{1}-y_{2}| < \delta/2 \), and \( y_{1} - M_{1}d \), \( y_{2} - M_{1}d \) have same signs.
  Let \( z \in \mathcal{A}_{M_{1}}\) be any element such that \( z - M_{1}d \) has the opposite
  sign.  Applying Lemma \ref{lem:generating-elements-of-An} to \( y_{1} \) and \( z \) we can find
  \( M_{2} \ge M_{1} \) such that
  \begin{itemize}
  \item either
    \( \mathcal{A}_{n} \) is \( \delta \)-dense in \( \mathcal{U}_{\epsilon}(nd) \) for any
    \( n \ge M_{2} \), or
  \item 
    \( \mathcal{A}_{n} \) is \( 2c \)-dense in \( \mathcal{U}_{\epsilon}(nd) \) for any
    \( n \ge M_{2} \), where 
    \[ c = \gcd(y_{1} - M_{1}d, z - M_{1}d), \]
    and moreover \( nd - c, nd + c \in \mathcal{A}_{n} \).
  \end{itemize}
  Note that there are only finitely many possibilities for the choice of
  \( y_{1}, y_{2}, z \in \mathcal{A}_{M_{1}} \), which let us choose \( M_{2} \) so
  large as to work for all of them at the same time.  In the first case we are done.  Suppose the
  latter holds.  Note that \( y_{2} + (M_{2}-M_{1})d \in \mathcal{A}_{M_{2}} \), and depending on
  whether \( y_{2} - M_{1}d < 0 \) or \( y_{2} - M_{1}d > 0 \), apply Lemma
  \ref{lem:generating-elements-of-An} to \( y_{2} + (M_{2}-M_{1})d \) and \( M_{2}d + c \) or to
  \( y_{2} + (M_{2}-M_{1})d \) and \( M_{2}d - c \).  Again, \( M_{3} \) can be picked so large as
  to work for all possible \( y_{2} \), \( c \), and \( R \) at the same time. We claim that the
  resulting \( M_{3} \) works, because \( c' = \gcd(y_{2} - M_{1}d, c) < \delta \), which guarantees
  hat \( \mathcal{A}_{n} \) is \( \delta \)-dense in \( \mathcal{U}_{\epsilon}(nd) \) for all \( n
  \ge M_{3} \).  Indeed, since \( y_{1} - M_{1}d \) is a multiple of \( c \), both
  \( y_{1} - M_{1}d \) and \( y_{2} - M_{1}d \) are integer multiples of \( c' \), whence
  \( 0 < |y_{1} - y_{2}| < \delta \) implies \( c' < \delta \).
\end{proof}

\subsection{Moving points: General case.}
\label{sec:moving-interv-gener}

We are now back to the situation of general sequences \( (d_{k}) \), \( (R_{k}) \) with possibly
distinct \( d_{k} \) and \( R_{k} \).  Our first lemma is a close analog of Lemma
\ref{lem:delta-density} where \( (d_{k}) \) and \( (R_{k}) \) are not necessarily constant.

\begin{lemma}
  \label{lem:propagation-of-freedom-pigeon-hole}
  For any \( D \in \mathbb{R}^{\ge 0} \), any \( 0 < \epsilon \le 1 \) , \(0 < \delta \le \epsilon,
  \) \( 0 < \eta \le 1\), there is
  \( M = M_{\Lem \ref{lem:propagation-of-freedom-pigeon-hole}}(D,\epsilon, \delta,\eta) \) such that for any
  \( n \ge M \), any sequence \( (d_{k})_{k=1}^{n} \) of positive
  reals, \( 2\epsilon + \alpha \le d_{k} \le D \), any sequence
  \( R_{k} \subseteq \mathcal{U}_{\epsilon}(d_{k}) \cap \tileable \) satisfying
  \begin{itemize}
  \item \( d_{k} \in R_{k} \);
  \item \( R_{k} \cap \freqa^{-1}[0,\rho - \eta] \) and
    \( R_{k} \cap \freqa^{-1}[\rho + \eta,1] \) are \( \epsilon \)-dense in
    \( \mathcal{U}_{\epsilon}(d_{k}) \);
  \end{itemize}
  the set \( \mathcal{A}_{n} \) is \( \delta \)-dense in
  \( \mathcal{U}_{\epsilon}(\sum_{k=1}^{n}d_{k}) \).
\end{lemma}

\begin{proof}
  The set \( \tileable \cap (0, D] \) is finite and thus the set of possible choices of
  \( R_{k} \subseteq \mathcal{U}_{\epsilon}(d_{k})\cap \tileable \) for each \( k \) is bounded
  above by \( 2^{|\tileable \cap [0, D + 1]|} \).  Set
  \[ M = \bigl| \mathcal{T} \cap [0,D] \bigr| \cdot 2^{| \tileable \cap [0, D +1] |} \cdot
  \max_{\genfrac{}{}{0pt}{}{d \in \mathcal{T}}{d \le D}}M_{\Lem
    \ref{lem:delta-density}}(d,\epsilon,\delta,\eta).\]
  Let \( n \ge M \), and let \( d_{k} \), \(k \le n \), be given.  By the choice of \( M \) and the
  pigeon-hole principle there are indices \( k_{1} < k_{2} < \cdots < k_{N} \) such that for all
  \( i,j \le N \)
  \[ d_{k_{j}} = d_{k_{i}} =: d, \qquad R_{k_{j}} = R_{k_{i}} =: R, \]
  and \( N \ge M_{\Lem \ref{lem:delta-density}}(d,\epsilon,\delta,\eta) \).

  It is helpful to go back to Figure \ref{fig:def-of-An} for a moment.  Our situation now is special
  in the following aspects.  First of all distances \( d_{k} \in R_{k} \), so once \( z_{k-1}' \)
  has been picked we may always set \( z_{k}' = z_{k-1}' + d_{k} \), which will never violate the
  condition \( \dist(z_{k}, z_{k}') < \epsilon \).  Also, we have a large collection of indices
  \( k_{i} \), \( i \le N \), where all the gaps \( d_{k_{i}} \) and sets of admissible distances
  \( R_{k_{i}} \) are the same.  The natural \( N \) is so large that Lemma \ref{lem:delta-density}
  applies and the set \( \mathcal{A}_{N}\bigl(\epsilon, (d)_{i=1}^{N}, (R)_{i=1}^{N}\bigr) \) is \(
  \delta \)-dense in \( \mathcal{U}_{\epsilon}(Nd) \).  Recall that \( x \in
  \mathcal{A}_{N}\bigl(\epsilon, (d)_{i=1}^{N}, (R)_{i=1}^{N}\bigr) \) means that there exist \(
  y_{i} \in R \), \( i \le N \), such that \( x = \sum_{i=1}^{N}y_{i} \) and \( \bigl|
  \sum_{i=1}^{r}(d-y_{i}) \bigr| < \epsilon \) for all \( r \le N \).   Any such \( x \) naturally
  corresponds to an element \( \tilde{x} \in \mathcal{A}_{n}\bigl(\epsilon, (d_{k})_{k=1}^{n},
  (R_{k})_{k=1}^{n}\bigr) \) defined by \( \tilde{x} = \sum_{k=1}^{n} \tilde{y}_{k} \) where
  \begin{displaymath}
    \tilde{y}_{k} =
    \begin{cases}
      y_{i} & \textrm{if } k = k_{i},\\
      d_{k} & \textrm{otherwise.}
    \end{cases}
  \end{displaymath}
  Since
  \[ \tilde{x} - \sum_{k=1}^{n}\tilde{y}_{k} = x - \sum_{i=1}^{N}y_{i}, \]
  \( \mathcal{A}_{N}\bigl(\epsilon,(d)_{i=1}^{N}, (R)_{i=1}^{N}\bigr) \) is \( \delta \)-dense in
  \( \mathcal{U}_{\epsilon}(Nd) \) if and only if
  \( \mathcal{A}_{n}\bigl(\epsilon, (d_{k})_{k=1}^{n}, (R_{k})_{k=1}^{n}\bigr) \) is
  \( \delta \)-dense in \( \mathcal{U}_{\epsilon}\bigl(\sum_{k=1}^{n} d_{k} \bigr) \).  The lemma follows.
\end{proof}

We are finally ready to prove the main technical result of this section. The following lemma will
supply the step of induction in the proof of Theorem \ref{thm:tiling-general-flows}.  Lemma
\ref{lem:main-induction-lemma} is a formalization of the ``Propagation of Freedom'' principle to
which we alluded earlier in this section.  There are two differences from Lemma
\ref{lem:propagation-of-freedom-pigeon-hole}.  First, we no longer assume that \( d_{k} \in R_{k}
\).  And second, the conclusion is stronger as subsets of \( \mathcal{A}_{n} \) of elements having
prescribed \( \alpha \)-frequencies are claimed to be \( \delta \)-dense.

\begin{lemma}[Co-sparse Induction Step]
  \label{lem:main-induction-lemma}
  For any \( D \in \mathbb{R}^{\ge 0} \), any \( \epsilon, \delta, \eta, \nu, \nu' > 0 \), where
  \( 0 < \delta \le \epsilon \) and \( 0 \le \nu' < \nu \le \eta \), there exists
  \( M = M_{\Lem\ref{lem:main-induction-lemma}}(D,\epsilon, \delta,\eta,\nu,\nu') \in \mathbb{N} \)
  such that for any \( n \ge M \), any sequence of reals \( (d_{k})_{k=1}^{n} \), and any sequence
  of tileable families \( R_{k} \subseteq \mathcal{U}_{\epsilon}(d_{k}) \cap \tileable \),
  \( 1 \le k \le n \), satisfying
  \begin{enumerate}[(i)]
  \item \( 2\epsilon + \alpha \le d_{k} \le D \);
  \item \( R_{k} \cap \freqa^{-1}[0, \rho - \eta] \) and
    \( R_{k} \cap \freqa^{-1}[\rho + \eta, 1] \) are \( \epsilon/6 \)-dense in
    \( \mathcal{U}_{\epsilon}(d_{k}) \);
  \end{enumerate}
  sets \( \mathcal{A}_{n} \cap \freqa^{-1}[\rho - \nu, \rho - \nu'] \) and
  \( \mathcal{A}_{n} \cap \freqa^{-1}[\rho + \nu',\rho + \nu] \) are \( \delta \)-dense in
  \( \mathcal{U}_{\epsilon/2}(\sum_{k=1}^{n}d_{k}) \).
\end{lemma}

\begin{proof}
  We construct \( M = M_{1} + M_{2} \) as a sum of two numbers: using the first \( M_{1} \) segments
  we achieve the \( \delta \)-density, and then use \( M_{2} \) intervals to push frequencies to \(
  [\rho - \nu, \rho - \nu'] \) and \( [\rho + \nu', \rho + \nu] \) intervals.
  Let \( M_{1} = M_{\Lem \ref{lem:propagation-of-freedom-pigeon-hole}}(D+2,\epsilon, \delta,\eta) \), \( \zeta =
  \frac{\nu - \nu'}{6} \), and
  \begin{displaymath}
    \begin{aligned}
      M_{2}' &= M_{\Lem \ref{lem:frequency-boost}}(D+1,\zeta)\\
      M_{2}'' &= \Bigl\lceil \frac{2\beta}{\alpha} \cdot \frac{M_{1}D +1}{\zeta} \Bigr\rceil\\
      M_{2} &= \max\{M_{2}', M_{2}''\}.
    \end{aligned}
  \end{displaymath}
  We claim that \( M = M_{1} + M_{2} \) works.  Let \( n \ge M \) be given.  For a sequence
  \( d_{k} \) and \( R_{k} \) as in the assumptions we argue as follows.  Similarly to the proof of Lemma
  \ref{lem:frequency-boost}, construct inductively \( \tilde{d}_{k} \) for \( k \le M_{1} \) such that
  \begin{enumerate}[(a)]
  \item \( |d_{k} - \tilde{d}_{k}| < \epsilon/6 \);
  \item \( |\sum_{k=1}^{r}(d_{k} - \tilde{d}_{k})| < \epsilon/6 \) for all \( r \le M_{1} \).
  \end{enumerate}
  Set \( \widetilde{R}_{k} = R_{k} \cap \mathcal{U}_{5\epsilon/6}(\tilde{d}_{k}) \) and note that
  \[ \mathcal{A}_{M_{1}}\bigl( 5\epsilon/6, (\tilde{d}_{k})_{k=1}^{M_{1}},
  (\widetilde{R}_{k})_{k=1}^{M_{1}}\bigr) \subseteq \mathcal{A}_{M_{1}}\bigl(\epsilon,
  (d_{k})_{k=1}^{M_{1}}, (R_{k})_{k=1}^{M_{1}}\bigr) \]
  because for any element in
  \( \mathcal{A}_{M_{1}}\bigl( 5\epsilon/6, (\tilde{d}_{k})_{k=1}^{M_{1}},
  (\widetilde{R}_{k})_{k=1}^{M_{1}}\bigr) \) each shift made by \( \tilde{d}_{k} \) is also an
  admissible shift for \( d_{k} \).  Whence
  \( \mathcal{A}_{M_{1}}\bigl( 5\epsilon/6, (\tilde{d}_{k})_{k=1}^{M_{1}},
  (\widetilde{R}_{k})_{k=1}^{M_{1}}\bigr) \) is \( \delta \)-dense in
  \( \mathcal{U}_{5\epsilon/6}\bigl(\sum_{k=1}^{M_{1}} \tilde{d}_{k}\bigr) \) by the choice of
  \( M_{1} \).  It follows that
  \[ \mathcal{A}_{M_{1}}\bigl(\epsilon, (d_{k})_{k=1}^{M_{1}}, (R_{k})_{k=1}^{M_{1}}\bigr) \textrm{
    is \( \delta \)-dense in } \mathcal{U}_{4\epsilon/6}\bigl(\sum_{k=1}^{M_{1}}d_{k}\bigr). \]

  We now improve the frequency while keeping the \( \delta \)-density by constructing elements \( y_{1},
  y_{2}\) such that for  \( \mathcal{A}_{M_{1}} \) constructed above each element in \(
  \mathcal{A}_{M_{1}} + y_{1} \) will have \( \alpha \)-frequency in \( [\rho - \nu, \rho - \nu'] \)
  and each one in \( \mathcal{A}_{M_{1}} + y_{2} \) will have frequency from \( [\rho + \nu', \rho +
  \nu] \).  Moreover, both \( \mathcal{A}_{M_{1}} +y_{1} \) and \( \mathcal{A}_{M_{1}} + y_{2} \)
  will be subsets of \( \mathcal{A}_{n}\bigl(\epsilon, (d_{k})_{k=1}^{n}, (R_{k})_{k=1}^{n}\bigr)
  \), and the lemma will therefore be proved.  Here are the details of the construction of \( y_{1}
  \) and \( y_{2} \).

  Since \( M_{2} \ge M_{2}' \), we can apply the conclusion of Lemma \ref{lem:frequency-boost} to
  the sequence \( (d_{k})_{k=M_{1}+1}^{n} \),
  \[ \bar{R}_{k} = R_{k} \cap \mathcal{U}_{\epsilon/6}(d_{k}), \]
  \( \gamma_{1} = \rho + \frac{\nu + \nu'}{2} \), \( \gamma_{2} = \rho - \frac{\nu + \nu'}{2} \),
  and \( \zeta = \frac{ \nu - \nu'}{6} \).  This yields elements
  \( y_{1}, y_{2} \in \mathcal{A}_{n}\bigl(\epsilon/6, (d_{k})_{k=M_{1}+1}^{n},
  (\bar{R}_{k})_{k=M_{1}+1}^{n}\bigr) \) such that
  \[ \bigl|\freqa(y_{j}) - \gamma_{j}\bigr| < \zeta \quad \textrm{for \( j = 1,2 \)}. \]
  Observe that
  \begin{displaymath}
    \mathcal{A}_{M_{1}}\bigl(\epsilon, (d_{k})_{k=1}^{M_{1}}, (R_{k})_{k=1}^{M_{1}}\bigr) \cap
    \mathcal{U}_{5\epsilon/6}\bigl(\sum_{k=1}^{M_{1}}d_{k}\bigr) + y_{j} \subset \mathcal{A}_{n}\bigl(\epsilon,
    (d_{k})_{k=1}^{n}, (R_{k})_{k=1}^{n}\bigr) \quad \textrm{for \( j = 1,2 \)}
  \end{displaymath}
  and these sets are \( \delta \)-dense in
  \( \mathcal{U}_{\epsilon/2}\bigl(\sum_{k=1}^{n}d_{k}\bigr) \).  Finally, since
  \( M_{2} \ge M_{2}'' \) from Lemma \ref{lem:frequency-length-estimate} it follows that
  \[ \bigl|\freqa(z) - \freqa(y_{j}) \bigr| < \zeta\quad \textrm{for any } z \in
  \mathcal{A}_{M_{1}}\bigl(\epsilon, (d_{k})_{k=1}^{M_{1}}, (R_{k})_{k=1}^{M_{1}}\bigr) \cap
  \mathcal{U}_{5\epsilon/6}\bigl(\sum_{k=1}^{M_{1}}d_{k}\bigr) + y_{j}, \]
  which implies that \( \freqa(z) \in [\rho - \nu, \rho - \nu'] \) for \( j = 1 \) and \( \freqa(z)
  \in [\rho + \nu', \rho + \nu] \) when \( j =2 \).
  \end{proof}

\section{Tiled reals}
\label{sec:flex-part}

Throughout this section \( \cl{\eta} = (\eta_{k})_{k=0}^{\infty} \) denotes a strictly decreasing
sequence of positive reals converging to zero;  we assume that \( \eta_{0} = 1 \) and \( \eta_{1}
\le \min\{\rho, 1-\rho\} \).
Symbol \( \cl{L} = (L_{k})_{k=0}^{\infty} \), with possible superscripts, will denote a
strictly increasing unbounded sequence \( (L_{k})_{k=0}^{\infty} \) with the agreement that \( L_{0}
= \beta \).

A \emph{tiled} real is a number \( z \in \tileable \cup \{0\} \) together with a sequence
\( (\sigma_{i})_{i=1}^{n} \), called the \emph{decomposition of \( z \)}, such that
\( \sigma_{i} \in \{\alpha, \beta\} \) and \( \sum_{i=1}^{n} \sigma_{i} = z \).  One should think of
a tiled real as an interval which has been cut into pieces of length \( \alpha \) and \( \beta \).
The set of tiled reals is denoted by \( \tiled \).  By convention, \( 0 \in \tiled \).  Given two
tiled reals \( z_{1}, z_{2} \in \tiled \), we define their sum \( z_{1} + z_{2} \in \tiled \) in the
natural way: if \( (\sigma_{i}^{(1)})_{i=1}^{n_{1}} \) is the decomposition of \( z_{1} \) and
\( (\sigma_{i}^{(2)})_{i=1}^{n_{2}} \) decomposes \( z_{2} \), then \( z_{1} + z_{2} \) is
partitioned by \( (\sigma_{i})_{i=1}^{n_{1} + n_{2}} \)
\begin{displaymath}
  \sigma_{k} =
  \begin{cases}
    \sigma_{k}^{(1)} & \textrm{ if \( k \le n_{1} \)}\\
    \sigma_{k - n_{1}}^{(2)} & \textrm{ if \( n_{1} < k \le n_{1} + n_{2} \)}.
  \end{cases}
\end{displaymath}
Note that the addition of tiled reals is associative, but not commutative.  If \( \mathcal{C} \) is
an \( \{\alpha,\beta\} \)-regular cross section, then any of its segments between a pair of points
within an orbit naturally corresponds to a tiled real.

We shall need to control two parameters of a tiled real \textemdash{} its length and its
\( \alpha \)-frequency.  Note that any tiled real can be viewed as a tileable one by forgetting the
decomposition, and so each tiled real has an \( \alpha \)-frequency associated to it. 
We introduce the following sets: given \( \eta > 0 \) and real \( L > 0 \) 
\begin{displaymath}
  \begin{aligned}
    B_{\eta}[L] &= \bigl\{ z \in \tiled \bigm| z \le L,\ |\freqa(z) - \rho| \le
    \eta \bigr\},\\
    SB_{\eta}[L] &= \bigl\{ z \in \tiled \bigm| z = \sum_{i=1}^{r}y_{i},\
    y_{i} \in B_{\eta}[L],\ r \in \mathbb{Z}^{> 0} \bigr\},\\
  \end{aligned}
\end{displaymath}
The frequency of zero, \( \freqa(0) \), is undefined, but by our convention \( 0 \in B_{\eta}[L] \)
for all \( L \) and all \( \eta \).  Note that \( SB_{\eta}[L] \) is just the semigroup generated by
\( B_{\eta}[L] \).  The following properties are immediate from the definitions.
\begin{proposition}
  \label{prop:properties-of-SB-sets}
  Sets of the form \( B_{\eta}[L] \) and \( SB_{\eta}[L] \) satisfy the following:
  \begin{enumerate}[(i)]
  \item \( B_{\eta}[L] \subseteq SB_{\eta}[L] \).
  \item \( | \freqa(z) - \rho | \le \eta \) for all \( z \in SB_{\eta}[L] \setminus \{0\} \).
  \item\label{item:SB-closed-sum} If \( z_{1}, z_{2} \in SB_{\eta}[L] \), then also
    \( z_{1} + z_{2} \in SB_{\eta}[L] \).
  \item If \( \eta' \ge \eta \) and \( L' \ge L \), then \( B_{\eta}[L] \subseteq
    B_{\eta'}[L'] \) and \( SB_{\eta}[L]  \subseteq SB_{\eta'}[L'] \).
  \end{enumerate}
\end{proposition}

\begin{definition}
  \label{def:flexible-sequence}
  A sequence \( \cl{L} \) is said to be \( \cl{\eta} \)-flexible if for any \( n \in \mathbb{N} \), any
  \( 0 \le \nu' < \nu \le \eta_{n+1} \), the sets\footnote{Recall that we assume that \( \eta_{1}
    \le \min\{\rho, 1-\rho\} \), which ensures that both \( [\rho-\nu, \rho-\nu'] \) and \( [\rho +
    \nu', \rho + \nu] \) are subintervals of \( [0,1] \).}
  \[ \bigcap_{k=0}^{n} SB_{\eta_{k}}[L_{k}] \cap \freqa^{-1}[\rho-\nu,\rho-\nu']  \textrm{ and }
   \bigcap_{k=0}^{n} SB_{\eta_{k}}[L_{k}] \cap \freqa^{-1}[\rho+\nu',\rho+\nu] \] are
  asymptotically dense in \( \mathbb{R}^{>0} \).
\end{definition}

The definition of an \( \cl{\eta} \)-flexible sequence says that a large real can be shifted
slightly to a real which for any \( 0 \le k \le n \) can be cut into pieces from
\( B_{\eta_{k}}[L_{k}] \) and moreover the \( \alpha \)-frequency of the whole interval is very
close to any given number in \( [\rho - \eta_{k+1}, \rho + \eta_{k+1}] \).  In Lemma
\ref{lem:regular-tilings} below, we show that any \( \cl{\eta} \) admits an \( \cl{\eta} \)-flexible
sequence \( \cl{L} \).  The construction of \( \cl{L} \) is inductive, and the base step amounts to
checking that
\[ SB_{\eta_{0}}[L_{0}] \cap \freqa^{-1}[\rho-\nu,\rho-\nu'] \textrm{ and } SB_{\eta_{0}}[L_{0}]
\cap \freqa^{-1}[\rho+\nu',\rho+\nu] \]
are asymptotically dense in \( \mathbb{R}^{>0} \).  Since we always assume that \( \eta_{0} = 1 \),
and \( L_{0} = \beta \), one has \( \{\alpha,\beta\} \subseteq B_{\eta_{0}}[L_{0}] \), and therefore
\( SB_{\eta_{0}}[L_{0}] = \tiled \).  The base case is thus covered by the following lemma.

\begin{lemma}
  \label{lem:flexibitility-of-tileable-reals}
  Let \( 0 \le \nu' < \nu  \) be given.
  \begin{enumerate}[(i)]
  \item\label{item:plus-one-minus-rho} If \( \nu \le 1- \rho \), then the set
    \( \freqa^{-1}[\rho + \nu', \rho + \nu] \) is asymptotically dense in \( \mathbb{R} \).
  \item\label{item:minus-rho} If \( \nu \le \rho \), then the set
    \( \freqa^{-1}[\rho - \nu, \rho - \nu'] \) is asymptotically dense in \( \mathbb{R} \).
  \end{enumerate}
\end{lemma}

\begin{proof}
  We give a proof of \eqref{item:plus-one-minus-rho}, item \eqref{item:minus-rho} is proved
  similarly.  Pick \( \epsilon > 0 \).  We show that \( \freqa^{-1}[\rho + \nu', \rho + \nu] \) is
  \( \epsilon \)-dense in \( [N, \infty) \) for sufficiently large \( N \).  Since
  \( \rho + (\nu + \nu')/2 < 1 \) we may pick \( x \in \tileable \) such that
  \[ \Bigl| \rho + \frac{\nu + \nu'}{2} - \freqa(x) \Bigr| < \zeta \quad \textrm{where } \zeta =
  \frac{\nu - \nu'}{4}. \]
  Since \( \alpha \) and \( \beta \) are rationally independent, the group
  \( \langle \alpha, \beta \rangle \) is dense in \( \mathbb{R} \), so we may pick \( s_{1}, \ldots,
  s_{n} \in \langle \alpha, \beta \rangle\) such that \( \{s_{1}, \ldots, s_{n}\} \) is \(
  \epsilon/2 \)-dense in \( [0, x] \).  Let
  \[ s_{i} = p_{i}\alpha + q_{i}\beta \]
  and set \( m = \max\bigl\{|p_{i}|, |q_{i}| : i \le n \bigr\} \).  The set \( \{ s_{i} + m : i \le n
  \} \) therefore consists of tileable reals and is \( \epsilon/2 \)-dense in \( [m, m + x] \).
  Using Lemma \ref{lem:frequency-length-estimate} we may find \( K \) so big that for all \( k \ge K \)
  \[ \bigl| \freqa(kx + s_{i} + m) - \freqa(kx) \bigr| < \zeta\quad \textrm{for all } i \le n. \]
  Therefore \( \freqa(kx + s_{i} +m) \in [\rho + \nu', \rho + \nu] \) for all \( i \le n \), and the
  set
  \[ \bigl\{kx + s_{i} + m : i \le n, k \ge K\bigr\} \quad \textrm{is \( \epsilon \)-dense in }
  [Kx+m, \infty). \qedhere \]
\end{proof}

\begin{lemma}
  \label{lem:regular-tilings}
  For any \( \cl{\eta} \) there exists an \( \cl{\eta} \)-flexible sequence \( \cl{L} \).
\end{lemma}
\begin{proof}
  The sequence \( (L_{n})_{n=0}^{\infty} \) is constructed inductively, we ensure at the
  \( n^{\mathrm{th}} \) stage of the construction that
  \[ \bigcap_{k=0}^{n} SB_{\eta_{k}}[L_{k}] \cap \freqa^{-1}[\rho-\nu,\rho-\nu'] \quad \textrm{ and
  } \quad \bigcap_{k=0}^{n} SB_{\eta_{k}}[L_{k}] \cap \freqa^{-1}[\rho+\nu',\rho+\nu]\]
  are asymptotically dense in \( \mathbb{R} \) for any choice of
  \( 0 \le \nu' < \nu \le \eta_{n+1} \).  The base
  case \( n = 0 \) is covered by Lemma \ref{lem:flexibitility-of-tileable-reals}.

  Suppose we have constructed \( (L_{k})_{k=0}^{n} \).  By the inductive assumption we may choose
  \( N \in \mathbb{R} \), such that both
  \begin{equation}
    \label{eq:sb-dense}
    \bigcap_{k=0}^{n} SB_{\eta_{k}}[L_{k}] \cap \freqa^{-1}\bigl[\rho - \eta_{n+1}, \rho -
    \eta_{n+2}]  
    \textrm{ and } \bigcap_{k=0}^{n} SB_{\eta_{k}}[L_{k}] \cap \freqa^{-1}[\rho + \eta_{n+2}, \rho +
    \eta_{n+1}]
  \end{equation}
  are \( 1/6 \)-dense in \( [N, \infty) \).  We take 
  \[ L_{n+1} = \max \bigl\{L_{n} + 1, 2(N+2) \bigr\}. \]
  and claim that
  \[ \bigcap_{k=0}^{n+1} SB_{\eta_{k}}[L_{k}] \cap \freqa^{-1}[\rho-\nu,\rho-\nu'] \quad \textrm{
    and } \quad \bigcap_{k=0}^{n+1} SB_{\eta_{k}}[L_{k}] \cap \freqa^{-1}[\rho+\nu',\rho+\nu]\]
  are asymptotically dense in \( \mathbb{R} \) for any choice of
  \( 0 \le \nu' < \nu \le \eta_{n+2} \).  Let \( 0 \le \nu' < \nu \le \eta_{n+2} \) be given and
  pick \( \epsilon \le 1/2 \).  We show that
  \[ \bigcap_{k=0}^{n+1}SB_{\eta_{k}}[L_{k}] \cap \freqa^{-1}[\rho + \nu', \rho + \nu] \]
  is \( \epsilon \)-dense in \( [\widetilde{N}, \infty) \) for sufficiently large \( \widetilde{N}
  \).  The argument for the \( \epsilon \)-density of
  \[ \bigcap_{k=0}^{n+1}SB_{\eta_{k}}[L_{k}] \cap \freqa^{-1}[\rho - \nu, \rho - \nu'] \]
  is similar.

  Let \( d \) be any tileable real satisfying \( N+1 \le d \le N + 2 \) and
  \( d \in \bigcap_{k=0}^{n} SB_{\eta_{k}} [L_{k}] \) which exists by \eqref{eq:sb-dense}; set
  \[ R = \bigcap_{k=0}^{n}SB_{\eta_{k}}[L_{k}] \cap \mathcal{U}_{1}(d) \cap \freqa^{-1}[\rho -
  \eta_{k+1}, \rho + \eta_{k+1}]. \]
  Note that by the choice of \( N \) and \eqref{eq:sb-dense}, both
  \[ R \cap \freqa^{-1}[0,\rho - \eta_{k+2}] \textrm{ and } R \cap \freqa^{-1}[\rho
  + \eta_{k+2},1] \]
  are \( \epsilon/6 \)-dense in \( \mathcal{U}_{\epsilon}(d) \) and
  \( R \subseteq B_{\eta_{k+1}}[L_{k+1}] \).  Pick some \( \tilde{\nu}, \tilde{\nu}' \) such that
  \( 0 \le \nu' < \tilde{\nu}' < \tilde{\nu} < \nu \le \eta_{n+2} \).  By Lemma
  \ref{lem:main-induction-lemma} there is
  \[ M = M_{\Lem \ref{lem:main-induction-lemma}}(d,1,\epsilon, \eta_{n+2}, \tilde{\nu},
  \tilde{\nu}') \] such that for any \( m \ge M \) the set
  \[ \mathcal{A}_{m}\bigl(1,(d)_{1}^{m},(R)_{1}^{m}\bigr) \cap \freqa^{-1}[\rho + \tilde{\nu}', \rho
  + \tilde{\nu}] \]
  is \( \epsilon \)-dense in \( \mathcal{U}_{1/2}(md) \).  Any element of \( \mathcal{A}_{m} \) is a
  sum of elements in \( R \), so by item \ref{item:SB-closed-sum} of
  Proposition \ref{prop:properties-of-SB-sets}, this implies that
  \begin{equation}
    \label{eq:SB-dense-in-md}
    \bigcap_{k=0}^{n+1} SB_{\eta_{k}}[L_{k}] \cap \freqa^{-1}[\rho + \tilde{\nu}', \rho +
    \tilde{\nu}] \textrm{ is \( \epsilon \)-dense in } \mathcal{U}_{1/2}(md).
  \end{equation}
  Take \( \widetilde{N} = (\widetilde{M}+1)d \) where \( \widetilde{M} \ge M \) is so big that
  \begin{equation}
    \label{eq:N-tilde-big}
    \freqa(z) \in [\rho+\tilde{\nu}', \rho + \tilde{\nu}] \implies \freqa(z + x) \in [\rho + \nu',
    \rho + \nu] 
  \end{equation}
  holds for all tileable \( z \ge \widetilde{M}N \) and all tileable \( x \le d \) (cf. Lemma
  \ref{lem:frequency-length-estimate}).

  We show that \( \widetilde{N} \) works.  Let \( z \in [\widetilde{N}, \infty) \) and let
  \( m \ge \widetilde{M} \) be the unique natural such that 
  \[ z \in \bigl[(m+1)d, (m+2)d\bigr). \]
  By the choice of \( N \) and \eqref{eq:sb-dense} there is some
  \[ y \in \bigcap_{k=0}^{n}SB_{\eta_{k}}[L_{k}] \cap \freqa^{-1}[\rho + \eta_{n+2}, \rho +
  \eta_{n+1}] \cap (0, 2N+4) \subseteq B_{\eta_{n+1}}[L_{n+1}]\]
  such that \( |z - y - md| < 1/6 \) for some \( m \ge M \), so
  \( z - y \in \mathcal{U}_{1/2}(md) \).  Finally, since
  \( \bigcap_{k=0}^{n+1}SB_{k}[L_{k}] \cap \freqa^{-1}[\rho +\tilde{\nu}', \rho+\tilde{\nu}] \) is
  \( \epsilon \)-dense in \( \mathcal{U}_{1/2}(md) \) by \eqref{eq:SB-dense-in-md} there is some
  \[ x \in \bigcap_{k=0}^{n+1}SB_{k}[L_{k}] \cap \freqa^{-1}[\rho +\tilde{\nu}',
  \rho+\tilde{\nu}]\]
  such that \( |z - y - x| < \epsilon \).  Using \eqref{eq:N-tilde-big} we conclude that
  \[ y + x \in \bigcap_{k=0}^{n+1} SB_{\eta_{k}}[L_{k}] \cap \freqa^{-1}[\rho + \nu', \rho +
  \nu].  \]
  Since \( z \) was arbitrary, the lemma follows.
\end{proof}

\begin{definition}
  \label{def:eta-L-tiled}
  We say that \( z \in \tiled \) is \( (\cl{\eta}, \cl{L}) \)-tiled if \( z \in
  SB_{\eta_{n}}[L_{n}] \) whenever \( z \ge L_{n} \).  Note that if \( \cl{L}' \) is such that \(
  L_{n}' \ge L_{n} \) for all \( n \), then any \( z \) which is \( (\cl{\eta}, \cl{L}) \)-tiled is
  also \( (\cl{\eta}, \cl{L}') \)-tiled.
\end{definition}

For any given \( z \) the property of being \( (\cl{\eta}, \cl{L}) \)-tiled depends only on a finite
segment of \( \cl{L} \), because the condition is vacuous for \( n \) such that \( L_{n} > z \).  On
the other hand we may refer to a family of tiled reals as being \( (\cl{\eta}, \cl{L}) \)-tiled,
meaning that each member of the family is \( (\cl{\eta}, \cl{L}) \)-tiled, and in that case all the
elements of the sequence \( \cl{L} \) are substantial as soon as the family has arbitrarily large
tiled reals.

\begin{lemma}
  \label{lem:tiled-sum}
  Let \( \cl{\eta} \), \( \cl{L} \) be given.  For any \( m_{0} \in \mathbb{N} \), any
  \( n_{0} \in \mathbb{N} \), \( n_{0} \ge 2 \), and any \( K \in \mathbb{R}^{>0} \) there is
  \( \cl{L}' = \cl{L}'_{\Lem \ref{lem:tiled-sum}}(\cl{\eta}, \cl{L}, m_{0}, n_{0}, K) \) such that
  \begin{enumerate}[(i)]
  \item \( L_{k}' = L_{k} \) for \( k \le m_{0} \);
  \item \( L_{k}' \ge L_{k} \) for \( k \in \mathbb{N} \);
  \end{enumerate}
  and for any \( z_{1}, \ldots, z_{n_{0}}, y_{1}, \ldots, y_{n_{0}-1} \in \tiled  \), satisfying
  \begin{itemize}
  \item   \( y_{i} \le K \);
  \item \( z_{i}, y_{i} \in SB_{\eta_{k}}[L_{k}] \) for all \( k \le m_{0} \);
  \item \( z_{i} \) are \( (\cl{\eta}, \cl{L}) \)-tiled;
  \end{itemize}
  the sum
  \[ J = z_{1} + y_{1} + z_{2} + y_{2} + \cdots + z_{n_{0}-1} + y_{n_{0}-1} + z_{n_{0}} \]
  is \( (\cl{\eta}, \cl{L}') \)-tiled and \( J \in SB_{\eta_{k}}[L_{k}'] \) for all \( k \le m_{0} \).
\end{lemma}

Here is a more verbose explanation of the statement.  The case \( n_{0} = 3 \) is shown in Figure
\ref{fig:small-y-cant-change-frequency-by-much}.
\begin{figure}[htb]
  \centering
  \begin{tikzpicture}
    \foreach \x in {0, 3, 4.7, 6.5, 8.0, 10.2} {
      \filldraw (\x, 0) circle (1pt);
    }
    \draw (0,0) -- (3,0);
    \draw (4.7,0) -- (6.5,0);
    \draw (8,0) -- (10.2,0);
    \draw (1.5,0.2) node {\( z_{1} \)};
    \draw (5.6,0.2) node {\( z_{2} \)};
    \draw (9.1,0.2) node {\( z_{3} \)};
    \draw (3.85,0) node {\( y_{1} \)};
    \draw (7.25,0) node {\( y_{2} \)};
  \end{tikzpicture}
  \caption{}
  \label{fig:small-y-cant-change-frequency-by-much}
\end{figure}
Numbers \( y_{1} \) and \( y_{2} \) are known to be bounded from above by \( K \), but \( z_{i} \)'s
may be arbitrarily large.  We know that \( z_{i} \) and \( y_{j} \) are all elements of
\( SB_{\eta_{k}}[L_{k}] \) for \( k \le m_{0} \), so their sum \( J \) is also an element of
\( SB_{\eta_{k}}[L_{k}] \).  We assume that \( z_{i} \)'s are \( (\cl{\eta}, \cl{L}) \)-flexible,
i.e., if \( z_{i} \) is particularly long, \( z_{i} \ge L_{n} \), then it can be partitioned into
pieces of size at most \( L_{n} \) each having \( \alpha \)-frequency \( \eta_{n} \)-close to
\( \rho \).  The sum is not necessarily \( (\cl{\eta}, \cl{L}) \)-flexible, but since
\( n_{0} \) is fixed, for the sum to be large, at least one of the terms \( z_{i} \) has to be large,
and the lemma claims that all tiled reals \( J \) of such form will necessarily be
\( (\cl{\eta}, \cl{L}') \)-flexible with respect to a larger sequence \( \cl{L}' \).

\begin{proof}
  We start with the case \( n_{0} = 2 \), and therefore \( J = z_{1} + y_{1} + z_{2} \).  For
  \( i \le m_{0} \) we put \( L_{i}' = L_{i} \) and define \( L_{k}' \) by induction as follows.
  Suppose \( L_{k-1}' \) has been defined.  Let \( n_{1} \ge k \) be so large that
  \begin{itemize}
  \item \( L_{n_{1}} \ge L_{k-1}' \);
  \item for any tileable \( x \ge L_{n_{1}} \) and any tileable \( \tilde{z}
    \le K + L_{k} \) if \( x \) is \( \eta_{n_{1}} \)-close to \( \rho \), then
     \begin{equation}
    \label{eq:freq-small-change}
    | \freqa(x + \tilde{z}) - \rho | \le \eta_{k}.
  \end{equation}
    The possibility to choose such \( n_{1} \) is based on Lemma
    \ref{lem:frequency-length-estimate}.
  \end{itemize}
 
  Set \( L_{k}' = 4L_{n_{1}} + K \).

  We now check that the sequence \( \cl{L}' \) constructed this way satisfies the conclusions of the
  lemma for \( n_{0} = 2 \).  Suppose we are given \( z_{1} \), \( y_{1} \), and \( z_{2} \).  Since
  \( z_{1}, y_{1}, z_{2} \in SB_{\eta_{k}}[L_{k}] \), \( k \le m_{0} \), it is immediate to see that
  \( J \in SB_{\eta_{k}}[L_{k}'] \) for all \( k \le m_{0} \).  We show that for \( k > m_{0} \) one
  has \( J \in SB_{\eta_{k}}[L_{k}'] \) whenever \( J \ge L_{k}' \).  If \( J \ge L_{k}' \), then
  \( z_{1} \ge 2L_{n_{1}} \) or \( z_{2} \ge 2L_{n_{1}} \).  Suppose for definiteness that
  \( z_{1} \ge 2L_{n_{1}} \).  Since \( z_{1} \) is \( (\cl{\eta},\cl{L}) \)-tiled this implies that
  \( z_{1} = \sum_{i=1}^{r_{1}} x_{i} \) with \( x_{i} \in B_{\eta_{n_{1}}}[L_{n_{1}}] \).  Note
  that \( x_{i} \in B_{\eta_{k}}[L_{k}'] \) for all \( i \), because \( L_{n_{1}} \le L_{k}' \) and
  \( k \le n_{1} \).  We now have two cases.

  \textbf{Case 1:} \( z_{2} < L_{k} \).  In this case let \( j \) be the largest index \( j \le
  r_{1} \) such that
  \[ x_{j} + x_{j+1} + \cdots + x_{r_{1}} \ge L_{n_{1}}. \]
  and note that since \( x_{i} \le L_{n_{1}} \) we have to have
  \[ L_{n_{1}} \le x_{j} + x_{j+1} + \cdots + x_{r_{1}} \le 2L_{n_{1}}. \]
  Since \( 2L_{n_{1}} + K + L_{k} \le L_{k}' \), using \( y_{1} + z_{2} \le K + L_{k} \),
  \[
  x_{j} + \cdots + x_{r_{1}} \in SB_{\eta_{n_{1}}}[L_{n_{1}}], \]
  and equation \eqref{eq:freq-small-change} we get
  \[ x_{j} + x_{j+1} + \cdots + x_{r_{1}} + y_{1} + z_{2} \in B_{\eta_{k}}[L_{k}'] \]
  and therefore \( J \in SB_{\eta_{k}}[L_{k}'] \) is witnessed by the decomposition
  \[ J = x_{1} + x_{2} + \cdots + x_{j-1} + (x_{j} + x_{j+1} + \cdots + x_{r_{1}} + y_{1} + z_{1}) \]
  in which every summand is an element of \( B_{\eta_{k}}[L_{k}'] \).  

  \textbf{Case 2:} \( z_{2} \ge L_{k} \).  There is a decomposition
  \( z_{2} = \sum_{i=1}^{r_{2}}x_{i}' \) with \( x_{i}' \in B_{\eta_{k}}[L_{k}] \) as \( z_{2} \) is
  \( (\cl{\eta}, \cl{L}) \)-tiled.  We therefore have \( x_{i}' \in B_{\eta_{k}}[L_{k}'] \) and if,
  as in previous case, \( j \le r_{1} \) is the largest index such that
  \[ x_{j} + x_{j+1} + \cdots + x_{r_{1}} \ge L_{n_{1}},  \]
  then
  \(  x_{j} + x_{j+1} + \cdots + x_{r_{1}} + y_{1} \in B_{\eta_{k}}[L_{k}'] \) by
  \eqref{eq:freq-small-change}.  Thus
  \[ J = x_{1} + x_{2} + \cdots + x_{j-1} + (x_{j} + x_{j+1} + \cdots + x_{r_{1}} + y_{1}) + x_{1}'
  + \cdots + x_{r_{2}}' \]
  with each summand being an element of \( B_{\eta_{k}}[L_{k}'] \).  This proves the lemma for
  \( n_{0} = 2 \).

  For \( n_{0} > 2 \) the lemma follows easily by induction. Suppose the lemma has been proved for
  \( n_{0} - 1 \) and the sequence \( \cl{L}' \) has been constructed.  Apply this lemma to
  \( \cl{L}' \) and \( n_{0} = 2 \) to get
  \( \cl{L}'' = \cl{L}'_{\Lem \ref{lem:tiled-sum}}(\cl{\eta}, \cl{L}', m_{0}, 2, K)
  \).  We claim that \( \cl{L}'' \) works for \( n_{0} \).  Indeed, suppose \( z_{i}, y_{i}
  \) are given.  By inductive assumption
  \[ J' = z_{1} + y_{1} + z_{2} + y_{2} + \cdots + y_{n_{0}-2} + z_{n_{0}-1}\]
  is \( (\cl{\eta}, \cl{L}') \)-tiled.  Note that \( z_{n_{0}} \) by assumption is \( (\cl{\eta},
  \cl{L}) \)-tiled and is therefore also \( (\cl{\eta}, \cl{L}') \)-tiled.  By the choice of \(
  \cl{L}'' \), \( J = J' + y_{n_{0} -1} + z_{n_{0}} \) must be \( (\cl{\eta}, \cl{L}'') \)-tiled and
  the lemma follows.
\end{proof}

Note that \( z_{i} = 0 \) is allowed in this lemma.  The case \( z_{n_{0}} = 0 \) is used in the
proof of the next lemma.

\begin{remark}
  \label{rem:improving-lemma-for-other-form-of-J}
  A simple observation is that if \( \cl{L}' \) satisfies the conclusion of Lemma
  \ref{lem:tiled-sum}, then any larger sequence (which also starts with \( L_{k} \),
  \( k \le m_{0} \)) will do so as well.  This implies the following immediate strengthening: given
  \( \cl{\eta} \), \( \cl{L} \), \( m_{0} \), \( n_{0} \), and \( K \) as above there exists
  \( \cl{L}' \), \( L_{k}' = L_{k} \) for \( k \le m_{0} \), such that all elements of the form
  \begin{displaymath}
       z_{1} + y_{1} + z_{2} + y_{2} + \cdots + y_{\tilde{n}-1} + z_{\tilde{n}}, 
  \end{displaymath}
  for \emph{all} \( \tilde{n} \le n_{0} \) are \( (\cl{\eta}, \cl{L}') \)-tiled.
\end{remark}

\begin{definition}
  \label{def:near-rho}
  We say that a tileable \( z \in \tileable \) is \emph{\( N \)-near \( \rho \)}, where \( N \in
  \mathbb{N} \), if for \( z = p\alpha + q\beta \) one of the two possibilities holds:
  \begin{itemize}
  \item \( \freqa(z) \le \rho \) and
    \( \freqa(z + N\alpha) = \frac{p + N_{0}}{p + q + N_{0}} \ge \rho \);
  \item \( \freqa(z) \ge \rho \) and
    \( \freqa(z + N\beta) = \frac{p}{p + q + N_{0}} \le \rho \).
  \end{itemize}
  In plain words, by adding \( N \) tiles of the right type one may flip the \( \alpha \)-frequency
  to the other side of \( \rho \).  Note that if \( x \) is \( N \)-near \( \rho \) and \( y \) is
  \( N' \)-near \( \rho \), then \( x+y \) is \( N+N' \)-near \( \rho \).

  We say that a tileable \( z = p\alpha + q\beta \) is \emph{\( N \)-far} from \( \rho \) if
  \begin{itemize}
  \item \( \freqa(z) \le \rho \) implies \( q \ge N \) and \( \freqa(z - N\beta) \le \rho \);
  \item \( \freqa(z) \ge \rho \) implies \( p \ge N \) and \( \freqa(z - N\alpha) \ge \rho \).
  \end{itemize}

  Note that if \( z \) is \( N \)-far from \( \rho \) and \( y \) is \( N \)-near \( \rho \), then 
  \[ \freqa(z + y) \le \rho \iff \freqa(z) \le \rho. \]
\end{definition}

The next lemma encapsulates the step of induction in Theorem \ref{thm:tiling-sparse-flows}.

\begin{lemma}[Sparse induction step]
  \label{lem:sparse-induction-step}
  Let \( \cl{L} \)  be an
  \( \cl{\eta} \)-flexible sequence. Given \( \epsilon > 0 \), \( m_{0} \in \mathbb{N} \), and
  \( N \in \mathbb{N} \) there exist a real number
  \( K = K_{\Lem \ref{lem:sparse-induction-step}}(\cl{\eta}, \cl{L}, \epsilon, m_{0}, N) \) and a
  sequence
  \( \cl{L}' = \cl{L}'_{\Lem \ref{lem:sparse-induction-step}}(\cl{\eta}, \cl{L}, \epsilon, m_{0}, N)
  \) such that for
  \( N' = N'_{\Lem \ref{lem:sparse-induction-step}}(\cl{\eta}, \cl{L}, \epsilon, m_{0}, N) = \lfloor
  \frac{K+1}{\alpha} \rfloor \) one has: \hfill
  \begin{itemize}
  \item \( L_{k}' = L_{k} \)  for all \( k \le m_{0} \);
  \item \( L_{k}' \ge L_{k} \) for all \( k \);
  \end{itemize}
  for any \( n \ge 2 \), any family \( z_{1}, \ldots, z_{n} \in \tiled \),
  \( y_{1}, \ldots, y_{n-1} \in \mathbb{R}^{>0} \) satisfying
  \begin{itemize}
  \item \( z_{i} \) are \( (\cl{\eta}, \cl{L}) \)-tiled;
  \item \( z_{i} \in SB_{\eta_{k}}[L_{k}] \) for \( k \le m_{0} \);
  \item \( z_{i} \) are \( N \)-near \( \rho \);
  \item \( K/2 \le y_{i} \le K \);
  \end{itemize}
  there are tiled \( \tilde{y}_{1}, \ldots, \tilde{y}_{n-1} \in \tiled \) such that
  \begin{enumerate}[(i)]
  \item\label{item:eps-close-sum}
    \( \Bigl| \sum_{i=1}^{r}( \tilde{y}_{i} - y_{i}) \Bigr| < \epsilon \) for all \( r \le n-1 \);
  \item \( \tilde{y}_{i} \in \bigcap_{j=0}^{m_{0}} SB_{\eta_{j}}[L_{j}] \);
  \item\label{item:y-are-tiled} \( \tilde{y}_{i} \) are
    \( (\cl{\eta}, \cl{L}') \)-tiled
  \end{enumerate}
  and setting
  \begin{displaymath}
    \tilde{J}= z_{1} + \tilde{y}_{1} + z_{2} + \tilde{y}_{2} + \cdots + z_{n-1} + \tilde{y}_{n-1} +
    z_{n}
  \end{displaymath}
  one also has
  \begin{enumerate}[(i)]
    \setcounter{enumi}{3}
  \item\label{item:main-item} \( \tilde{J} \) is \( (\cl{\eta}, \cl{L}') \)-tiled;
  \item\label{item:in-SB-eta-m0-plus-one} \( \tilde{J} \in SB_{\eta_{k}}[L_{k}'] \) for \( k \le
    m_{0} + 1 \);
  \item\label{item:N1-near} \( \tilde{J} \) is \( N' \)-near \( \rho \).
  \end{enumerate}
\end{lemma}
The set up of this lemma differs from the one of Lemma \ref{lem:tiled-sum} in the following aspects.
Reals \( y_{i} \) are not necessarily tileable.  In the context of Figure
\ref{fig:sparse-induction-figure}, we perturb each \( y_{i} \) into
\( \tilde{y}_{i} \) in a way that results in shifting each \( z_{i} \) by no more than
\( \epsilon \).
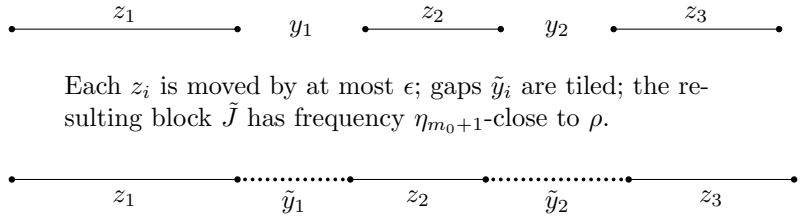
\begin{figure}[h]
  \centering
  \begin{tikzpicture}
    \foreach \x in {0, 3, 4.7, 6.5, 8.0, 10.2} {
      \filldraw (\x, 0) circle (1pt);
    }
    \draw (0,0) -- (3,0);
    \draw (4.7,0) -- (6.5,0);
    \draw (8,0) -- (10.2,0);
    \draw (1.5,0.2) node {\( z_{1} \)};
    \draw (5.6,0.2) node {\( z_{2} \)};
    \draw (9.1,0.2) node {\( z_{3} \)};
    \draw (3.85,0) node {\( y_{1} \)};
    \draw (7.25,0) node {\( y_{2} \)};
    \foreach \x in {0, 3, 4.5, 6.3, 8.2, 10.4} {
      \filldraw (\x, -2) circle (1pt);
    }
    \draw (0,-2) -- (3,-2);
    \draw (4.5,-2) -- (6.3,-2);
    \draw (8.2,-2) -- (10.4,-2);
    \draw (1.5,-2.25) node {\( z_{1} \)};
    \draw (5.4,-2.25) node {\( z_{2} \)};
    \draw (9.3,-2.25) node {\( z_{3} \)};
    \draw (5.2,-1) node[text width=9cm] {Each \( z_{i} \) is moved by at most \( \epsilon \); gaps
      \( \tilde{y}_{i} \) are tiled;  the resulting block \( \tilde{J} \) has frequency \(
      \eta_{m_{0}+1} \)-close to \( \rho \).};
    \foreach \x in {3.1, 3.2, ..., 4.4, 6.4, 6.5, ..., 8.2} {
      \filldraw (\x, -2) circle (0.5pt);
    }
    \draw (3.75,-2.3) node {\( \tilde{y}_{1} \)};
    \draw (7.25,-2.3) node {\( \tilde{y}_{2} \)};
  \end{tikzpicture}
  \caption{Constructing \( \tilde{J} \).}
  \label{fig:sparse-induction-figure}
\end{figure}
Similarly to Lemma \ref{lem:tiled-sum} the sums \( \tilde{J} \) are claimed
to be \( (\cl{\eta}, \cl{L}') \)-tiled, but the crucial difference is that this time we do not fix
the number of summands \( n \), but the assumption on \( z_{i} \)'s is stronger: they are
additionally assumed to be \( N \)-near \( \rho \) for a fixed natural \( N \).  Moreover, each \(
z_{i} \) may only be \( \eta_{m_{0}} \)-close to \( \rho \), while the resulting \( \tilde{J} \) is
at least \( \eta_{m_{0}+1} \)-close to \( \rho \), so the approximation of the \( \alpha
\)-frequency improves.  Also note that \( \tilde{J} \) satisfies assumptions similar to those
imposed on \( z_{i} \) which will allow us to continue the process inductively.

\begin{proof}
  Since \( \cl{L} \) is \( \cl{\eta} \)-flexible, we may take \( K \) so large that both sets of
  \begin{equation}
    \label{eq:y-plusmin-in-SB}
    y^{-}\mkern-12mu, \mkern 8mu y^{+} \in \bigcap_{j=0}^{m_{0}}SB_{\eta_{j}}[L_{j}]
  \end{equation}
  satisfying the following items are \( \epsilon \)-dense in \( [K/2-1, \infty) \):
  \begin{itemize}
  \item \( \freqa(y^{-}) < \rho \), \( \freqa(y^{+}) > \rho \);
  \item \( y^{-} \) and \( y^{+} \) are \( 2N \)-far from \( \rho \).
  \item \( |\freqa(y^{-}) - \rho| < \eta_{m_{0}+1}/2 \) and
    \( |\freqa(y^{+}) - \rho| < \eta_{m_{0}+1}/2 \).
  \end{itemize}
  By enlarging \( K \) if necessary we may assume that \( K \) is so big that
  \begin{equation}
    \label{eq:y-plus-small-x-m0-plus-one-close-to-rho}
    |\freqa(y^{-} + x) - \rho| < \eta_{m_{0}+1}\quad \textrm{ and }\quad |\freqa(y^{+} + x) - \rho| <
  \eta_{m_{0}+1}
  \end{equation}
  holds for all tileable \( x \le 2L_{m_{0}+1} \) and all \( y^{-}, y^{+} \ge K/2 - 1 \) as above. 

  We now describe the construction \( \tilde{y}_{i} \).  Since \( z_{1} \) and \( z_{2} \) are
  \( N \)-near \( \rho \), we may find
  \( \tilde{y}_{1} \in \bigcap_{j=0}^{m_{0}} SB_{\eta_{j}}[L_{j}] \) such that
  \begin{itemize}
  \item \( \tilde{y}_{1} \) is \( 2N \)-far from \( \rho \);
  \item if \( \freqa(z_{1} + z_{2}) \le \rho \) we want \( \freqa(\tilde{y}_{1}) > \rho \); if
    \( \freqa(z_{1} + z_{2}) > \rho \) we require \( \freqa(\tilde{y}_{1}) < \rho \);
  \item \( |y_{1} - \tilde{y}_{1}| < \epsilon \);
  \item \( |\freqa(\tilde{y}_{1}) - \rho| < \eta_{m_{0}+1}/2 \).
  \end{itemize}
  Note that \( z_{1}+\tilde{y}_{1}+z_{2} \) is necessarily \( N' \)-near \( \rho \).  We now
  construct \( \tilde{y}_{k} \) inductively; if \( \tilde{y}_{1}, \ldots, \tilde{y}_{k-1} \) have
  been constructed, we may find \( \tilde{y}_{k} \in \bigcap_{j=0}^{m_{0}}SB_{\eta_{j}}[L_{j}] \)
  such that
  \begin{itemize}
  \item \( |\tilde{y}_{k} - y_{k}| < \epsilon \);
  \item \( |\freqa(\tilde{y}_{k}) - \rho| < \eta_{m_{0}+1}/2 \);
  \item \( \tilde{y}_{k} \ge y_{k} \) if \( \sum_{j=1}^{k-1}\tilde{y}_{j} \le \sum_{j=1}^{k-1}y_{j} \) and
    \( \tilde{y}_{k} < y_{k} \) if \( \sum_{j=1}^{k-1}\tilde{y}_{j} > \sum_{j=1}^{k-1}y_{j} \);
  \item if 
    \[ \freqa\Bigl( \sum_{j=1}^{k-1}(z_{j} + \tilde{y}_{j}) + z_{k} + z_{k+1} \Bigr) \le \rho, \]
    then we pick \( \freqa(\tilde{y}_{k}) > \rho \);
    otherwise we ensure \( \freqa(\tilde{y}_{k}) < \rho \).
  \end{itemize}
  We still need to construct the sequence \( \cl{L}' \), but note that items
  \eqref{item:eps-close-sum} and \eqref{item:N1-near} do not depend on \( \cl{L}' \) and are
  satisfied by the construction of \( \tilde{y}_{i} \).  Moreover, item \eqref{item:y-are-tiled}
  will be trivially satisfied once we ensure that \( L_{m_{0}+1}' > K+1 \ge \tilde{y}_{i} \).  Also,
  it is easy to see that \eqref{item:in-SB-eta-m0-plus-one} is automatic once
  \( L_{m_{0}+1}' > 2L_{m_{0}+1} + K+1 \).  Indeed, \( \tilde{J} \in SB_{\eta_{j}}[L_{j}'] \) for
  \( j \le m_{0} \) since \( z_{i}, \tilde{y}_{i} \in SB_{\eta_{j}}[L_{j}'] \).  By the assumptions
  on the choice of \( \tilde{y}_{i} \), we have
  \( | \freqa(\tilde{y}_{i}) - \rho| < \eta_{m_{0}+1} \).  If \( z_{i} \ge L_{m_{0}+1} \), then
  \( z_{i} \in SB_{\eta_{m_{0}+1}}[L_{m_{0}+1}] \) because \( z_{i} \) is assumed to be
  \( (\cl{\eta}, \cl{L}) \)-tiled.  If, on the other hand, \( z_{i} < L_{m_{0}+1} \), then
  \( | \freqa(z_{i} + \tilde{y}_{i}) - \rho | < \eta_{m_{0}+1} \) by
  \eqref{eq:y-plus-small-x-m0-plus-one-close-to-rho}.  So, in either case
  \( z_{i} + \tilde{y}_{i} \in SB_{\eta_{m_{0}+1}}[L_{m_{0}+1}'] \) for all \( i \), and also
  \( z_{n-1} + \tilde{y}_{n-1} + z_{n} \in SB_{\eta_{m_{0}+1}}[L_{m_{0}+1}'] \).  This shows that
  \( \tilde{J} \in SB_{\eta_{m_{0}+1}}[L_{m_{0}+1}'] \) as \eqref{item:in-SB-eta-m0-plus-one} claims.

  It therefore remains to construct \( \cl{L}' \) which will ensure satisfaction of
  \eqref{item:main-item}.  For \( k \le m_{0} \) we set \( L_{k}' = L_{k} \) and define \( L_{k}' \)
  for \( k > m_{0} \) inductively.
  
  Note that for any \( 1 \le j_{1} < j_{2} \le n_{0} \) we have
  \( \sum_{i=j_{1}}^{j_{2}} z_{i} + \tilde{y}_{i} \) is \( 2N' \)-near \( \rho \), therefore given
  \( k > m_{0} \) and using \( \tilde{y}_{i} \ge K/2 -1 \) we can find \( n(k) \) such that
  \[ \Bigl| \freqa \bigl( \sum_{i=j_{1}}^{j_{2}-1} z_{i} + \tilde{y}_{i} \bigr) - \rho \Bigr| \le
  \eta_{k} \quad \textrm{whenever } j_{2} - j_{1} \ge n(k).\]
  We apply Lemma \ref{lem:tiled-sum} together with Remark
  \ref{rem:improving-lemma-for-other-form-of-J} for given \( \cl{\eta}, \cl{L}, m_{0} \) and
  \( 2n(k)+1, K+2 \).  Let \( \cl{L}'' \) be the output of this application and set
  \[ L_{m_{0}+1}' = \max\bigl\{ 2L_{m_{0}+1}+K+2, L_{m_{0}+1}''\bigr\} \textrm{ and } L_{k}' =
  \max\{L_{k-1}', L_{k}''\} \textrm{ for } k > m_{0}+1. \]
  We show that for the sequence \( \cl{L}' \) all elements of the form
  \( \tilde{J} \) are \( (\cl{\eta}, \cl{L}') \)-tiled.  Take \( k > m_{0} \) and consider
  \[ \tilde{J} = z_{1} + \tilde{y}_{1} + z_{2} + \tilde{y}_{2} + \cdots + z_{n-1} + \tilde{y}_{n-1}
  + z_{n}. \]
  Assume that \( \tilde{J} \ge L_{k} \), we aim at showing that
  \( \tilde{J} \in SB_{\eta_{k}}[L_{k}'] \).  If \( n \le 2n(k)+1 \), then this follow immediately
  from the choice of \( \cl{L}'' \) give by Lemma \ref{lem:tiled-sum}.  So, we may assume that
  \( n > 2n(k)+1 \).  In this case, we may find
  \[ 1 = j_{0} < j_{1} < j_{2} < \cdots < j_{s-1} < j_{s} = n \]
  such that \( n(k) \le j_{i+1} - j_{i} \le 2n(k) \) for all \( 0 \le i < s \).  Consider
  elements 
  \[ \tilde{J}_{i} = \sum_{l = j_{i}}^{j_{i+1}-1} (z_{l} + \tilde{y}_{l}) \textrm{ for } 0 \le i < s-1 \quad
  \textrm{and} \quad \tilde{J}_{s-1} = \sum_{l = j_{s-1}}^{j_{s}-1} (z_{l} + \tilde{y}_{l}) + z_{n}. \]
  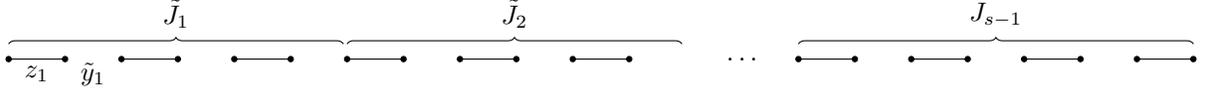
\begin{figure}[htb]
    \centering
    \begin{tikzpicture}
      \foreach \x in {0, 1.5, ..., 7.5, 10.5, 12, ..., 15}{
        \filldraw(\x, 0) circle (1pt);
        \filldraw(\x + 0.75, 0) circle (1pt);
        \draw (\x, 0) -- (\x+0.75,0);
      }
      \draw (0.375, -0.2) node {\( z_{1} \)};
      \draw (1.125, -0.2) node {\( \tilde{y}_{1} \)};
      \draw (9.75,0) node {\( \ldots \)};
      \draw[
      decoration={
        brace,
        raise=0.2cm
      },
      decorate
      ] (0.00,0) -- (4.45,0) node [pos=0.5,anchor=south,yshift=0.3cm]{\( \tilde{J}_{1} \)};
      \draw[
      decoration={
        brace,
        raise=0.2cm
      },
      decorate
      ] (4.5,0) -- (8.95,0) node [pos=0.5,anchor=south,yshift=0.3cm]{\( \tilde{J}_{2} \)};
      \draw[
      decoration={
        brace,
        raise=0.2cm
      },
      decorate
      ] (10.5,0) -- (15.75,0) node [pos=0.5,anchor=south,yshift=0.3cm]{\( \tilde{J}_{s-1} \)};
      
    \end{tikzpicture}
    \caption{Partitioning \( \tilde{J} \) into pieces \( \tilde{J}_{i} \).}
    \label{fig:partition-into-Ji-pieces}
  \end{figure}
  In this notation we have
  \[ \tilde{J} = \tilde{J}_{0} + \tilde{J}_{1} + \cdots + \tilde{J}_{s-2} + \tilde{J}_{s-1}. \]
  By the choice of \( n(k) \), each element in the sum above is \( \eta_{k} \)-close to \( \rho \).
  We claim that all \( \tilde{J}_{i} \) are elements of \( SB_{\eta_{k}}[L_{k}'] \) implying that so
  is \( \tilde{J} \).  If \( \tilde{J}_{i} < L_{k}' \), then
  \( \tilde{J}_{i} \in B_{\eta_{k}}[L_{k}'] \) and we are done.  Otherwise, Lemma
  \ref{lem:tiled-sum} (together with Remark \ref{rem:improving-lemma-for-other-form-of-J})
  applies\footnote{For an application of Lemma \ref{lem:tiled-sum} to \( \tilde{J}_{i} \),
    \( i < s-1 \), we add an extra \( \tilde{z}_{j_{i+1}} = 0 \) at the end.} and by the choice of
  \( L_{k}' \) we get \( \tilde{J}_{i} \in SB_{\eta_{k}}[L_{k}'] \).
\end{proof}

\section{Regular cross sections of sparse flows}
\label{sec:rudolphs-method}

We now prove the main theorem under the additional assumption that the flow is sparse.

\begin{theorem}[Regular cross sections of sparse flows]
  \label{thm:tiling-sparse-flows}
  Let \( \mff \) be a free sparse Borel flow on a standard Borel space \( \Omega \).  There is a
  Borel \( \{\alpha, \beta\} \)-regular cross section \( \mathcal{C} \) such that moreover for
  any \( \eta > 0 \) there exists \( N(\eta) \) such that for all \( x \in \mathcal{C} \) and
  \( n \ge N(\eta) \)
  \begin{equation}
    \label{eq:uniform-density-sparse}
    \biggl|\,\rho - \frac{1}{n}
    \sum_{k=0}^{n-1}\chi_{\mathcal{C}_{\alpha}}\bigl(\phi_{\mathcal{C}}^{k}(x)\bigr)\,\biggr| <
    \eta.
  \end{equation}
\end{theorem}

\begin{proof}
  Let \( \epsilon_{n} = 2^{-n} \cdot \frac{\min\{\alpha,1\}}{3} \), \( n \ge 1 \), which will serve
  as bounds on the size of jumps at step \( n \), and let \( \eta_{n} \) be any strictly decreasing
  positive sequence converging to zero such that \( \eta_{0} = 1 \) and
  \( \eta_{1} < \min\{\rho, 1-\rho\} \).  Note that 
  \[ \sum_{n=1}^{\infty} \epsilon_{n} \le \alpha/3, \]
  so if we start with two points which are at least \( \alpha \) apart, and if each one is moved by at
  most \( \sum \epsilon_{n} \), no two points will be ``glued together.''

  Based on Lemma \ref{lem:sparse-induction-step}, we construct reals \( K_{n} \), naturals
  \( N_{n} \), and sequences \( \cl{L}^{(n)} \), \( n \in \mathbb{N} \), as follows.  To start, pick
  any \( \cl{\eta} \)-flexible sequence \( \cl{L}^{(0)} \), which exists by Lemma
  \ref{lem:regular-tilings}, pick \( K_{0} \ge 4 \cdot \max\{1,\beta\} \) such that
  \( SB_{\eta_{0}}[L_{0}^{(0)}] \) is \( \epsilon_{1} \)-dense in \( [K_{0} - 2, \infty) \) and put
  \( N_{0} = 1 \).  Define
  \begin{displaymath}
    \begin{aligned}
      K_{n+1} &= \max\bigl\{ K_{\Lem \ref{lem:sparse-induction-step}} (\cl{\eta}, \cl{L}^{(n)}, \epsilon_{n+1},
      n, N_{n}),\, K_{n} + 4\bigr\},\\
      \cl{L}^{(n+1)} & = \cl{L}'_{\Lem \ref{lem:sparse-induction-step}}(\cl{\eta},
      \cl{L}^{(n)}, \epsilon_{n+1}, n, N_{n}),\\
      N_{n+1} &= N'_{\Lem \ref{lem:sparse-induction-step}}(\cl{\eta},
      \cl{L}^{(n)}, \epsilon_{n+1}, n, N_{n}).\\
    \end{aligned}
  \end{displaymath}
  Note that by construction \( L_{k}^{(n)} = L_{k}^{(n+1)} \) for all \( k \le n \).  We construct
  Borel cross sections \( \mathcal{C}_{n} \), and Borel function
  \( h_{n+1} :\mathcal{C}_{n} \to (-\epsilon_{n},\epsilon_{n}) \), which will represent shifts of
  points in \( \mathcal{C}_{n} \).

  We start with an application of Lemma \ref{lem:at-least-two-classes-Kn} which gives a sparse
  cross section \( \mathcal{C}_{0} \) such that
  \begin{enumerate}[(i)]
  \item \( \rgap[\mathcal{C}_{0}](x) > K_{0} \) for all \( x \in \mathcal{C}_{0} \).
  \item Each \( \oer[\le K_{n+1}]{\mathcal{C}_{0}} \)-class consists of at least two \( \oer[\le
    K_{n}]{\mathcal{C}_{0}} \)-classes for all \( n \in \mathbb{N} \).
  \item\label{item:gap-size-around-average} Distance between adjacent
    \( \oer[\le K_{n}]{\mathcal{C}_{0}} \)-classes within a given
    \( \oer[\le K_{n+1}]{\mathcal{C}_{0}} \)-class is \( 1 \)-close to \( (K_{n+1} + K_{n})/2 \).
    Since \( K_{n+1} \ge K_{n} + 4 \), and since \( \sum \epsilon_{n} \le 1/3 \), even if each point
    of \( \mathcal{C}_{0} \) is perturbed by at most \( \sum \epsilon_{n} \), no two
    \( \oer[\le K_{n}]{\mathcal{C}_{0}} \)-classes will be ``glued'', nor any
    \( \oer[\le K_{n}]{\mathcal{C}_{0}} \)-class will ``split'' into two or more classes.  This
    guarantees that during our construction the structure of \( \oer[\le K_{n}]{} \)-classes will
    remain intact.
  \end{enumerate}

  {\bf Step 1:} Constructing \( \mathcal{C}_{1} \).  Consider an \( \oer[\le K_{1}]{\mathcal{C}_{0}}
  \)-class.  It consists of a number of points \( x_{1} < x_{2} < \cdots < x_{n} \) such that 
  \[ K_{0} < \dist(x_{i}, x_{i+1}) \le K_{1} \quad \textrm{for all } i < n.\]
  As was explained in Subsection \ref{sec:sparse-case}, we shall move each \( x_{i} \) by at most \(
  \epsilon_{1} \).  This is done via an application of Lemma \ref{lem:sparse-induction-step}.
  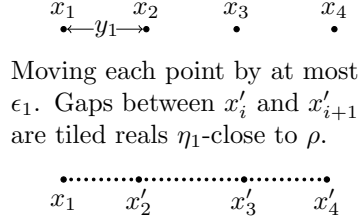
\begin{figure}[ht]
    \centering
    \begin{tikzpicture}
      \foreach \p/\x in {1/0, 2/1.1, 3/2.3, 4/3.6} {
        \filldraw (\x, 0) circle (1pt);
        \draw (\x, 0.25) node {\( x_{\p} \)};
      }
      \draw (0.575,0) node {\( y_{1} \)};
      \draw[very thin, ->] (0.7,0) -- (1.05,0);
      \draw[very thin, <-] (0.05,0) -- (0.4,0);
      \draw (1.8,-1) node[text width=5cm] {
        Moving each point by at most \( \epsilon_{1} \).
        Gaps between \( x_{i}' \) and \( x_{i+1}' \) are tiled
        reals \( \eta_{1} \)-close to \( \rho \).
    };
    \filldraw (0, -2) circle (1pt);
    \draw (0, -2.3) node {\( x_{1} \)};
    \foreach \p/\x in {2/1.0, 3/2.4, 4/3.5} {
      \filldraw (\x, -2) circle (1pt);
      \draw (\x, -2.3) node {\( x'_{\p} \)};
    }
    \foreach \x in {0.0, 0.1, ..., 3.5} {
      \filldraw (\x,-2) circle (0.5pt);
    }      
    \end{tikzpicture}
    \caption{Constructing \( \mathcal{C}_{1} \).  A single
      \( \mathsf{E}^{\le K_{1}}_{\mathcal{C}_{0}} \)-class is shown.}
    \label{fig:constructing-C1}
  \end{figure}
  In the context of this lemma, \( y_{i} = \dist(x_{i}, x_{i+1}) \) correspond to sizes of the
  gaps, and \( z_{i} = 0 \) correspond to sizes of \( \oer[\le K_{0}]{\mathcal{C}_{0}} \)-classes,
  which are single points, thus have zero length and are \( N_{0} \)-near \( \rho \).  By property
  \eqref{item:gap-size-around-average} of the cross section \( \mathcal{C}_{0} \),
  \[ \frac{K_{1}}{2} \le \frac{K_{0}+K_{1}}{2} - 1 \le y_{i} \le \frac{K_{0}+K_{1}}{2} + 1 \le K_{1}
  \quad \textrm{for all } i < n, \]
  and also assumptions on \( z_{i} \) in Lemma \ref{lem:sparse-induction-step} are trivially
  satisfied, so this lemma, and we move as it prescribes each
  \( x_{i} \) to \( x_{i}' \) and tile the gaps \( \dist(x_{i}', x_{i}') \).  The process is
  depicted in Figure \ref{fig:constructing-C1}.  Since the argument of Lemma
  \ref{lem:sparse-induction-step} describes an effective algorithm of performing such a tiling, this
  procedure can be run in a Borel way over all \( \oer[\le K_{1}]{\mathcal{C}_{0}} \)-classes
  resulting in a cross section \( \mathcal{C}_{1} \).

  To summarize, \( \mathcal{C}_{1} \) has points of two kinds.  It contains points from
  \( \mathcal{C}_{0} \) shifted by at most \( \epsilon_{1} \); the shift is given by a Borel
  function \( h_{1} : \mathcal{C}_{0} \to (-\epsilon_{1}, \epsilon_{1}) \) such that
  \[ \mathcal{C}_{0} + h_{1} := \bigl\{x + h_{1}(x) : x \in \mathcal{C}_{0}\bigr\} \subseteq
  \mathcal{C}_{1}. \]
  All other points in \( \mathcal{C}_{1} \) were added during ``tiling the gaps'' procedure inside
  \( \oer[\le K_{1}]{\mathcal{C}_{1}} \)-classes.  Note that \( \oer[\alpha,\beta]{\mathcal{C}_{1}}
  \)-classes are in one-to-one correspondence with \( \oer[\le K_{1}]{\mathcal{C}_{0}} \)-classes.

  The conclusion of Lemma \ref{lem:sparse-induction-step} guarantees the following:
  \begin{itemize}
  \item Any such class is necessarily \( (\cl{\eta}, \cl{L}^{(1)}) \)-tiled (by
    \ref{lem:sparse-induction-step}\eqref{item:main-item}).
  \item \( \oer[\alpha,\beta]{\mathcal{C}_{1}} \)-classes are \( N_{1} \)-near \( \rho \) (by
    \ref{lem:sparse-induction-step}\eqref{item:N1-near}).
  \item Every \( \oer[\alpha,\beta]{\mathcal{C}_{1}} \)-class is in
    \( SB_{\eta_{i}}[L_{i}^{(1)}] \) for \( i = 0,1 \) (by
    \ref{lem:sparse-induction-step}\eqref{item:in-SB-eta-m0-plus-one}); and is, in particular, \(
    \eta_{1} \)-close to \( \rho \).
  \end{itemize}
  First of all, note that these items form the set of with the assumptions on elements \( z_{i} \)
  in Lemma \ref{lem:sparse-induction-step}.  During the next step of the construction, when
  \( \mathcal{C}_{2} \) will be defined, \( \oer[\alpha,\beta]{\mathcal{C}_{1}} \)-classes will play
  the role of elements \( z_{i} \) in that lemma.  Also, for the last item we need to fix a witness,
  i.e., we may fix a Borel sub cross section \( \mathcal{D}_{1}^{(1)} \subseteq \mathcal{C}_{1} \)
  such that
  \begin{itemize}
  \item \( \rgap[\mathcal{D}_{1}^{(1)}](x) \in B_{\eta_{1}}[L_{1}^{(1)}] \) for all \( x \in
    \mathcal{D}_{1}^{(1)} \) such that \( x \, \oer[\alpha,\beta]{\mathcal{C}_{1}}\,
    \phi_{\mathcal{D}_{1}^{(1)}}(x) \);
  \item \( \min[x]_{\oer[\alpha,\beta]{\mathcal{C}_{1}}} \in \mathcal{D}_{1}^{(1)} \) and
    \( \max[x]_{\oer[\alpha,\beta]{\mathcal{C}_{1}}} \in \mathcal{D}_{1}^{(1)} \) for all
    \( x \in \mathcal{C}_{1} \).
  \end{itemize}
  Geometrically \( \mathcal{D}_{1}^{(1)} \) gives a partition of each
  \( \oer[\alpha,\beta]{\mathcal{C}_{1}} \)-class into pieces of length at most
  \( L_{1}^{(1)} \) each having \( \alpha \)-frequency \( \eta_{1} \)-close to \( \rho \).
  This finishes the first step of the construction.

  {\bf Step 2:} Constructing \( \mathcal{C}_{2} \).  Consider the relation
  \( \oer[\le K_{2}]{\mathcal{C}_{1}} \).  Each \( \oer[\le K_{2}]{\mathcal{C}_{1}} \)-class
  consists of at least two \( \oer[\alpha,\beta]{\mathcal{C}_{1}} \)-classes with gaps \( y_{i} \)
  between them being \( 2 \)-close\footnote{It was \( 1 \)-close in \( \mathcal{C}_{0} \), but we
    have moved points during the construction of \( \mathcal{C}_{1} \) by \( \epsilon_{1} \).} to
  \( (K_{2} + K_{1})/2 \) (see Figure \ref{fig:constructing-C2}).
  \begin{figure}[ht]
  \centering
  \begin{tikzpicture}
    \foreach \x in {0, 1.6, 3.3, 5.0, 6.6, 7.6} { \filldraw (\x, 0) circle (1pt); }
    \foreach \x in {0.0, 0.1, ..., 1.6,3.3, 3.4, ...,5.0,6.6,6.7, ...,7.6} {
      \filldraw (\x, 0) circle (0.5pt);
    }
    \draw (0.8, 0.2) node {\( z_{1} \)};
    \draw (4.15, 0.2) node {\( z_{2} \)};
    \draw (7.1, 0.2) node {\( z_{3} \)};
    \draw (2.45, 0) node {\( y_{1} \)};
    \draw (5.8, 0) node {\( y_{2} \)};
    \draw (3.8, -1) node[text width=7cm] {
      Moving each \( \oer[\alpha,\beta]{\mathcal{C}_{1}} \)-class \( z_{i} \) by at
      most \( \epsilon_{2} \).
      Gaps are tiled by reals from \( \bigcap_{j=0}^{1} SB_{\eta_{j}}[L_{j}^{(1)}] \).
    };
    \foreach \x in {0, 1.6, 3.2, 4.9, 6.8, 7.8} { \filldraw (\x, -2) circle (1pt); }
    \foreach \x in {0.0, 0.1, ..., 7.8} {\filldraw (\x, -2) circle (0.5pt);}
    \draw (2.4,-2.3) node {\( \tilde{y}_{1} \)};
    \draw (5.85,-2.3) node {\( \tilde{y}_{2} \)};
  \end{tikzpicture}

  \caption{Constructing \( \mathcal{C}_{2} \).}
  \label{fig:constructing-C2}
\end{figure}
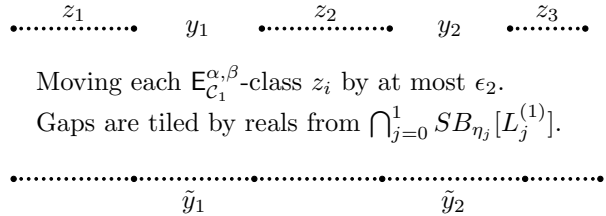
  Note that
  \( K_{2} \ge K_{1} + 4 \) implies
  \[ K_{2}/2 \le \frac{K_{1} + K_{2}}{2} - 2 \le y_{i} \le \frac{K_{1} + K_{2}}{2} + 2 \le K_{2}. \]
  We set \( z_{i} \)'s to be the tiled reals that correspond to
  \( \oer[\alpha,\beta]{\mathcal{C}_{1}} \)-classes.  By the choice of \( K_{2} \),
  \( \cl{L}^{(2)} \), and \( N_{2} \) given by Lemma \ref{lem:sparse-induction-step},
  one may move each
  \( \oer[\alpha,\beta]{\mathcal{C}_{1}} \)-class by at most \( \epsilon_{2} \) and tile the gaps by
  tiled reals from \( \bigcap_{j=0}^{1}SB_{\eta_{j}}[L_{j}^{(1)}] \) as
  shown in Figure \ref{fig:constructing-C2}.  This defines the cross sections \( \mathcal{C}_{2} \)
  which consists of shifted points from \( \mathcal{C}_{1} \) and newly added points in between
  \( \oer[\alpha,\beta]{\mathcal{C}_{1}} \)-classes.  Let
  \( h_{2} : \mathcal{C}_{1} \to (-\epsilon_{2}, \epsilon_{2}) \) be the shift function:
  \[ \mathcal{C}_{1} + h_{2} := \bigl\{x + h_{2}(x) : x \in \mathcal{C}_{1}\bigr\} \subseteq
  \mathcal{C}_{2}. \]
  Note that \( h_{2} \) is constant on \( \oer[\alpha,\beta]{\mathcal{C}_{1}} \)-classes.

  Again, the conclusion of Lemma \ref{lem:sparse-induction-step} ensures that each
  \( \oer[\alpha,\beta]{\mathcal{C}_{2}} \)-class satisfies the necessary conditions for another
  round of application of the same lemma.  Moreover, it ensures that each
  \( \oer[\alpha,\beta]{\mathcal{C}_{2}} \)-class is an element of \( SB_{\eta_{j}}[L_{j}^{(2)}] \) for
  \( j \le 2 \).  We now define sub cross sections
  \( \mathcal{D}_{j}^{(2)} \subseteq \mathcal{C}_{2} \), \( j = 1,2 \) as follows.  For \( j=1 \) we
  start with \( \mathcal{D}_{1}^{(1)} \) which partitions
  \( \oer[\alpha,\beta]{\mathcal{C}_{1}} \)-classes.  The cross section \( \mathcal{D}_{1}^{(2)} \)
  will consist of points of two types.  First it will contain the copy of
  \( \mathcal{D}_{1}^{(1)} \) inside \( \mathcal{C}_{2} \), i.e., it will have points
  \[ \mathcal{D}_{1}^{(1)} + h_{2} := \Bigl\{ x + h_{2}(x) : x \in \mathcal{D}_{1}^{(1)} \Bigr\}. \]
  Recall that in the conclusion of Lemma \ref{lem:sparse-induction-step} tiled reals \( \tilde{y}_{i} \)
  are taken from \( SB_{\eta_{j}}[L'_{j}] \) for \( j \le m_{0} \), i.e., in our case the gaps
  between the \( \oer[\alpha,\beta]{\mathcal{C}_{1}} \)-classes inside \( \mathcal{C}_{2} \) are
  elements of \( SB_{\eta_{1}}[L^{(1)}_{1}] \), hence each such gap can be partitioned into pieces
  \( B_{\eta_{1}}[L_{1}^{(1)}] \).  We include all such partition points into \(
  \mathcal{D}_{1}^{(2)} \) .  This results in \( \mathcal{D}_{1}^{(2)} \) partitioning each \(
  \oer[\alpha,\beta]{\mathcal{C}_{2}} \)-classes into pieces of size at most \( L_{1}^{(1)} \)
  each having \( \alpha \)-frequency \( \eta_{1} \)-close to \( \rho \).  Note that (modulo the \(
  h_{2} \)-shift) \( \mathcal{D}_{1}^{(2)} \) extends \( \mathcal{D}_{1}^{(1)} \).

  The remaining sub cross section \( \mathcal{D}_{2}^{(2)} \) is defined based on
  \ref{lem:sparse-induction-step}\eqref{item:in-SB-eta-m0-plus-one}.  More precisely, each
  \( \oer[\alpha,\beta]{\mathcal{C}_{2}} \)-class is guaranteed to be an element of
  \( SB_{\eta_{2}}[L^{(2)}_{2}] \), and we let \( \mathcal{D}_{2}^{(2)} \) to represent a partition
  of each such class into pieces of size at most \( L_{2}^{(2)} \) each having
  \( \alpha \)-frequency \( \eta_{2} \)-close to \( \rho \).  This concludes the second step of the
  construction.

  {\bf Step \( k + 1 \):} Constructing \( \mathcal{C}_{k+1} \) and \( \mathcal{D}_{j}^{(k+1)} \),
  \( j \le k + 1\) from \( \mathcal{C}_{k} \) and \( \mathcal{D}_{j}^{(k)} \), \( j \le k \).  The
  construction continues in a similar fashion as in step 2, and produces a cross section
  \( \mathcal{C}_{k+1} \) and a Borel shift function
  \( h_{k + 1} : \mathcal{C}_{k} \to (-\epsilon_{k+1}, \epsilon_{k+1}) \) such that
  \[ \mathcal{C}_{k} + h_{k+1} := \bigl\{x + h_{k+1}(x) : x \in \mathcal{C}_{k}\bigr\} \subseteq
  \mathcal{C}_{k+1}. \]
  The shift function is constant on \( \oer[\alpha,\beta]{\mathcal{C}_{k}} \)-classes.  These data
  is produced via the construction of Lemma \ref{lem:sparse-induction-step} for the parameters
  \( \epsilon = \epsilon_{k+1} \), \( m_{0} = k \), \( N = N_{k} \), \( \cl{L} = \cl{L}^{(k)} \).
  Specifically \( m_{0} = k \) ensures \( L^{(k+1)}_{j} = L^{(k)}_{j} \) for \( j \le k \), and the
  gaps between \( \oer[\alpha,\beta]{\mathcal{C}_{k}} \)-classes inside
  \( \oer[\alpha,\beta]{\mathcal{C}_{k+1}} \)-classes represent tiled reals from
  \( \bigcap_{j=0}^{m_{0}} SB_{\eta_{j}}[L^{(k)}_{j}] \).  This allows us to ``extend'' each
  \( \mathcal{D}_{j}^{(k)} \) to \( \mathcal{D}_{j}^{(k+1)} \), \( j \le k \).  Finally
  \( \mathcal{D}_{k+1}^{(k+1)} \) is constructed based on the conclusion
  \ref{lem:sparse-induction-step}\eqref{item:in-SB-eta-m0-plus-one} that each
  \( \oer[\alpha,\beta]{\mathcal{C}_{k+1}} \)-class is an element of
  \( SB_{\eta_{m_{0}+1}}[L^{(k+1)}_{k+1}] \).  Sub cross sections \( \mathcal{D}_{j}^{(k+1)} \) will
  satisfy:
  \begin{itemize}
  \item \( \rgap[\mathcal{D}_{j}^{(k+1)}](x) \in B_{\eta_{j}}[L_{j}^{(k+1)}] \) for all \( x \in
    \mathcal{D}_{j}^{(k+1)} \) such that \( x \, \oer[\alpha,\beta]{\mathcal{C}_{k+1}}\,
    \phi_{\mathcal{D}_{j}^{(k+1)}}(x) \);
  \item \( \min[x]_{\oer[\alpha,\beta]{\mathcal{C}_{k+1}}} \in \mathcal{D}_{j}^{(k+1)} \) and
    \( \max[x]_{\oer[\alpha,\beta]{\mathcal{C}_{k+1}}} \in \mathcal{D}_{j}^{(k+1)} \) for all
    \( x \in \mathcal{C}_{k+1} \) and \( j \le k+1 \).
  \end{itemize}

  \medskip

  Once cross sections \( \mathcal{C}_{n} \) have been constructed, the desired 
  section \( \mathcal{C} \) is defined to be the ``limit'' of \( \mathcal{C}_{n} \).  More formally,
  let \( f_{n,n+1} : \mathcal{C}_{n} \to \mathcal{C}_{n+1} \) be given by
  \( f_{n,n+1}(x) = x + h_{n+1}(x) \), set 
  \[ f_{m,n} = f_{n-1,n} \circ f_{n-2,n-1} \circ \cdots \circ f_{m+1, m+2} \circ f_{m,m+1}\]
  to be the embedding \( \mathcal{C}_{m} \to \mathcal{C}_{n} \) for \( m \le n \) with the natural
  agreement that \( f_{m,m} \) is the identity map.  Define
  \[ H_{n} : \mathcal{C}_{n} \to \Bigl(-\mkern-12mu \sum_{k=n+1}^{\infty}\mkern-8mu \epsilon_{k},
  \sum_{k=n+1}^{\infty}\mkern-8mu \epsilon_{k}\Bigr)\] to be given by
  \[ H_{n}(x) = \sum_{k=n}^{\infty}h_{k+1}\bigl(f_{n,k}(x)\big). \] 
  The function \( H_{n} \) is just the ``total shift'' of each point in \( \mathcal{C}_{n} \).
  Note that 
  \[ \mathcal{C}_{m} + H_{m} \subseteq \mathcal{C}_{n} + H_{n} \quad \textrm{for all } m \le n.\]
  The limit cross section \( \mathcal{C} \) is defined to be the (increasing) union
  \[ \mathcal{C} = \bigcup_{k} \bigl( \mathcal{C}_{k} + H_{k} \bigr) := \bigcup_{k} \bigl\{\, x +
  H_{k}(x) : x \in \mathcal{C}_{k}\,\bigr\}. \]
  It is immediate from the construction that \( \mathcal{C} \) is an \( \{\alpha,\beta\} \)-regular
  cross section.

  The moreover part of the theorem follows from the properties of cross sections
  \( \mathcal{D}_{j}^{(k)} \).  First, we set for \( j \ge 1 \)
  \[ \mathcal{D}_{j} = \bigcup_{k \ge j} \bigl( \mathcal{D}_{j}^{(k)} + H_{k} \bigr).\]
  For each \( j \) the cross section \( \mathcal{D}_{j} \) partitions \( \mathcal{C} \) into pieces
  of size at most \( L_{j}^{(j)} \) each piece having \( \alpha \)-frequency \( \eta_{j} \)-close to
  \( \rho \).  Since \( \eta_{j} \to 0 \) as \( j \to \infty \), Lemma
  \ref{lem:equivalent-formulation-of-regularity} implies the moreover part of the theorem. 
\end{proof}

\section{Battle for every last orbit}
\label{sec:tiling-general-flows}

In this section we finally prove that a free Borel flow can always be \( \{\alpha,\beta\} \)-tiled.

\begin{theorem}[Main Theorem]
  \label{thm:tiling-general-flows}
  Let \( \mff \) be a free Borel flow on a standard Borel space \( \Omega \).  There is an
  \( \{\alpha, \beta\} \)-regular cross section \( \mathcal{C} \) such that moreover
  for any \( \eta > 0 \) there exists
    \( N(\eta) \) such that for all \( x \in \mathcal{C} \) and
    \( n \ge N(\eta) \)
    \begin{equation}
      \label{eq:main-uniform-density}
     \biggl|\,\rho - \frac{1}{n}
    \sum_{k=0}^{n-1}\chi_{\mathcal{C}_{\alpha}}\bigl(\phi_{\mathcal{C}}^{k}(x)\bigr)\,\biggr| < \eta.
    \end{equation}
\end{theorem}

\begin{proof}
  Let \( \mff \) be a free Borel flow on a standard Borel space \( \Omega \).  Not unlike the proof
  of Theorem \ref{thm:tiling-sparse-flows} we fix sequence 
  \[ (\epsilon_{n})_{n=1}^{\infty} = \Bigl(2^{-n} \cdot
  \frac{\min\{\alpha,1\}}{3}\Bigr)_{n=1}^{\infty} \]
  and \( (\eta_{n})_{n=0}^{\infty} \) is any strictly decreasing positive sequence converging to
  zero such that \( \eta_{0} = 1 \) and \( \eta_{1} < \min\{\rho, 1-\rho\} \).
  We also let \( \nu_{n}' = \eta_{n+1} + \frac{\eta_{n} - \eta_{n+1}}{3} \) and
  \( \nu_{n} = \eta_{n+1} + 2\frac{\eta_{n}-\eta_{n+1}}{3} \).  In our argument we shall need to
  take an interval strictly inside \( [\eta_{n+1}, \eta_{n}] \) and we are going to use
  \( [\nu'_{n}, \nu_{n}] \) for this purpose since 
  \[ \eta_{n+1} < \nu'_{n} < \nu_{n} < \eta_{n}. \]
  Similarly to the proof of Theorem \ref{thm:tiling-sparse-flows}, \( \epsilon_{n} \) controls the
  maximum shift of points, the total shift will therefore be bounded by
  \( \sum \epsilon_{n=1}^{\infty} \le 1/3 \).  We construct a sequence of Borel cross sections \(
  \mathcal{C}_{n} \) and the desired cross section \( \mathcal{C} \) will be defined as the
  ``limit'' of this sequence.

  Pick \( K_{0} \) so large that
  \begin{itemize}
  \item For any \( x \ge K_{0} \) sets \( \freqa^{-1}[\rho-\eta_{1},\rho-\eta_{2}] \) and
    \( \freqa^{-1}[\rho+\eta_{2},\rho+\eta_{1}] \) are \( \epsilon_{1}/6 \)-dense in
    \( \mathcal{U}_{\epsilon_{1}}(x) \).  The possibility to choose such \( K_{0} \) is guaranteed
    by Lemma \ref{lem:flexibitility-of-tileable-reals}.
    \item To avoid some possible collapses, we also make sure that \( K_{0} \) is large compared to \(
    1, \alpha \), and \( \beta \).  For instance, \( K_{0} > 2 \max\{1, \beta\} \) will work.
  \end{itemize}
  We let \( \mathcal{C}_{0} \) to be a cross section such that
  \( \rgap[\mathcal{C}_{0}](x) \in [K_{0}+1, K_{0}+2] \) for all \( x \in \mathcal{C}_{0} \), which
  exists by Corollary \ref{cor:wagh-gap-bounds}.  Each
  point will be shifted by at most \( \sum \epsilon_{n} \), therefore the distance between
  \( \mathcal{C}_{0} \)-neighbors will always be somewhere between \( K_{0} \) and \( K_{0} + 3 \).

  We now construct \( \mathcal{C}_{1} \).  For this we take a large natural number \( N_{1} \),
  precise bounds on which to be provided later.  In short, \( \mathcal{C}_{1} \) is constructed by
  selecting pairs of points in \( \mathcal{C}_{0} \) with at least \( N_{1} \)-many points between
  any two pairs, moving the right point within each pair by at most \( \epsilon_{1} \) so as to make
  the gap tileable, and adding points to tile the gaps, see Figure \ref{fig:Construction-C1}.
  \begin{figure}[ht]
    \centering
    \begin{tikzpicture}[scale=2]
      \foreach \x in {0,1,...,8}{
        \filldraw (\x,0) circle (0.75pt);
      }
      \draw (0.8,-0.2) rectangle (2.2,0.2);
      \draw (6.8,-0.2) rectangle (8.2,0.2);
      \draw[decoration={
        brace,
        raise=0.5cm
      },decorate]
      (2.8,0) -- (6.2,0) node[pos=0.5, anchor=south,yshift=0.6cm]
      {\( N_{1} \le \# \textrm{Points} \le 2N_{1}+1 \)};
      \draw (4,-0.6) node[text width=8cm]
      {Right points within each pair are shifted by \( < \epsilon_{1} \),
        and  gaps are tiled.};
      \foreach \x in {0,1,3,4,5,6,7}{
        \filldraw (\x,-1.2) circle (0.75pt);
      }
      \filldraw (2.2,-1.2) circle (0.75pt);
      \filldraw (7.8,-1.2) circle (0.75pt);
      \foreach \x in {1.2,1.4,...,2.0} {
        \filldraw (\x, -1.2) circle (0.5pt);
      }
      \foreach \x in {7.2,7.4,7.6} {
        \filldraw (\x, -1.2) circle (0.5pt);
      }
      \foreach \x in {1.6,7.4} {
        \draw (\x,-1.4) node {rank \( 1 \) block};
      }
    \end{tikzpicture}
    \caption{Constructing \( \mathcal{C}_{1} \) from \( \mathcal{C}_{0} \).}
    \label{fig:Construction-C1}
  \end{figure}
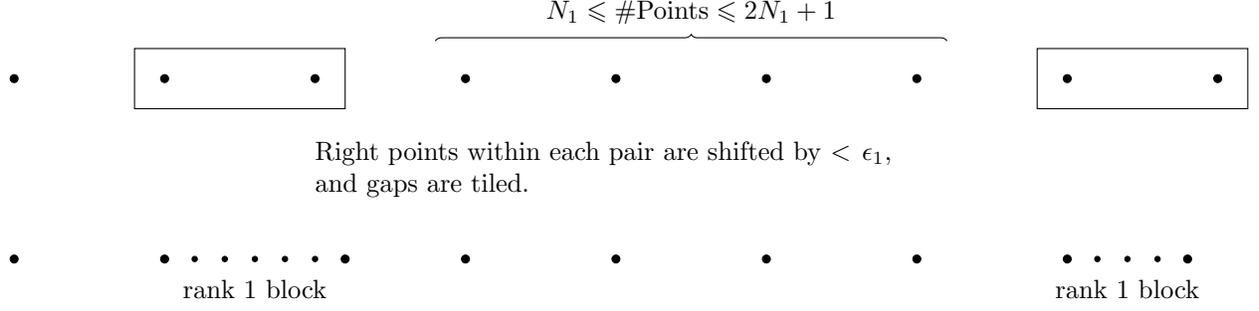

  In symbols, we select a sub cross section \( \mathcal{P} \subseteq \mathcal{C}_{0} \) such that
  \begin{itemize}
  \item \( \mathcal{P} \) consists of pairs: for any \( x \in \mathcal{P} \) exactly one of the
    two things happens \textemdash{} either \(
    \phi_{\mathcal{C}_{0}}(x) \in \mathcal{P} \) or \( \phi^{-1}_{\mathcal{C}_{0}}(x) \in \mathcal{P} \).
  \item Between any two pairs there are at least \( N_{1} \), at most \( 2N_{1}+1 \) points from
    \( \mathcal{C}_{0} \): if \( x \in \mathcal{P} \) is the right point of a pair, then
    \( \phi_{\mathcal{P}}(x) = \phi_{\mathcal{C}_{0}}^{n(x)}(x) \) with
    \( N_{1} \le n(x) \le 2N_{1}+1 \).
  \end{itemize}
  Given such \( \mathcal{P} \), we shift the right point of each pair to make the distance to
  the left one tileable.  That is for each right point \( x \in \mathcal{P} \) we find
  \( h_{1}(x) \in (-\epsilon_{1}, \epsilon_{1}) \) such that
  \[ \dist\bigl(\phi^{-1}_{\mathcal{C}_{0}}(x), x + h_{1}(x)\bigr) \in \freqa^{-1}[\rho - \eta_{1},
  \rho + \eta_{1}]. \]
  Such an \( h_{1} \) can be found by the choice of \( K_{0} \). 
  It is convenient to extend \( h_{1} \) to a function on \( \mathcal{C}_{0} \) by declaring
  \( h_{1}(x) = 0 \) whenever \( x \) is not a right point of a pair in \( \mathcal{P} \).  And we
  can make sure that the function \( h_{1} : \mathcal{\mathcal{C}}_{0} \to (-\epsilon_{1},
  \epsilon_{1}) \) is Borel.

  Consider the relation \( \oer[\alpha,\beta]{\mathcal{C}_{1}} \) on \( \mathcal{C}_{1} \).
  Equivalence classes of \( \oer[\alpha,\beta]{\mathcal{C}_{1}} \) are of two types: some of them
  consist of a single point, which is necessarily an element of \( \mathcal{C}_{0} \); others
  consist of two points from \( \mathcal{C}_{0} \) (one of which may be shifted) together with newly
  added points in the midst.  We shall refer to the latter as {\it blocks of rank \( 1 \)}.
  Isolated points from \( \mathcal{C}_{1} \) are in this sense {\it blocks of rank \( 0 \)}, see
  Figure \ref{fig:Construction-C1}.  Note that each block of rank \( 1 \) in \( \mathcal{C}_{1} \)
  corresponds to a tiled real of \( \alpha \)-frequency \( \eta_{1} \)-close to \( \rho \).

  This finishes the definition of \( \mathcal{C}_{1} \), but we still owe the reader a bound on
  \( N_{1} \).  For this we let \( D_{1} = K_{0} + 3 \), it will serve as an upper bound on the
  distance between points in \( \mathcal{C}_{1} \).  Let
  \( M_{1} = M_{\Lem \ref{lem:main-induction-lemma}}(D_{1}, \epsilon_{2}/6, \eta_{2}, \nu_{2},
  \nu_{2}') \).  {\it Our first requirement is \( N_{1} \ge M_{1} \).}  There will be one more largeness
  assumptions, but let us elaborate on this condition first.

  Take a segment between rank \( 1 \) blocks in \( \mathcal{C}_{1} \) as shown in Figure
  \ref{fig:prop-freedom}.  Let \( d_{1}, \ldots, d_{n} \) denote the gaps inside this segment, where
  \( N_{1} \le n \le 2N_{1} + 1 \),  and note that \( d_{k} \ge K_{0} \).  Let
  \[ R_{k} = \mathcal{T} \cap \mathcal{U}_{\epsilon_{1}}(d_{k}) \cap \freqa^{-1}[\rho-\eta_{1},
  \rho+\eta_{1}] \]
  denote the set of admissible shifts between the \( k \)th and \( k-1 \)th points when the
  former is allowed to be shifted by at most \( \epsilon_{1} \).
  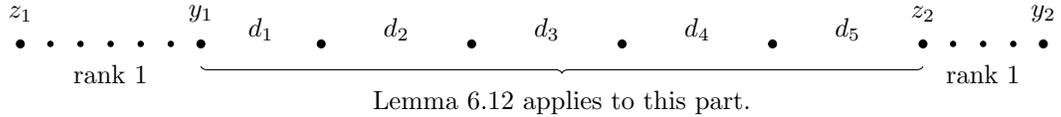
\begin{figure}[ht]
    \centering
    \begin{tikzpicture}[scale=2]
      \foreach \x/\pind in {1/-1,2.2/0,3/1,4/2,5/3,6/4,7/5,7.8/6}{
        \filldraw (\x,0) circle (0.75pt);
        \ifnumgreater{\pind}{0}{
          \ifnumgreater{6}{\pind}{
            \draw (\x/2 + \lastp/2,0.1) node {\( d_{\pind} \)};
          }{}
        }{}
        \xdef\lastp{\x}
      }
      \foreach \x in {1.2,1.4,...,2.0} {
        \filldraw (\x, 0) circle (0.5pt);
      }
      \foreach \x in {7.2,7.4,7.6} {
        \filldraw (\x, 0) circle (0.5pt);
      }
      \draw[decoration={
        brace,
        mirror,
        raise=0.3cm
      },decorate]
      (2.2,0) -- (7,0) node[pos=0.5, anchor=north, yshift=-0.5cm] {Lemma
        \ref{lem:main-induction-lemma} applies to this part.};
      \draw (1,0.2) node {\( z_{1} \)}
      (2.2,0.2) node {\( y_{1} \)}
      (7,0.2) node {\( z_{2} \)}
      (7.8,0.2) node {\( y_{2} \)};
      \draw (1.6,-0.2) node {rank \( 1 \)};
      \draw (7.4,-0.2) node {rank \( 1 \)};
    \end{tikzpicture}
    \caption{Propagation of freedom principle.}
    \label{fig:prop-freedom}
  \end{figure}
  Since \( n \ge M_{1} \), Lemma \ref{lem:main-induction-lemma}, applied relative to
  \( (d_{k}) \) and \( (R_{k}) \), implies that both
  \[ \mathcal{A}_{n}\bigl(\epsilon_{1}, (d_{k}), (R_{k}) \bigr) \cap
  \freqa^{-1}[\rho-\nu_{2},\rho-\nu_{2}'] \quad \textrm{and} \quad
  \mathcal{A}_{n}\bigl(\epsilon_{1}, (d_{k}), (R_{k}) \bigr) \cap
  \freqa^{-1}[\rho+\nu_{2}',\rho+\nu_{2}] \]
  are \( \epsilon_{2}/6 \)-dense in
  \( \mathcal{U}_{\epsilon_{1}/2}(\sum_{k} d_{k}) = \mathcal{U}_{\epsilon_{2}}(\sum_{k} d_{k}) \).
  In the notation of Figure \ref{fig:prop-freedom}, it means that one can move each rank \( 0 \)
  point by at most \( \epsilon_{1} \) and the rank \( 1 \) block \( [z_{2},y_{2}] \) by at most
  \( \epsilon_{2} \) so as to make all the gaps tileable.  Moreover, one has many possible ways of
  doing this with additional control on the \( \alpha \)-frequencies.  Set
  \[R = \mathcal{A}_{n}\bigl(\epsilon_{1}, (d_{k}), (R_{k}) \bigr) \cap
  \freqa^{-1}[\rho-\eta_{2},\rho + \eta_{2}] \cap \mathcal{U}_{\epsilon_{2}}\bigl( \sum_{k} d_{k}
  \bigr). \]
  This set will correspond to admissible ways of shifting rank \( 1 \) blocks during the next
  construction step. 

  It is perhaps helpful to list some properties of \( R \).
  \begin{itemize}
  \item \( R \subseteq \mathcal{U}_{\epsilon_{2}}\bigl( \sum_{k=1}^{n} d_{k} \bigr) \).
  \item Both \( R \cap \freqa^{-1}[\rho-\nu_{2},\rho-\nu_{2}'] \) and
    \( R \cap \freqa^{-1}[\rho+\nu_{2}',\rho+\nu_{2}] \) are \( \epsilon_{2}/6 \)-dense in
    \( \mathcal{U}_{\epsilon_{2}}\bigl( \sum_{k=1}^{n} d_{k} \bigr) \).
  \item \( R \) agrees with \( R_{k} \), \( k=1,\ldots, n \), in the sense that each \( a \in R \)
    can be decomposed as \( a = \sum_{k=1}^{n}b_{k} \) with \( b_{k} \in R_{k} \).
  \end{itemize}
  It is the last item that is particularly specific to this argument compared to the one in
  Theorem \ref{thm:tiling-sparse-flows}.  When dealing with sparse flows, we did not have any rank
  \( 0 \) blocks between rank \( 1 \) blocks, therefore to get such an \( R \) we only needed to
  know that the gap from \( y_{1} \) to \( z_{2} \) is large enough.  But now, we must consider
  only shifts that agree with blocks of lower rank in the midst.

  Note also that the above list of properties for \( R \) are precisely what Lemma
  \ref{lem:main-induction-lemma} is designed for.  Here is a tiny problem though.  The way Lemma
  \ref{lem:main-induction-lemma} is stated, it does not apply to blocks of points, so we shall
  need to pick representatives from each block.  These will be the left points of rank \( 1 \)
  blocks, e.g., points \( z_{1} \) and \( z_{2} \) in Figure \ref{fig:prop-freedom}.  Admissible
  tilings will therefore be given by \( \dist(z_{1},y_{1}) + R \) rather than just \( R \).  More
  formally, we
  define \( \mathcal{D}_{1} \subseteq \mathcal{C}_{1} \) to consists of left endpoints of rank \(
  1 \) blocks in \( \mathcal{C}_{1} \).  Note that the gaps in \( \mathcal{D}_{1} \) are bounded,
  \[ \dist(z_{1},z_{2}) \le \dist(z_{1},y_{1}) + (2N_{1}+2)D_{1} \le (2N_{1}+3)D_{1}. \]
  We set \( D_{2} = (2N_{1}+3)D_{1} \).  For an
  adjacent pair \( z_{1}, z_{2} \in \mathcal{D}_{1} \) we have the set of \emph{admissible shifts}
  \[ R(z_{1},z_{2}) = \dist(z_{1},y_{1}) + R. \]
  Lemma \ref{lem:main-induction-lemma} will be applied to segments of \( \mathcal{D}_{1} \).

  While adding \( \dist(z_{1},y_{1}) \) does not ruin the density part, and
  \( \dist(z_{1},y_{1}) + R \) is still \( \epsilon_{2}/6 \)-dense in
  \( \mathcal{U}_{\epsilon_{2}}\bigl(\dist(z_{1},z_{2}) \bigr) \), it may slightly alter the
  \( \alpha \)-frequencies of elements from \( R \).  This is where we use \( \nu_{2} \) and
  \( \nu_{2}' \).  Recall that they were selected to satisfy
  \( \eta_{3} < \nu_{2} < \nu_{2}' < \eta_{2} \).  So, if \( N_{1} \) is large enough (cf. Lemma
  \ref{lem:frequency-length-estimate}), then the \( \alpha \)-frequency of
  \( \dist(z_{1},y_{1}) + a \) differs from the \( \alpha \)-frequency of \( a \) by at most
  \( \min\{\eta_{2}-\nu_{2}, \nu_{2}'-\eta_{3}\} \) for any \( a \in R \).  In fact, since rank
  \( 2 \) interval will include both \( [z_{1},y_{1}] \) and \( [z_{2},y_{2}] \) segments, we want
  the combined influence of \( \dist(z_{1},y_{1}) \) and \( \dist(z_{2},y_{2}) \) on the
  \( \alpha \)-frequency of any \( a \in R \) to be small.  This brings us to the second largeness
  assumption on \( N_{1} \).  {\it The natural \( N_{1} \) is so large that for any \( a \in R \)
    and any tileable \( d \le 2D_{1} \) one has}
  \[ \bigl|\freqa(a + d) - \freqa(a)\bigr| < \min\{\eta_{2}-\nu_{2}, \nu_{2}'-\eta_{3}\}. \]
  Under this assumptions the set \( R(z_{1},z_{2}) \) satisfies:
  \begin{itemize}
  \item
    \( R(z_{1},z_{2}) \subseteq \mathcal{U}_{\epsilon_{2}}\bigl( \dist(z_{1},z_{2}) \bigr)
    \).
  \item Both \( R(z_{1},z_{2}) \cap \freqa^{-1}[\rho-\eta_{2},\rho-\eta_{3}] \) and
    \( R(z_{1},z_{2}) \cap \freqa^{-1}[\rho+\eta_{3},\rho+\eta_{2}] \) are \(
    \epsilon_{2}/6 \)-dense in \( \mathcal{U}_{\epsilon_{2}}\bigl(\dist(z_{1},z_{2}) \bigr) \).
  \item \( R(z_{1},z_{2}) \) agrees with sets \( R_{k} \) in the sense that each \( a \in
    R(z_{1},z_{2}) \) is of the form \( a = \dist(z_{1},y_{1}) + \sum_{k=1}^{n}b_{k} \) for
    some \( b_{k} \in R_{k} \).
  \end{itemize}
  It is the sets \( R(z_{1},z_{2}) \) that will play the role of \( R_{k} \) in the
  application of Lemma \ref{lem:main-induction-lemma} during the next stage, when we construct \(
  \mathcal{C}_{2} \).  The process depicted in Figure \ref{fig:Construction-C1} is run in Borel way
  over all orbits of the flow.

  We are now officially done with the second step of our construction, which can now be continued in
  a very similar fashion.  A cross section \( \mathcal{C}_{2} \) is constructed by selecting
  adjacent\footnote{Adjacent among rank \( 1 \) intervals, between such a pair there is, of course,
    a whole bunch of rank \( 0 \) points. } pairs of rank \( 1 \) blocks, shifting the right block
  in each pair by at most \( \epsilon_{2} \) according to some element of \( R(z_{1},z_{2}) \) and
  tiling the gaps.  Each \( \oer[\alpha,\beta]{\mathcal{C}_{2}} \)-class obtained this way is called
  a rank \( 2 \) block.  The pairs are selected in such a way that between any two of them there are
  at least \( N_{2} \), at most \( 2N_{2} + 1 \) rank \( 1 \) blocks, for large enough
  \( N_{2} \in \mathbb{N} \).  The process is depicted in Figure \ref{fig:constructing-C2-general}.
  \begin{figure}[ht]
    \centering
    \begin{tikzpicture}[scale=0.89]
      \foreach \x in {0,1,...,18}{
        \filldraw (\x,0) circle (1pt);
      }
      \foreach \x in {0.1,0.2,...,0.9,4.1,4.2,...,4.9,7.1,7.2,...,7.9,
        10.1,10.2,...,10.9,14.1,14.2,...,14.9,17.1,17.2,...,17.9}{
        \filldraw (\x,0) circle (0.3pt);
      }
      \draw (-0.2,-0.3) rectangle (5.2,0.3)
      (13.8,-0.3) rectangle (18.2,0.3);
      \draw (9,-1.5) node {Moving right blocks within each pair by \( < \epsilon_{2} \),
        points between blocks by \( <
        \epsilon_{1}\), and tiling the gaps.};
      \foreach \x in {0,1,2.2,3.1,3.9,4.9,6,7,...,15, 15.9,16.9,17.9}{
        \filldraw (\x,-3) circle (1pt);
      }
      \foreach \x in {0.1,0.2,...,4.9, 7.1,7.2,...,7.9,
        10.1,10.2,...,10.9,14, 14.1,...,17.9}{
        \filldraw (\x,-3) circle (0.3pt);
      }
      \draw (2.45,-3.4) node {rank \( 2 \) block}
      (15.95,-3.4) node {rank \( 2 \) block};
      \draw (7.5,-2.7) node {rank \( 1 \)}
      (10.5,-2.7) node {rank \( 1 \)};
      \draw[decoration={
        brace,
        mirror,
        raise=0.4cm
      },decorate] (5.5,-3) -- (13.5,-3) node[pos=0.5,anchor=north,yshift=-0.5cm] {\( N_{2} \le \#\
        \textrm{rank \( 1 \) blocks} \le 2N_2 +1 \)};
    \end{tikzpicture}
    \caption{Constructing \( \mathcal{C}_{2} \) out of \( \mathcal{C}_{1} \).}
    \label{fig:constructing-C2-general}
  \end{figure}
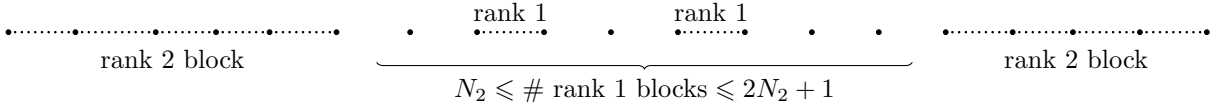
  The function
  \( h_{2} : \mathcal{C}_{1} \to (-\epsilon_{1}, \epsilon_{1}) \)
  represents shifts of points.  Note that if \( x \in \mathcal{C}_{1} \) is a member of a rank
  \( 1 \) block, then \( h_{2}(x) \in (-\epsilon_{2}, \epsilon_{2}) \).  The number \( N_{2} \) is
  assumed to be large in the same sense as \( N_{1} \):
  \begin{itemize}
  \item First of all
    \( N_{2} \ge M_{2} = M_{\Lem \ref{lem:main-induction-lemma}}(D_{2}, \epsilon_{3}/6, \eta_{3},
    \nu_{3}, \nu_{3}') \).  This by the same principle implies that given a pair of neighboring rank
    \( 2 \) blocks the right one can be shifted by at most \( \epsilon_{3} \), rank \( 1 \) blocks
    in between by no more than \( \epsilon_{2} \), and points of rank \( 0 \) by \( \epsilon_{1} \),
    and all this can be done in such a way that gaps become tileable and the \( \alpha \)-frequency
    of the whole segment becomes \( \eta_{2} \)-close to \( \rho \).  To each pair of adjacent rank
    \( 2 \) blocks we therefore may associate a set \( R \) of admissible shifts.
  \item Secondly, \( N_{2} \) is so large that the contribution of \( \alpha \)-frequency of a rank
    \( 2 \) block to the \( \alpha \)-frequency of any element in \( R \) is so small that the
    frequency of their sum is still in \( [\rho-\eta_{3}, \rho+\eta_{3}] \).
  \end{itemize}

  At the end of the day, we construct cross sections \( \mathcal{C}_{n} \), sub sections
  \( \mathcal{D}_{n} \subseteq \mathcal{C}_{n} \), and shift maps
  \( h_{n+1} : \mathcal{C}_{n} \to (-\epsilon_{1}, \epsilon_{1}) \) such that
  \begin{enumerate}[(i)]
  \item   \( x + h_{n+1}(x) \in \mathcal{C}_{n+1} \) for all \( x \in \mathcal{C}_{n} \).
  \item\label{item:constant-on-classes} \( h_{n+1} \) is constant on
    \( \oer[\alpha,\beta]{\mathcal{C}_{n}} \)-classes.
  \item\label{item:rank-k-epsilon-k-small} If \( x \in \mathcal{C}_{n} \) belongs to a rank \( k \)
    block, then \( h_{n+1}(x) \in (-\epsilon_{k+1}, \epsilon_{k+1}) \).
  \item\label{item:if-shift-max-rank} If \( h_{n+1}(x) \ne 0 \), then \( x + h_{n+1}(x) \) belongs
    to a rank \( n+1 \) block in \( \mathcal{C}_{n+1} \).
  \item Each orbit in \( \mathcal{C}_{n} \) has blocks of rank \( n \), and \( \mathcal{D}_{n}
    \) consists of left endpoints of rank \( n \) blocks.
  \item Cross sections \( \mathcal{D}_{n} \) have bounded gaps: \( \rgap[\mathcal{D}_{n}](x) \le
    D_{n} \) for all \( x \in \mathcal{D}_{n}  \).
  \item To every pair of adjacent points in \( \mathcal{D}_{n} \) we associate a set of admissible
    shifts which consists of tileable reals \( \eta_{n} \)-close to \( \rho \).
  \end{enumerate}

  Let \( f_{n,n+1} : \mathcal{C}_{n} \to \mathcal{C}_{n+1} \) be given by
  \( f_{n,n+1}(x) = x + h_{n+1}(x) \), set 
  \[ f_{m,n} = f_{n-1,n} \circ f_{n-2,n-1} \circ \cdots \circ f_{m,m+1}\]
  to be the embedding \( \mathcal{C}_{m} \to \mathcal{C}_{n} \) for \( m \le n \) with the natural
  agreement that \( f_{m,m} \) is the identity map.  With this notations we may define
  \[ H_{n} : \mathcal{C}_{n} \to \Bigl(-\sum_{k=1}^{\infty} \epsilon_{k},\sum_{k=1}^{\infty}
  \epsilon_{k}\Bigr) \quad \textrm{to be} \quad H_{n}(x) =
  \sum_{k=n}^{\infty}h_{k+1}\bigl(f_{n,k}(x)\big). \]
  Despite the fact that each \( h_{n+1} \) can for some points of \( \mathcal{C}_{n} \) be as large
  as \( \epsilon_{1} \), the sum converges.  This follows from items
  \eqref{item:rank-k-epsilon-k-small} and \eqref{item:if-shift-max-rank} above.

  The limit cross section \( \mathcal{C} \) is defined to be the union
  \[ \mathcal{C} = \bigcup_{n} \bigl\{\, x + H_{n}(x) : x \in \mathcal{C}_{n} \,\bigr\}. \]
  Note that this union is increasing by \eqref{item:constant-on-classes}.  It is easy to see that
  the \( \{\alpha,\beta\} \)-chain relation \( \oer[\alpha,\beta]{\mathcal{C}} \) on
  \( \mathcal{C} \) is the union of (shifted) \( \oer[\alpha,\beta]{\mathcal{C}_{n}} \)
  relations: \( x \, \oer[\alpha,\beta]{\mathcal{C}} \, y \) if and only if there is \( n \) such
  that \( \bigl( x-H_{n}(x) \bigr)\, \oer[\alpha,\beta]{\mathcal{C}_{n}}\, \bigl( y-H_{n}(y) \bigr) \).

  For \( \mathcal{C} \) to be an \( \{\alpha,\beta\} \)-regular cross section, the relation
  \( \oer[\alpha,\beta]{\mathcal{C}} \) must coincide with \( \oer{\mathcal{C}} \), which is not
  necessarily the case.  But \( \mathcal{C} \) has arbitrarily large regular blocks, which as we
  noted earlier in Subsection \ref{sec:big-picture} is enough.  We remind the reader, that all
  orbits fall into three categories.
  \begin{itemize}
  \item On some orbits \( \mathcal{C} \) is \( \{\alpha,\beta\} \)-regular, these
    are precisely the orbits on which \( \oer[\alpha,\beta]{\mathcal{C}} \)
    consists of a single equivalence class.
  \item On other orbits \( \oer[\alpha,\beta]{\mathcal{C}} \) may have at least
    two classes, one of which is infinite.
  \item Finally, all \( \oer[\alpha,\beta]{\mathcal{C}} \)-classes are finite one some
    orbits. 
  \end{itemize}
  The decomposition of the space into orbits of these types is Borel, and we may deal with each of
  them separately.  The restriction of the flow on the second type of orbits is smooth, since we
  may select finite endpoints of infinite classes for a Borel transversal.  The restriction of the
  flow \( \mff \) of the orbits of the third type is sparse.  The sparse cross sections is
  given by the endpoints of \( \oer[\alpha,\beta]{\mathcal{C}} \)-classes.  We may
  therefore employ Theorem \ref{thm:tiling-sparse-flows} and build the required cross section
  \( \mathcal{C} \) on this part of the space.  It is in this
  sense that our argument here is complementary to the one of Theorem
  \ref{thm:tiling-sparse-flows} \textemdash{} where the current method fails, the one from
  \ref{thm:tiling-sparse-flows} succeeds.

  We are finally left with the orbits on which \( \mathcal{C} \) is indeed an
  \( \{\alpha,\beta\} \)-regular cross section.  There is no loss in generality to assume that all
  the orbit are of this sort.  We need to verify \eqref{eq:main-uniform-density}.  During the
  construction of \( \mathcal{C}_{n} \), we also defined sub sections
  \( \mathcal{D}_{n} \subseteq \mathcal{C}_{n} \), which consisted of left endpoints of blocks of
  rank \( n \), as well as reals \( D_{n} \) such that all the gaps in \( \mathcal{D}_{n} \) are
  bounded by \( D_{n} \).  Let \( \widetilde{\mathcal{D}}_{n} \) be the sub section of
  \( \mathcal{C} \) that corresponds to \( \mathcal{D}_{n} \),
  \[ \widetilde{\mathcal{D}}_{n} = \{\, x +H_{n}(x) : x \in \mathcal{D}_{n} \,\}. \]
  Since gaps in \( \widetilde{\mathcal{D}}_{n} \) are still bounded, in view of of Lemma
  \ref{lem:equivalent-formulation-of-regularity} it is enough to show that all the gaps in
  \( \widetilde{\mathcal{D}}_{n} \) have \( \alpha \)-frequencies \( \eta_{n} \)-close to
  \( \rho \): 
  \[ \bigl| \freqa\bigl( \rgap[\widetilde{\mathcal{D}}_{n}](x) \bigr) - \rho \bigr| < \eta_{n}. \]
  Indeed, by construction for any adjacent \( z_{1}, z_{2} \in \mathcal{D}_{n} \), we selected a
  family \( R(z_{1}, z_{2}) \) of admissible shifts such that any \( a \in R(z_{1},z_{2}) \) has
  \( \alpha \)-frequency \( \eta_{n} \)-close to \( \rho \).  Since for any \( m \ge n \), shifts at
  stage \( m \) agree with admissible shifts at stage \( n \), it means that whenever
  \( z_{1},z_{2} \) are shifted, the \( \dist(z_{1},z_{2}) \) becomes tileable with
  \( \alpha \)-frequency \( \eta_{n} \)-close to \( \rho \).  Finally, since we are working on the
  part of the space, where all orbits of \( \mathcal{C} \) are \( \{\alpha,\beta\} \)-regular, it
  follows that all adjacent \( z_{1},z_{2} \) are tiled at some stage of the construction, 
  and the theorem follows.
\end{proof}

\section{Lebesgue orbit equivalence}
\label{sec:lebesg-orbit-equiv}

In this closing section we would like to offer an application of our Main Theorem to the
classification of Borel flows up to Lebesgue orbit equivalence.  But first we need to do some
preliminary work.

\subsection{Compressible equivalence relations}
\label{sec:compr-equiv-relat}

Let \( \oer{} \) be a countable Borel equivalence relation on a standard Borel space \( X \).  A
\emph{partial full group} of \( \oer{} \) is the set of all partial Borel injections which act within
\( \oer{} \)-classes:
\[ \fgr{\oer{}} = \bigl\{\, f : A \to B\, \bigm|\, A, B \subseteq X \textrm{ are Borel, } f \textrm{
  is a Borel bijection and } x\, \oer{}\, f(x) \textrm{ for all } x,y \in A \,\bigr\}. \]
For Borel \( A, B \subseteq X \) we let \( A \sim B \) to denote existence of a bijection
\( f : A \to B \), \( f \in \fgr{\oer{}} \).  Recall that \( \oer{} \) is said to be
\emph{compressible} if either of the following two equivalent conditions is satisfied:
\begin{itemize}
\item There exist a Borel set \( A \subseteq X \) such that \( X \sim A \) and \( X \setminus A \)
  intersects each \( \oer{} \)-class.
\item There are Borel \( A_{i} \subseteq X \), \( i \in \mathbb{N} \), such that \( X \sim A_{i} \)
  for all \( i \) and \( A_{i} \cap A_{j} = \es \) for \( i \ne j \).
\end{itemize}
A celebrated theorem of Nadkarni \cite{nadkarni_existence_1990} and Becker--Kechris 
\cite{becker_descriptive_1996}  gives a powerful characterization of compressible relations.
\begin{theorem-nn}[Becker--Kechris--Nadkarni]
  An aperiodic\footnote{An equivalence relation is aperiodic if each equivalence class is
    infinite.} countable Borel equivalence relation \( \oer{} \) is compressible if and only if it
  does not admit a finite invariant measure.
\end{theorem-nn}

\begin{definition}
  A subset \( A \subseteq X \) is said to be \( \oer{} \)-\emph{syndetic} (or just \emph{syndetic}
  when \( \oer{} \) is understood) if there exist \( n \in \mathbb{N} \) and Borel sets \( A_{i}
  \subseteq X \), \( 1 \le i \le n \), such that \( A \sim A_{i} \) and 
  \[ X = \bigcup_{i = 1}^{n} A_{i}. \]
\end{definition}
Here is a typical example of a syndetic set.  Suppose \( \oer{} \) is the orbit equivalence relation
of an aperiodic Borel automorphism \( T : X \to X \), \( \oer{} = \oer[T]{X} \).  Suppose
\( A \subseteq X \) has bounded gaps in the sense that there is \( N \in \mathbb{N} \) such that for
each \( x \in A \) there are \( 1 \le i,j \le N \) satisfying \( T^{i}(x) \in A \) and
\( T^{-j}(x) \in A \).  If \( A \) intersects each orbit of \( T \) then \( A \) is necessarily
\( \oer{} \)-syndetic as we may take \( A_{i} = T^{i}|_{A} \) for \( 0 \le i \le N \).

The concept of syndeticity allows for a convenient reformulation of compressibility.
\begin{proposition}
  \label{prop:syndetic-reformulation-of-compressibility}
  Let \( \oer{} \) be an aperiodic countable Borel equivalence relation on a standard Borel space
  \( X \).  It is compressible if and only if \( A \sim B \) for any two Borel syndetic
  sets \( A, B \subseteq X \).
\end{proposition}
\begin{proof}
  We begin by proving necessity.  Suppose \( \oer{} \) is compressible.  Since the set \( X \)
  itself is obviously syndetic and \( \sim \) is a symmetric and transitive relation, it is enough
  to show that \( A \sim X \) holds for any Borel syndetic set \( A \subseteq X \).  Pick a syndetic
  set \( A \subseteq X \) and let
  \( \oer{A} \) be the restriction of \( \oer{} \) onto \( A \):
  \[ \oer{A} = \oer{} \cap (A \times A). \]
  Our first goal is to show that \( \oer{A} \) is compressible.  Let \( f_{i} \in \fgr{\oer{}} \),
  \( 1 \le i \le n \), \( \dom(f_{i}) = A \), be such that \( X = \bigcup_{i=1}^{n} f_{i}(A) \).
  Set \( \tilde{A}_{1} = A \) and define \( \tilde{A}_{i} \) for \( i \le n \) inductively by setting
  \[ \tilde{A}_{i+1} = f_{i+1}^{-1}\Bigl( X \setminus \bigcup_{k=1}^{i} f_{k}(A) \Bigr). \]
  Sets \( \tilde{A}_{i} \subseteq A \) correspond to the partition of \( X \):
  \[ X = \bigsqcup_{i=1}^{n} f_{i}(\tilde{A}_{i}). \]
  If \( \mu \) is a finite \( \oer{A} \)-invariant measure on \( A \), then
  \( \sum_{i=1}^{n} (f_{i})_{*} \mu_{i} \) is a finite \( \oer{} \)-invariant measure on \( X \),
  where \( \mu_{i} = \mu|_{\tilde{A}_{i}} \).  By Becker--Kechris--Nadkarni Theorem the
  compressibility of \( \oer{} \) implies the compressibility of \( \oer{A} \).

  We may therefore pick Borel injections \( \tau_{i} \in \fgr{\oer{A}} \), \( i \in \mathbb{N} \),
  such that \( \tau_{i}(A) \cap \tau_{j}(A) = \es \) for \( i \ne j \).  Let \( h : X \to A \) be
  given by
  \[ h(x) = \tau_{i}\bigl( f_{i}^{-1}(x) \bigr) \quad \textrm{for } x \in f_{i}(\tilde{A}_{i}). \]
  Note that the map \( h : X \to A  \) is injective and witnesses \( X \sim h(X) \).
  On the other hand the identity map \( \mathrm{id} : A \to X \) is an injection in the other direction.
  The usual Schr\"oder--Bernstein argument produces a bijection \( \theta : A \to X \), \( \theta \in
  \fgr{\oer{}}\), thus proving \( A \sim X \) as required.

  It remains to show sufficiency.  Since \( X \) is necessarily syndetic, it is enough to construct
  a syndetic set \( B \subseteq X \) such that \( X \setminus B \) intersects each
  \( \oer{} \)-class; any \( f \in \fgr{\oer{}} \) satisfying \( f(X) = B \) will then witness
  compressibility of \( \oer{} \).  By the proof of Feldman--Moore Theorem
  \cite{feldman_ergodic_1977} we may pick a countable family \( h_{i} \in \fgr{\oer{}} \),
  \( i \in \mathbb{N} \), and Borel \( B_{i} \subseteq X \) such that \( \dom(h_{i}) = B_{i} \),
  \( B_{i} \cap h_{i}(B_{i}) = \es \), and
  \[ x\, \oer{}\, y \iff x = y \textrm{ or } \Bigl(\exists n\ x \in B_{n} \textrm{ and } h_{n}(x) =
  y\Bigr).\]
  Define inductively sets \( \tilde{B}_{n} \) by setting \( \tilde{B}_{0} = B_{0} \) and
  \[ \tilde{B}_{n+1} = \Bigl\{ x \in B_{n+1} : x, h_{n+1}(x) \not \in \bigcup_{i \le n}
  \bigl(\tilde{B}_{i} \cup h_{i}(\tilde{B}_{i}) \bigr)\Bigr\}. \]
  Let \( B = \bigcup_{i} \tilde{B}_{i} \) and set \( \theta : B \to X \) to be given by \( \theta(x) =
  h_{n}(x) \), where \( n \) is such that \( x \in \tilde{B}_{n} \) (note that such \( n \) is
  necessarily unique).  The map \( \theta \in \fgr{\oer{}} \), sets \( B \) and \( \theta(B) \) are
  disjoint and moreover, for each \( x \in X \) the set \( [x]_{\oer{}} \setminus \bigl( B \cup
  \theta(B) \bigr) \) has cardinality at most \( 1 \), for if 
  \[ y_{1}, y_{2} \in [x]_{\oer{}} \setminus \bigl( B \cup \theta(B) \bigr) \textrm{ are
    distinct}, \]
  then \( y_{1} \in \tilde{B}_{n} \) for \( n \) such that \( h_{n}(y_{1}) = y_{2} \), contradicting
  the definition of \( B \).  The restriction of \( \oer{} \) onto the set of \( x \in X \) where
  \[ \bigl| [x]_{\oer{}} \setminus \bigl( B \cup \theta(B) \bigr) \bigr| = 1 \]
  is smooth and we may modify the set \( B \) and the map \( \theta \) on this part to
  ensure that \( X = B \sqcup \theta(B) \).  The set \( B \) is thus \( \oer{} \)-syndetic and \( X
  \setminus B \) intersects every \( \oer{} \)-class as desired.
\end{proof}

\subsection{Matching equidense sets}
\label{sec:match-equid-sets}

In the previous subsection we worked in the context of countable Borel equivalence relations.  We
now turn to a more specific situation of orbit equivalence relations that arise from the actions of
\( \mathbb{Z} \).  In this case each orbit naturally inherits the linear order from \( \mathbb{Z} \).

Let \( T : X \to X \) be an aperiodic automorphism.  We say that a subset \( A \subseteq X \) has
\emph{uniform frequency} \( \rho \in [0,1] \) if for every \( \eta > 0 \) there exists \( N \) such that
for all \( n \ge N \) and all \( x \in X \)
\[ \bigl| \rho - \frac{1}{n}\sum_{k=0}^{n-1} \chi_{A}\bigl(T^{k}(x)\bigr) \bigr| < \eta. \]
Note that when \( \rho > 0 \), any set \( A \subseteq X \) of uniform frequency \( \rho \) is
necessarily \( \oer[T]{X} \)-syndetic.  Note also that the Main Theorem constructs an \(
\{\alpha,\beta\} \)-regular cross section for which sets of \( \alpha \)- and \(
\beta \)-points have uniform frequency \( \rho \) and \( 1-\rho \) respectively.  

\begin{theorem}
  \label{thm:frequency-rho-can-map}
  Let \( T : X \to X \) be an aperiodic Borel automorphism.  If \( A, B \subseteq X \) are sets of
  the same uniform frequency \( \rho \in (0,1] \), then \( A \sim B \).
\end{theorem}
\begin{proof}
  Define inductively sets \( A_{k} \) and \( B_{k} \), \( k \in \mathbb{N} \), by the formula
  \begin{displaymath}
    \begin{aligned}
      A_{k} &= \bigl\{x \in A \setminus \bigcup_{i < k} A_{i} : T^{k}(x) \in B \setminus \bigcup_{i
        < k} B_{i }\bigr\},\\
      B_{k} &= T^{k}(A_{k}).
    \end{aligned}
  \end{displaymath}
  Note that sets \( A_{k} \) are pairwise disjoint and \( A_{k} \subseteq A \) for all \( k \).
  Similarly, for the sets \( B_{k} \).  Set \( A_{\infty} = \bigcup_{k} A_{k} \) and \( B_{\infty} =
  \bigcup_{k} B_{k}\). Define \( \tilde{\theta} : A_{\infty} \to B_{\infty} \) by \( \tilde{\theta}(x) =
  T^{k}(x) \) whenever \( x \in A_{k} \).

  We claim that \( \mu(A \setminus A_{\infty}) = 0 \) for any \( T \)-invariant probability measure
  \( \mu \) on \( X \).  Indeed, let \( \mu \) be such a measure which we furthermore assume
  ergodic.  By Birkhoff's Ergodic Theorem \( \mu(A) = \rho = \mu(B) \).  Since
  \( \mu(A_{\infty}) = \mu(B_{\infty}) \), we get
  \( \mu(A \setminus A_{\infty}) = \mu(B \setminus B_{\infty}) \).  Suppose
  \( \mu(A \setminus A_{\infty}) > 0 \).  By ergodicity
  \( \mu\bigl(\bigcup_{k} T^{k}(A \setminus A_{\infty})\bigr) = 1\), and therefore
  \[ \Bigl(\bigcup_{k}T^{k}(A \setminus A_{\infty})\Bigr) \cap (B\setminus B_{\infty}) \ne \es. \]
  If \( x \in A \setminus A_{\infty} \) is such that \( T^{k}(x) \in B \setminus B_{\infty} \) for
  some \( k \in \mathbb{N} \), then \( x \in A_{k} \) contradicting \( x \not \in A_{\infty} \).  We
  conclude that \( \mu(A \setminus A_{\infty}) = 0 \).  By ergodic decomposition the same is true
  for all (not necessarily ergodic) invariant probability measures \( \mu \) on \( X \).

  Set \( Z = [A \setminus A_{\infty}]_{\oer[T]{X}} \cup [B \setminus B_{\infty}]_{\oer[T]{X}} \).  By
  above \( \mu(Z) = 0 \) for any \( T \)-invariant measure, and so the
  restriction of \( T \) onto \( Z \) is compressible.  Proposition
  \ref{prop:syndetic-reformulation-of-compressibility} applies and produces a map
  \( \bar{\theta} \in \fgr{T|_{Z}} \) such that \( \bar{\theta}(A \cap Z) = B \cap Z \).  The
  map \( \theta \) defined by
  \begin{displaymath}
    \theta(x) =
    \begin{cases}
      \tilde{\theta}(x) & \textrm{when } x \in A \setminus Z,\\
      \bar{\theta}(x)   & \textrm{when } x \in A \cap Z,\\
    \end{cases}
  \end{displaymath}
  witnesses \( A \sim B \).
\end{proof}

\subsection{Lebesgue orbit equivalence}
\label{sec:lebesg-orbit-equiv-subsec}

Recall that an \emph{orbit equivalence} between two group actions \( \Gamma_{1} \acts X_{1} \)
and \( \Gamma_{2} \acts X_{2} \) is a Borel bijection \( \phi : X_{1} \to X_{2} \) which sends
orbits onto orbits: \( \phi \bigl( \orbit[\Gamma_{1}](x) \bigr) =
\orbit[\Gamma_{2}]\bigl(\phi(x)\bigr) \) for all \( x \in X_{1} \).  In the language of equivalence
relations, two actions are orbit equivalent if and only if the orbit equivalence relations they
generate are isomorphic.  Actions of \( \mathbb{Z} \) give rise to \emph{hyperfinite} equivalence
relations, i.e., equivalence relations which can be written as an increasing union of finite
equivalence relations.  Hyperfinite relations have been classified by Dougherty, Jackson, and
Kechris in \cite{dougherty_structure_1994}.  The important part of their classification is as
follows.

\begin{theorem-nn}[Dougherty--Jackson--Kechris]
  \label{thm:djk}
  Non-smooth aperiodic hyperfinite equivalence relations are isomorphic if and only if they admit
  the same number of invariant probability ergodic measures. 
\end{theorem-nn}

When one considers actions of non-discrete groups, e.g., Borel
flows, the number of invariant ergodic probability measures is no longer an invariant of orbit
equivalence.  When actions under consideration are free, one can overcome this obstacle by
considering the notion of Lebesgue orbit equivalence.  For simplicity of the presentation we
restrict ourselves to the case of Borel flows.  For rigorous proofs of the following statements and
for a more general treatment see \cite{slutsky_lebesgue_2015}.

Let \( \mathfrak{F} \) be a free Borel flow on a standard Borel space \( X \).  Recall from
Subsection \ref{sec:cross sections-flows} that any orbit of the
action can be endowed with a copy of the Lebesgue measure:  for \( A \subseteq X \) let \(
\lambda_{x}(A) = \lambda(\{r \in \mathbb{R} : x + r \in A\}) \).  It is immediate to see that \(
\lambda_{x} = \lambda_{y} \)  for all \( y \in \orbit(x) \), i.e., the measure \( \lambda_{x} \)
depends only on the orbit of \( x \), but not on the \( x \) itself.

Suppose \( \mathfrak{F}_{1} \) and \( \mathfrak{F}_{2} \) are free Borel flows on \( X_{1} \) and
\( X_{2} \) respectively, and let \( \phi : X_{1}\to X_{2} \) be an orbit equivalence between these
flows.  We say that \( \phi \) is a \emph{Lebesgue orbit equivalence} (LOE for short) if \( \phi \)
is a Lebesgue measure preserving map when restricted onto any orbit:
\( \phi_{*} \lambda_{x} = \lambda_{\phi(x)} \).  Any LOE map preserves the number of invariant
ergodic probability measures.

In \cite{slutsky_lebesgue_2015} the analog of DJK classification has been proved for free Borel
actions of \( \mathbb{R}^{n} \).  Here we would like to show how the Main Theorem gives a simple
proof of this classification for the particular case of flows.  Recall the notation from Subsection
\ref{sec:invariant-measures}, where \( \invm(\mathfrak{F}) \) denotes the set of invariant ergodic
probability measures for the flow \( \mathfrak{F} \).

\begin{theorem}
  \label{thm:djk-for-flows}
  Let \( \mathfrak{F_{1}} \) and \( \mathfrak{F_{2}} \) be free non-smooth Borel flows.  These flows
  are LOE if and only if
  \[ \bigl| \invm(\mathfrak{F}_{1}) \bigr| = \bigl| \invm(\mathfrak{F}_{2}) \bigr|. \]
\end{theorem}

\begin{proof}
  As we mentioned earlier, \( \implies \) is an easy direction and the reader is referred to Section
  \( 4 \) of \cite{slutsky_lebesgue_2015}.  Suppose \( \mathfrak{F}_{1} \) and
  \( \mathfrak{F}_{2} \) have the same number of invariant measures.  By the Main Theorem each flow
  can be represented as a flow under a two-valued function, and moreover one may assume that gaps of
  each type occur with uniform frequency \( 1/2 \) within every orbit.  More formally, we may pick Borel
  \( \{\alpha, \beta\} \)-regular cross sections \( \mathcal{C}^{1} \subseteq X_{1} \) and
  \( \mathcal{C}^{2} \subseteq X_{2} \), \( \alpha \) and \( \beta \) are some positive
  rationally independent reals, say \( 1 \) and \( \sqrt{2} \), such that both
  \( \mathcal{C}^{i}_{\alpha} \) and \( \mathcal{C}^{i}_{\beta} \) occur with uniform frequency \( 1/2 \)
  within every orbit, \( i = 1, 2 \), where
  \[ \mathcal{C}_{\alpha}^{i} = \bigl\{ x \in \mathcal{C}^{i} : \rgap[\mathcal{C}^{i}](x) =
  \alpha\bigr\}. \]
  By Theorem \ref{thm:frequency-rho-can-map}, there are Borel bijections
  \( \theta^{i} : \mathcal{C}^{i}_{\alpha} \to \mathcal{C}^{i}_{\beta} \) which preserve the orbit
  equivalence relation: \( \theta^{i} \in \fgr{\oer{\mff_{i}}} \).

  Note that \( \mathcal{C}_{\alpha}^{i} \) have to have bounded gaps and in view of Theorem
  \ref{thm:correspondence-between-invariant-measures} the induced automorphisms
  \( \phi_{\mathcal{C}^{1}_{\alpha}} \) and \( \phi_{\mathcal{C}^{2}_{\alpha}} \) have the same
  number of invariant ergodic probability measure.  Therefore by the DJK classification there
  exists a Borel orbit equivalence
  \( \psi : \mathcal{C}^{1}_{\alpha} \to \mathcal{C}^{2}_{\alpha} \) between \(
  \oer{\mathcal{C}_{\alpha}^{1}} \) and \( \oer{\mathcal{C}_{\alpha}^{2}} \).  Maps \( \theta^{i} \) allow us
  to extend \( \psi \) to a an orbit equivalence \( \mathcal{C}^{1} \to \mathcal{C}^{2} \) by
  setting
  \[ \psi\bigl(\theta^{1}(x)\bigr) = \theta^{2}\bigl(\psi(x)\bigr)\quad \textrm{for all } x \in
  \mathcal{C}^{1}_{\alpha}.\]

  The map \( \psi : \mathcal{C}^{1} \to \mathcal{C}^{2} \) sends \( \alpha \)-points to \( \alpha
  \)-points and also \( \beta \)-points to \( \beta \)-points.  This allows us to extend it linearly
  to a LOE \( \psi : X_{1} \to X_{2} \) by setting
  \[ \psi(x + r) = \psi(x) + r \]
  where \( r \in [0,\alpha) \) if \( x \in \mathcal{C}^{1}_{\alpha} \) and \( r \in [0, \beta) \)
  whenever \( x \in \mathcal{C}^{1}_{\beta} \).  The map \( \psi \) when restricted onto any orbit
  is a piecewise translation with countably many pieces and is therefore obviously Lebesgue measure
  preserving.
\end{proof}

The construction of LOE presented in \cite{slutsky_lebesgue_2015} is based on
a similar idea on a sparse piece, but uses a different construction on a compressible part,
resulting in a LOE  which is a piecewise translation with countably many pieces, but
sizes of pieces may be arbitrarily small on some orbits.  The argument presented above has the
advantage of constructing LOE with all translation pieces being of size \( \alpha \) or
\( \beta \).

\bibliographystyle{alpha}
\bibliography{refs}

\end{document}